\documentclass[aap]{imsart}

\RequirePackage{amsthm,amsmath,amsfonts,amssymb}
\RequirePackage[numbers]{natbib}

\usepackage{iftex}
\ifPDFTeX
  \usepackage[utf8]{inputenc}
  \usepackage[T1]{fontenc}
\fi

\usepackage[british]{babel}

\usepackage{etoolbox}
\usepackage{amsmath}
\usepackage{amssymb}
\usepackage{amsthm}
\usepackage{mathtools}

\usepackage{enumitem}

\usepackage{nicefrac}

\usepackage{hyperref}

\hypersetup{hidelinks}

\startlocaldefs
\theoremstyle{plain}
\newtheorem{theorem}{Theorem}
\newtheorem{proposition}{Proposition}
\newtheorem{corollary}{Corollary}
\newtheorem{lemma}{Lemma}

\theoremstyle{remark}

\newtheorem{assumption}{Assumption}
\newtheorem{remark}{Remark}

\setlist[enumerate,1]{label={(\roman*)}}
\setlist[enumerate,2]{label={(\alph*)}}
\setlist[enumerate,3]{label={(\Roman*)}}

\def\PP{\mathbb{P}}
\def\RR{\mathbb{R}}
\def\EE{\mathbb{E}}

\newcommand\fa{\mathfrak{a}}
\newcommand\fb{\mathfrak{b}}
\newcommand\fc{\mathfrak{c}}

\newcommand{\dd}{\mathrm{d}}

\def\E{\mathbb{E}}
\def\P{\mathbb{P}}
\def\R{\mathbb{R}}
\def\Q{\mathbb{Q}}
\def\1{\mathbf{1}}
\def\N{\mathbb{N}}
\def\d{\partial}
\def\Z{\mathbb{Z}}
\def\cD{{\cal D}}
\def\cF{{\cal F}}
\def\cL{{\cal L}}
\def\cA{{\cal A}}

\DeclareMathOperator{\supp}{supp}

\endlocaldefs

\begin{document}

\begin{frontmatter}
\title{A quasi-stationary approach to the long-term asymptotics of the growth-fragmentation equation}
\runtitle{A quasi-stationary approach to the growth-fragmentation equation}

\begin{aug}
\author[A]{\fnms{Denis}~\snm{Villemonais}\ead[label=e1]{denis.villemonais@unistra.fr}}
\and
\author[B]{\fnms{Alexander R.}~\snm{Watson}\ead[label=e2]{alexander.watson@ucl.ac.uk}}
\address[A]{Université de Strasbourg, IRMA, Strasbourg, France\printead[presep={,\ }]{e1}}
\address[B]{Department of Statistical Science, University College London, London, UK\printead[presep={,\ }]{e2}}
\end{aug}

\begin{abstract}
  In a growth-fragmentation system, cells grow in size slowly and
  split apart at random. Typically, the number of cells in the
  system grows exponentially and the distribution of the sizes
  of cells settles into an equilibrium
  `asymptotic profile'. In this work
  we introduce a new method to prove this asymptotic
  behaviour for the growth-fragmentation equation,
  and show that the convergence to the asymptotic profile
  occurs at exponential rate.
  We do this by identifying an associated sub-Markov
  process and studying its quasi-stationary behaviour
  via a Lyapunov function condition.
  By doing so, we are able to simplify and generalise
  results in a number of common cases and offer a unified framework
  for their study.
  In the course of this work we are also able to prove
  the existence and uniqueness of solutions to the growth-fragmentation
  equation in a wide range of situations.
\end{abstract}



\end{frontmatter}
\tableofcontents


\section{Introduction}

Growth-fragmentation describes
a system of objects which grow slowly and deterministically,
and split apart suddenly at random. It arises in biophysical models
of cell division \cite[\S 4]{Per07},
cellular aggregates \cite{Banasiak2004}
and protein polymerisation \cite{PruessPujo-MenjouetEtAl2006}.
We are concerned in this work with a mathematical model of
a growth-fragmentation system
which describes its average behaviour over time. We will
give general conditions for such a model to make sense, and
characterise its long-term behaviour,
by showing that cell numbers grow exponentially
and the cell size distribution settles into an equilibrium,
and that this occurs at exponential rate.

In a growth-fragmentation system, each cell has a trait associated with it, called its size. 
As time
progresses, the size of the cell increases in a deterministic way,
mathematically modelled by an ordinary differential equation.
At some random time, it undergoes fragmentation, and splits its size,
again at random, into a collection of descendant cells. 

A common starting point for the study of these phenomena
is the equation
\begin{equation}\label{e:gfe-1}
  \partial_t u_t(x) + \partial_x \bigl( c(x) u_t(x) \bigr)
  = \int_x^\infty u_t(y) k(y,x) \, \dd y
  - K(x) u_t(x),
\end{equation}
where $u_t(x)$ represents the density of cells of size $x$ at time
$t$, $c$ and $K$ are growth and fragmentation rates respectively,
and $k$ represents the repartition of size between parent and descendant cells.

This equation can be expressed in a more general form, without requiring
densities, by considering a semigroup $T$ which solves the
following equation:
\begin{align}
\label{eq:intro}
  \partial_t T_t f(x)
  = T_t \cA f(x),
  \qquad
  \cA f(x) 
  = \frac{\d f}{\d s}(x)+\int_{(0,x)}f(y)k(x,\dd y)-K(x)f(x),
\end{align}
for suitable functions $f$.
Here, $s$ represents the growth term,
$K(x)$ is again the rate at which a cell of size $x$ experiences
fragmentation, and $k(x,\dd y)$ is the rate at which
a cell of size $y$ appears as the result of the fragmentation
of a cell of size $x$.
We call \eqref{eq:intro} the \emph{growth-fragmentation equation}.

Speaking formally, if $u_t$ solves \eqref{e:gfe-1}, then
$T_tf(u_0) \coloneqq \int u_t(x) f(x) \, \dd x$
solves a version of \eqref{eq:intro} integrated
against $u_0(x)\, \dd x$, with $s(x) = \int_1^x \frac{\dd y}{c(y)}$.  On the
other hand, if $T$ solves \eqref{eq:intro} and $T_tf(x) = \int u_t(y) f(y) \dd y$,
and once again
$s(x) = \int_1^x \frac{\dd y}{c(y)}$,
then $u_t$ satisfies \eqref{e:gfe-1} with $u_0 = \delta_x$, the
Dirac delta measure. However, it is not straightforward to make this connection
rigorous (see \cite{DebiecDoumicEtAl2018} for one possible approach for
growth-fragmentation, and \cite[Theorem~8.1]{CHHK-msnte} for a related model).
We take \eqref{eq:intro} as our main object of study.

Our standing assumptions on the coefficients of \eqref{eq:intro} will
be given at the beginning of section~\ref{sec:ExistenceUniqueness}.
For the moment, we note that $s$ should be continuous and strictly
increasing, we define (see e.g.~\cite{PousoRodriguez2015} and references therein)
\[
  \frac{\d f}{\d s}(x)=
  \lim_{\delta\to 0, \ \delta >0}\frac{f(x+\delta )-f(x)}{s(x+\delta )-s(x)},
\]
and $C^{(s)}$ to be the set of  continuous functions $f\colon (0,+\infty)\to(0,+\infty)$
such that $\d f/\d s$ is well-defined on $(0,+\infty)$.
We also write $C^{(s)}_c$ for functions $f \in C^{(s)}$ with compact support and
with $\d f/\d s$ bounded; and $C^{(s)}_{\text{loc}}$ the set of functions
$f\in C^{(s)}$ with $\d f/\d s$ locally bounded.

The purpose of this work is to give general conditions for the existence
and uniqueness of the semigroup solving the growth-fragmentation
\eqref{eq:intro}, and to describe its long-term behaviour precisely.

For the first of our results, we require the following assumption on the
existence of a Lyapunov type function for $\cA$.

\begin{assumption}
  \label{assumption1}
  There exists a positive function $h\in C^{(s)}_{\text{loc}}$ such that,
  for all $M>0$,
  \begin{align}
    \label{as:1c_1}
    \sup_{x\in(0,M)} \int_{(0,x)} \frac{h(y)}{h(x)}\,k(x,\mathrm dy)<+\infty
  \end{align}
  and such that $(0,\infty) \ni x \mapsto \frac{\mathcal A h(x)}{h(x)}$
  is bounded from above and locally bounded.
\end{assumption}

This assumption is quite abstract, but we will show shortly that
it is verified for a wide class of coefficients, covering
many commonly studied cases in the literature.
The main result of this paper is the following statement on existence and uniqueness
of solutions to the growth-fragmentation equation. We consider semigroups
acting on the Banach space
\[
  B = \{ f \colon (0,\infty)\to\RR : f \text{ is Borel and } f/h \text{ is bounded}\,\}
\]
with associated norm $\lVert f\rVert_B = \lVert f/h\rVert_{L^{\infty}((0,\infty))}$.
We caution that our definition of semigroup (which we defer to page~\pageref{para:semigroup})
does not impose any strong continuity requirement.

\begin{theorem}
  \label{thm:semigroup}
  Assume that Assumption~\ref{assumption1} holds true.
  There exists a unique semigroup $(T_t)_{t\geq 0}$ on $B$
  such that, for all $f \in \cD(\mathcal{A}) \coloneqq C_c^{(s)}\cup\{h\}$,
  \begin{align}
    \label{eq:gfint}
    \int_0^t T_u \lvert \cA f\rvert (x) \,\dd u < \infty
    \text{ and }
    T_t f(x) = f(x) + \int_0^t T_u\cA f(x)\, \dd u.
  \end{align}
\end{theorem}

We
study the semigroup $T$ by connecting it to that of a Markov
process via an $h$-transform,
and this is a feature shared by other recent work
such as
\cite{BW-spec, Cavalli2019, BPR12, CGY-spectral}.
However, whereas these previous works have been concerned
with finding either a specific superhamonic function (in the first two cases)
or an eigenfunction (in the latter two) connected to $\cA$, we are quite free 
in our choice of the function $h$, provided that we verify
Assumption~\ref{assumption1}.
In turn, we make use of the theory of sub-Markov
processes and their quasi-stationary distributions.
This gives us a great deal
of freedom and accounts for the flexibility of our approach. 
In particular, we do not require conservation of size at splitting events
(i.e., $K(x)=\int \frac{y}{x} k(x,\dd y)$),
and both $K(x)\leq \int \frac{y}{x} k(x,\dd y)$
and $K(x)\geq \int \frac{y}{x} k(x,\dd y)$ are possible in our framework,
modelling respectively size creation and destruction;
see section~\ref{sec:entrance} for a representative example.

Other approaches, which do not adapt well to our situation, have been proposed.
An approach via Hille-Yosida theory may be found
in~\cite{BanasiakLambEtAl2019,BernardGabriel2017,BernardGabriel2020},
and further references therein;
a method using strongly continuous semigroups in $L^1$ spaces is contained
in~\cite{Cavalli2019,MischlerScher2016,MaillardPaquette2020};
\cite{Mokhtar-Kharroubi2020}~discusses perturbation results for
$C_0$-semigroups in well chosen function spaces;
an approach from martingale theory can be found in~\cite{BW-spec};
and~\cite{BCGM-Harris} uses a fixed point argument.

The second part of our work consists of describing the long-term
behaviour of $T$, the unique solution of the growth-fragmentation equation.
In order to do this, we leverage a representation of $T$ in terms of a sub-Markov
process which is developed in the proof of Theorem~\ref{thm:semigroup},
and make use of the theory of quasi-stationary limits in weighted
total variation distance.

The following additional assumptions are required. These are discussed in
more detail in section~\ref{sec:long-time}.
The first can be regarded as a sufficient condition for irreducibility
of the auxiliary Markov process referred to above; see Proposition~\ref{prop:Markov}
in section~\ref{sec:long-time} for a proof of this.

\begin{assumption}
  \label{assumptionIrr}
  For all $x\in(0,+\infty)$, the Lebesgue measure of $s(\{y\in(x,+\infty): k(y,(0,x))>0\})$ is positive.
\end{assumption}

The second assumption implies a certain Doeblin condition for said Markov process,
as shown in Proposition~\ref{prop:doeblin} in section~\ref{sec:long-time}.

\begin{assumption}
  \label{assumptionDoebSec}
  One of the following holds true:
  \begin{enumerate}[label={(\alph*)}, ref={\ref{assumptionDoebSec}(\alph*)}]
    \item \label{assumptionDoebBoth}
      There exist a positive constant $a>0$, a non-empty, open, compactly
      contained interval~$I\subset (0,+\infty)$, a function 
      $T\colon [0,1]\times (0,+\infty) \to (0,+\infty)$
      and a probability measure $\mu$ on $[0,1]$ such that, for all
      positive measurable function $f$ on $(0,+\infty)$,
      \[
        \int_{(0,+\infty)} f(y) k(x,\mathrm dy)\geq a
        \int_{[0,1]} f({T(\theta,x)}) \,\mu(\mathrm d\theta),
        \qquad x\in I,
      \]
      and such that for all $\theta\in[0,1]$, $s\circ
      T(\theta,\cdot)$ is continuously differentiable with respect to $s$ on $I$,
      with
      \begin{align}
        \label{eq:conditionderivativethetaTer}
        \frac{\partial s\circ T(\theta,\cdot)}{\partial s}(x)\neq 1,
        \qquad x\in I.
      \end{align}
    \item 
      \label{assumptionDoebF}
      There exists a non-negative non-zero kernel 
      $\beta$ from $(0,+\infty)$ to $(0,+\infty)$  such that, for all $x\in(0,+\infty)$,
      $k(x,\mathrm dy)\geq \beta(x,\mathrm dy)$
      and such that, for all measurable $A\subset (0,+\infty)$ and $x\in(0,+\infty)$ with $\beta(x,A)>0$,
      \begin{align}
          \label{eq:semi-cont-pos-assumption}
          \liminf_{y\to x,\,y<x} \beta(y,A)>0.
      \end{align}
  \end{enumerate}
\end{assumption}

With these in place, we can state the second main result.
A more general version of this appears, with proof, as Theorem~\ref{thm:spectgap}
and Proposition~\ref{prop:lambda0upperbound} in section~\ref{sec:long-time}.

\begin{theorem}
  \label{thm:spectgap-intro}
  Assume that Assumptions~\ref{assumption1}, \ref{assumptionIrr} and~\ref{assumptionDoebSec}
  hold, 
  and that there exist 
  positive functions $\psi,\xi \in C^{(s)}_\text{loc}$,
  constants $\lambda_1 \geq \lambda_2$ and $c_1,c_2,C>0$, and a 
  compact interval $L\subset (0,+\infty)$,
  such that
  $c_1\xi \le h \le c_2\psi$,
  $\lim_{x\to 0, +\infty} \frac{\xi(x)}{\psi(x)} = 0$, 
  $\sup_{x\in(0,M)} \int_{(0,x)} \frac{\xi(y)}{\xi(x)} k(x,\mathrm dy)<+\infty$ for all $M>0$, 
  and
  \begin{align*}
    \cA \psi(x) &\le -\lambda_1 \psi(x) + C\mathbf{1}_{L}(x),
    &x &\in (0,+\infty), \\
    \cA \xi(x) &\ge -\lambda_2 \xi(x),
    &x &\in (0,+\infty).
  \end{align*}
  Then, there exist
  $\lambda_0 \le \lambda_2$, a
  unique positive measure $m$ on $(0,+\infty)$ 
  and a unique function $\varphi:(0,+\infty)\to (0,+\infty)$ 
  such that $m(\psi)=1$ and $\|\varphi/\psi\|_\infty=1$ and such that, 
  for all $t\geq 0$, $mT_t=e^{\lambda_0 t}m$ and $T_t\varphi=e^{\lambda_0 t}\varphi$.
  Moreover, for all $f:(0,+\infty)\to\R$ such that $|f|\leq \psi$, we have
  \begin{align*}
    \left|e^{\lambda_0 t} T_t f(x)-\varphi(x)m(f)\right|
    \leq c e^{-\gamma t} \psi(x).
  \end{align*}
  for some constants $c,\gamma>0$.
  If moreover $\frac{\cA \xi}{\xi}$ is not constant, then $\lambda_0 < \lambda_2$.
\end{theorem}

This result is exactly what one hopes for from a Lyapunov function
approach, but the reader
may still wonder whether these conditions and assumptions can be
verified in practice.
A representative case is the following, which appears later, with proof,
as Proposition~\ref{prop:appli1}.

Consider the operator $\cA$ given in the form
\[
  \cA f = c(x) f'(x) + K(x) \left( \int_{(0,1)} f(ux) \, p(\dd u) - f(x) \right),
\]
where $p$ is a finite measure on $(0,1)$
such that $\int_{(0,1)} u\, p(\dd u) = 1$,
$K$ is right-continuous
and $c\colon (0,+\infty) \to (0,+\infty)$ is right-continuous and locally bounded.
This means that $p(\dd u)$ describes the rate
of seeing children of relative size $\dd u$ at splitting, regardless
of the size of the parent (we say that $k$ is `self-similar');
there is conservation of size
at splitting events; and that
prior to splitting, the size $x_t$ of a cell
follows the ordinary differential equation $\dot{x}_t = c(x_t)$.
To put this into the framework of \eqref{eq:intro}, we may take
$s(x) = \int_1^x \frac{\dd y}{c(y)}$ and $k(x,\cdot) = K(x) p\circ m_x^{-1}$,
where $m_x(u) = xu$.

\begin{proposition}
  \label{prop:appli1-intro}
  Assume  that $\sup_{x\in (0,M)} K(x)<+\infty$ for each $M>0$,
  that Assumptions~\ref{assumptionIrr} and~\ref{assumptionDoebSec}
  hold true, 
  that
  \begin{align*}
    \int_{(0,1)} \frac{K(x)}{c(x)}\, \dd x < +\infty
  \end{align*}
  and that there exists $\alpha>1$  such that, for all $u\in(0,1)$,
  \begin{align}
    \label{as:2appli1bis-intro}
    \liminf_{x\to+\infty}
    \int_{ux}^x \frac{K(x)}{c(x)}\, \dd x
    > \frac{-\alpha\ln u}{1-\int_{(0,1)} v^\alpha p(\mathrm dv)}.
  \end{align}
  Then,
  Assumption~\ref{assumption1} holds,
  and the conclusions of Theorem~\ref{thm:spectgap-intro} are valid,
  with $\lambda_0<0$.
\end{proposition}

In the case where $p(\dd u) = 2\dd u$, which represents splitting into
an average of two children with uniform size repartition,
Assumption~\ref{assumptionDoebBoth}
is satisfied provided $K$ has some positive lower bound on a
compact interval (see Remark~\ref{assumptionDoeb} in section~\ref{s:app}), and 
the inequality
\eqref{as:2appli1bis-intro} holds if
\begin{equation}\label{e:3sqrt8}
  \liminf_{x\to +\infty} \frac{x K(x)}{c(x)} > 3 + 2\sqrt{2}.
\end{equation}

On the contrary, when $p(\dd u) = 2\delta_{1/2}(\dd u)$, representing
equal mitosis, 
Assumption~\ref{assumptionDoebBoth}
holds provided that $K$
has some positive lower bound on a compact interval $I$
and that $c(x) \ne 2c(x/2)$ for $x\in I$ 
(see Remark~\ref{rem:equamitosi} in section~\ref{s:app}).
Moreover, the right-hand side of
\eqref{as:2appli1bis-intro} has minimum approximately
$-3.86\ln u$ (with the exact expression involving
an implicit function).
This implies that \eqref{as:2appli1bis-intro}
holds if
\begin{equation}\label{e:Kc-mitosis}
  \liminf_{x\to+\infty} \frac{xK(x)}{c(x)}
  > 3.86.
\end{equation}
Proposition~\ref{prop:appli1-intro} and inequalities
\eqref{e:3sqrt8} and \eqref{e:Kc-mitosis} give very concrete conditions for checking
the long-term behaviour in these common cases.

\paragraph*{Comparison with the literature}
We give here a brief overview of the literature and describe how our results
compare with the state of the art, without pretending to be comprehensive.

The situation in Proposition~\ref{prop:appli1-intro} was considered in Theorem
1.3 of \cite{CGY-spectral} where, as discussed earlier, the authors begin by
finding an eigenfunction $\varphi$ for $\cA$, using functional analysis
techniques, and then use Lyapunov function criteria for the convergence of the
resulting (conservative) semigroup.  Proposition~\ref{prop:appli1-intro}
improves upon this by reducing the regularity assumptions on $c$ and $K$ and
the requirements on the relative growth rates of these functions.
Proposition~\ref{prop:appli1-intro} also recovers Theorem~4.3 in
\cite{BCGM-Harris} while making slightly weaker assumptions on the growth rate
and balance between $K$ and $c$. Both of these cases are discussed in more
detail in section~\ref{sec:pseudo-entrance}.

Bertoin~\cite[section~3.6]{Ber-fk} considers the setting where $c$ is a positive,
continuous, approximately linear function and $k$ is self-similar, and develops
moment conditions for convergence at geometric rate.  The method used in that
work bears some similarity with our own: the authors identify a particular superharmonic function
and study an associated Markov process. Section~6 of \cite{BW-lln} applies
the same idea in a more specific setting. Our approach is based on the study of a sub-Markov process, which allows us to extend the geometric convergence to a broader class of models.
In particular, 
Proposition~\ref{prop:lnxmodified}, whose conditions are similar to those
in the model studied in~\cite{Ber-fk},  gives
a similar, but more general moment conditions for this to occur.
Cavalli~\cite{Cavalli2019} uses methods similar to those of~\cite{Ber-fk} in a situation in which the fragmentation
rate is bounded and cells can grow to positive size starting from size $0$.  We
reach the same conclusion in section~\ref{sec:entrance}, making weaker
regularity assumptions than \cite{Cavalli2019}. 

In \cite{BCG-fine,CCM11,GabrielSalvarani2014}, $K$ is comparable to a power law, $c$ is
either constant or linear, and $\frac{k(x,\cdot)}{xK(x)}$ is bounded both below
and above. The authors prove geometric convergence in a weighted $L^2$ norm.
In the case of self-similar $k$, Propositions~\ref{prop:appli1}
and~\ref{prop:lnx} offer the convergence in~\cite{CCM11}
and~\cite{GabrielSalvarani2014} respectively, albeit in a weighted $L^\infty$
space; the bounds on $\frac{k(x,\cdot)}{xK(x)}$ can be replaced with the
weaker Assumption~\ref{assumptionDoebSec}.  When $k$ is not
self-similar, our main result also applies under weaker requirements, as detailed in
Remark~\ref{rem:pseudo}.
We also note in this remark that our geometric convergence result covers, under weaker assumptions, the setting of D\k{e}biec et al.~\cite{DebiecDoumicEtAl2018}, who use a relative entropy method to prove convergence
without geometric rate, and of Bernard and Gabriel~\cite{BernardGabriel2020}, where the authors prove geometric convergence rates by means of quasi-compactness. 

Maillard and Paquette~\cite{MaillardPaquette2020} consider the particular case where $c(x)=x$, $K(x)=xR(x)$ and $k(x,\mathrm dy)=xR(x)\frac{2y}{x^2}\mathrm dy$ for some positive function $R$, so that there is no conservation of mass, but rather conservation of the number of fragments.
They establish a concise necessary and sufficient condition
for existence of a stationary distribution and convergence without geometric rates. Our results can be applied to this particular model to  provide sufficient conditions for convergence with a geometric rate. This is 
 discussed in Remark~\ref{rem:pseudo}, where we emphasize the additional requirements of our conditions (which ensure geometric rates of convergence) to the ones from~\cite{MaillardPaquette2020}. Bouguet~\cite{Bouguet2018} also studies the situation where the solution to the growth fragmentation equation is the semigroup of a (conservative) Markov process, under moment conditions  on the fragmentation kernel and asymptotic conditions on the growth and fragmentation rate, which also enter in our setting.

Finally, we give some additional pointers to the (wide) literature. In~\cite{BanasiakLamb2012} the authors consider a growth-fragmentation model in a discrete state space; in~\cite{BernardGabriel2017}, the authors study the non-convergence of the growth fragmentation equation; in~\cite{Cav-refr}, the author obtains a sharp bound for the coefficients of a critical growth-fragmentation equation (we actually recover this sharp bound in Section~\ref{sec:criticalKcst}); the two papers \cite{Marguet2019,BW-lln} study the branching process representation of the growth-fragmentation phenomenon and corresponding laws of large numbers; in~\cite{LaurencotPerthame2009}, the authors prove an explicit geometric rate of convergence under  a specific monotonicity condition on integrands of the kernel in the case where $s(x)=x-1$ and we do recover under weaker assumptions geometric convergence in the examples they consider, however it does not seem that their assumptions imply the Doeblin condition required to apply our result; in~\cite{Mokhtar-Kharroubi2020}, the author studies a situation in
which loss of mass occurs, either at division events or directly by cell
`death' and proves quasi-compactness properties.
We also refer the reader to the seminal works~\cite{DiekmannHeijmansEtAl1984,Heijmans1985} where dynamics in $(0,1]$ are studied; see also~\cite{BPR12,RudnickiPichor2000} for an extension of this model.

\paragraph*{Discussion of the growth term}
Besides giving general Lyapunov function criteria for solutions of the
growth-fragmentation and their long-term behaviour, the present work also makes
it possible to consider more general growth dynamics, since the growth term in
$\cA$ is given by the general differential $\d f/\d s$.  As intimated in the
previous example, the classical situation, where $\d f/\d s$ is replaced by 
$cf'$ for some continuous positive function $c$, can be recovered by setting
$s(x) = \int_1^x \frac{\dd y}{c(y)}$.  However, our setting allows us to
handle, in particular, situations where the drift $c$ vanishes and is not
Lipschitz.  Indeed, consider the case where $c(x)=\sqrt{|x-1|}$.  Then the flow
directed by the generator $f\mapsto cf'$, acting on continuously differentiable
functions, has multiple solutions, whereas the flow directed by the generator
$f\mapsto \d f/\d s$, acting on functions with bounded $s$ derivatives, admits
only one solution.  It also covers seamlessly the situation where the drift $c$
is not locally bounded.  The fact that the generator is not restricted to
continuously differentiable functions is of course a central component.

\paragraph*{Outline of the paper}
In section~\ref{sec:ExistenceUniqueness}, 
we prove that the  growth-fragmentation equation admits a unique solution, 
by representing it as an $h$-transform of the semigroup of a sub-Markov process. 
In section~\ref{sec:long-time}, we state and prove a general result which 
implies Theorem~\ref{thm:spectgap-intro}, and we provide several applications to 
different families of growth fragmentation equations, with a comparison to the
state of the art.
Finally, in Appendix~\ref{sec:appendix}, we give some extensions of Davis'
work \cite{Davis} on piecewise-deterministic Markov processes (PDMP)
which are required in section~\ref{sec:ExistenceUniqueness},
proving in particular that the martingale problem is well-posed.

\section{Existence of a unique solution to the growth-fragmentation equation}
\label{sec:ExistenceUniqueness}

This section is devoted to the proof of Theorem~\ref{thm:semigroup},
which is to say, the existence and uniqueness of a semigroup $T$
solving the growth-fragmentation \eqref{eq:intro}.
Before discussing this in detail, we should clarify our standing assumptions,
notation and definitions.

The coefficients of \eqref{eq:intro} have the following
standing assumptions in place.
Let $k$ be a positive kernel from $(0,+\infty)$ to itself 
such that $k(x,[x,+\infty))=0$ for all $x\in(0,+\infty)$,
let $s\colon (0,+\infty)\to\R$ be a strictly 
increasing 
continuous function such that $s(1)=0$ and $\lim_{x\to+\infty} s(x)=+\infty$, and 
$K:(0,+\infty)\to \mathbb R$ be a measurable locally bounded function. 

Recall the definition given earlier of the derivative of $f$
with respect to $s$, 
\[
  \frac{\d f}{\d s}(x)=
  \lim_{\delta\to 0, \ \delta>0}\frac{f(x+\delta)-f(x)}{s(x+\delta)-s(x)},
\]
and the function spaces $C^{(s)}$, $C^{(s)}_c$ and $C^{(s)}_{\text{loc}}$
of $s$-differentiable functions.
It is also useful at this point to observe that,
if a function $f$ is $s$-differentiable on the right with locally bounded derivatives in the above sense, then $f$ is 
$s$-absolutely continuous  (as defined in the appendix) and $\partial f/\partial s$ is its Radon--Nikodym derivative.
On the other hand, if $f$ is $s$-absolutely continuous, then
the right-hand side above is equal to its Radon--Nikodym derivative almost everywhere.

\phantomsection
\label{para:semigroup}%
We say that $T = (T_t)_{t\ge 0}$ is a \emph{semigroup} on a measurable space $E$ if
\begin{enumerate}
  \item for each $t\ge 0$, $T_t$ is a kernel from $E$ to itself,
  \item for each $t,u\ge 0$, $x\in E$ and measurable $A\subset E$,
    $T_{t+u}(x,A) = \int_{E} T_t(x,\dd y) T_u(y,A)$,
  \item $T_0(x,\cdot) = \delta_x$
\end{enumerate}
As is usual for kernels,
we can regard $T_t$ as acting on a measurable
function $f \colon E \to \RR_+$ by
the definition $T_t f(x) = \int_{E} T_t(x,\dd y) f(y)$,
and if $\mu$ is a measure on $E$, we can also define a measure
$\mu T_t = \int_E \mu(\dd x) T_t(x,\cdot)$.
If $B$ is some space of functions on $E$ with the property that
$T_t(B)\subset B$, we will refer to $T_t$ as a semigroup on $B$.
Crucially, we do not make the requirement that $T$ is strongly continuous. 
In addition (see Corollary~\ref{cor:consistence} below) the semigroup $T$ 
does not depend on the choice of $h$ made in Assumption~\ref{assumption1}.

The proof of our first theorem is based on the study of
an auxiliary sub-Markov process. 
More precisely, setting
\begin{equation}
  \label{e:b}
  b \coloneqq \sup_{x\in(0,+\infty)} \frac{\cA h(x)}{h(x)}
\end{equation}
which is finite by assumption, we show 
that the action $\cL f(x)=\frac{\cA (h f)(x)}{h(x)}-bf(x)$ 
on suitable $f$ uniquely characterises a sub-Markov process $X$,
whose killing rate is given by 
\[
  q(x)
  :=
  b - \frac{\cA h(x)}{h(x)}
  \ge 0,
  \quad\forall x\in(0,+\infty).
\]
The auxiliary sub-Markov process $X$ can be described informally as follows.
The growth-fragmentation equation \eqref{eq:intro} can be seen as
characterising the expected behaviour of a system of growing and dividing
cells (see \cite{BW-lln} for a precise description).  We assign each cell a
time-dependent weight, with the property that at each time, the sum of weights
of cells is less than $1$.  At time $t$, we treat the weights as a probability
distribution and select a cell based on this (with positive probability of
selecting no cell). The size of the selected cell is equal in law to $X_t$, and
this procedure can be iterated to obtain the law of a sub-Markov process $X$.
This process is characterised in the following result, which is proved in
section~\ref{sec:markov}; however, for our purposes, we do not need any
system of cells, and work solely with semigroups and their associated operators.

\begin{proposition}
  \label{tNirr}
  Assume that Assumption~\ref{assumption1} holds true.
  Let $E = (0,\infty)\cup\{\partial\}$, where $\partial$ is an isolated point.
  Consider the operator $\cL$ given by $\cL f(\d) = 0$ and
  \begin{align*}
    \cL f(x)
    &= \frac{\cA(h f)(x)}{h (x)}-bf(x)+q(x)f(\d), \\
    &= \frac{\partial f}{\partial s}(x) + \int_{(0,x)} \bigl( f(y)-f(x) \bigr) \frac{h(y)}{h(x)} k(x,\dd y)
    + q(x)\bigl( f(\d) - f(x)\bigr),
    \quad x\in(0,+\infty).
  \end{align*}
  with domain
  \[
    \cD(\cL) = \big\{f\colon E\to\RR : f(\d)\in\R\text{ and }\,f\rvert_{(0,\infty)} \in C_c^{(s)}\big\}.
  \]
  There exists a unique c\`adl\`ag solution to the martingale
  problem $(\cL,\cD(\cL))$
  for any initial measure $\delta_x$.
  Moreover, its semigroup $Q$ satisfies:  
  for all $t\geq 0$, all $x\in E$ and all $f\in\cD(\cL)$,  
  \begin{equation}\label{eq:Q-gen-X}
    \int_0^t Q_u |\cL f|(x)\,\dd u<+\infty\quad\text{ and }\quad Q_t f(x)=f(x)+\int_0^t Q_u \cL f(x)\,\dd u.
  \end{equation}
\end{proposition}

Then we show that there is at most one Markov semigroup $Q$ 
on $L^\infty(E)$ and satisfying~\eqref{eq:Q-gen-X}. 
A semigroup $Q$ on $E$ is called \emph{Markov} if $Q_t \1_E = \1_E$ for all $t\ge 0$.
The following result is proved in section~\ref{sec:uniqueness}.

\begin{proposition}
  \label{prop:uniqueQ}
  Assume that Assumption~\ref{assumption1} holds true. Then there is at most one Markov semigroup $Q$ on $L^\infty(E)$ satisfying:  
  for all $t\geq 0$, all $x\in E$ and all $f\in\cD(\cL)$,  
  \begin{equation}
    \label{eq:semigroupQ}
    \int_0^t Q_u |\cL f|(x)\,\dd u<+\infty
    \text{ and }
    Q_t f(x)=f(x)+\int_0^t Q_u \cL f(x)\,\dd u.
  \end{equation}
\end{proposition}

With these results on the auxiliary semigroup $Q$ in place, we can conclude
the proof of Theorem~\ref{thm:semigroup} in section~\ref{sec:endProofTh1},
by representing $Q$ as an $h$-transform of the semigroup $T$ solving
the growth-fragmentation equation~\ref{eq:gfint}.

Representing $T$ in terms of $Q$ is very useful, and not only for existence
and uniqueness: our results on the spectral gap in section~\ref{sec:long-time}
also rely on this technique.

We conclude with a few useful properties of $T$.  First, we note a simple bound
on the support of $\delta_x T$, and observe that~\eqref{eq:gfint} holds true on
an extension of the domain of $\mathcal A$.  This result is proved in
section~\ref{sec:proofcor1}.

\begin{corollary}
  \label{cor:1}
  Assume that Assumption~\ref{assumption1} holds true. 
  Then, for all $x\in(0,+\infty)$ and all $t\geq 0$,
  the support of $\delta_x T_t$ is included in $[0,s^{-1}(s(x)+t)]$. 
  Let $f\in C^{(s)}_{\text{loc}}$ such that $|f|/|h|$ is bounded and such that $\inf \frac{\cA f}{h}>-\infty$ or 
  $\sup \frac{\cA f}{h}<+\infty$. Then  
  equality~\eqref{eq:gfint} holds true.
\end{corollary}

Secondly, we observe that the solution to~\eqref{eq:gfint} does not depend on
the choice of~$h$.  This is proved in section~\ref{sec:cor2}, ensuring that 

\begin{corollary}
  \label{cor:consistence}
  Let $h_1$ and $h_2$ satisfy Assumption~\ref{assumption1}. 
  Then, the solution $T^1$ to~\eqref{eq:gfint} with $h_1$ instead of $h$,
  and the solution $T^2$ to~\eqref{eq:gfint} with $h_2$ instead of $h$,
  are identical.
\end{corollary}

\begin{remark}
  In this paper, we assume that sizes take values in $(0,+\infty)$.  However,
  when $0$ is an entrance boundary for the growth component, that is when
  $s(0+)>-\infty$, it is straightforward to adapt the method and results of
  this paper to the case where the space $(0,+\infty)$ is replaced by
  $[0,+\infty)$, with $k(x,\{0\})\geq 0$ for all $x\in [0,+\infty)$. 
\end{remark}

\begin{remark}\label{rem:Q}
  The primary difficulty in proving the results from this section is that we
  wish to characterise $T$ (and $Q$) in terms of analytic, rather than
  probabilistic conditions; that is, the semigroup condition and the
  equations \eqref{eq:gfint} and \eqref{eq:semigroupQ}. Moreover, we want
  to avoid assumptions such as the Feller property, which may not hold
  in the absence of similar conditions on the kernel $k$.

  Indeed, Proposition~\ref{tNirr}, concerning the probabilistic question of the
  well-posedness of the martingale problem in the space of càdlàg paths,
  is fairly straightforward to prove,
  and we rely primarily on the work of Davis~\cite{Davis} (slightly extended in
  Appendix~\ref{sec:appendix}) and standard techniques of Ethier and Kurtz~\cite{EK-mp}.

  Proposition~\ref{prop:uniqueQ} is substantially more involved, but the techinque
  is essentially to show that any semigroup solving \eqref{eq:semigroupQ}
  can be represented in terms of a càdlàg Markov process, which in turn is
  the unique solution of the martingale problem already addressed.
  The ideas in this part, such as the study of upcrossings and regularisation of
  paths, are familiar, but typical references (for instance, \cite{RW1}
  under Feller-type (\S III.7) or Ray (\S III.36) conditions) have sufficiently
  different hypotheses that we were not able to reduce the situation to
  one covered by these results.
\end{remark}

\subsection{An auxiliary Markov process}
\label{sec:markov}

This section is devoted to the proof of Proposition~\ref{tNirr}.

From now on, we set
\[
  k_h(x,\mathrm dy)=\frac{h(y)}{h(x)}k(x,\mathrm dy).
\]
so that, by Assumption~\ref{assumption1}, $x\mapsto k_h(x,(0,x))$ is bounded on $(0,M)$, for all $M>0$.
Before proving Proposition~\ref{tNirr}, we start with a useful technical lemma.
We define $f_-(x) = \max\{ -f(x),0\}$.

\begin{lemma}
  \label{lem:2bis}
  Assume $f \in \cD(\cL)$,
  meaning that $f\vert_{(0,+\infty)} \in C_c^{(s)}$,
  and that Assumption~\ref{assumption1} holds true. Then  
  \begin{enumerate}
    \item $\cL f$ is locally bounded;
    \item 
      \label{i:2bis:bdd-below}
      if $f$ is non-negative, then $\cL f$ is  bounded below;
    \item if $f$ is non-negative and $f(\d)=0$, then, for all $M> 0$, $\sup_{x\in (0,M)} \cL f(x)<+\infty$.
  \end{enumerate}
\end{lemma}
\begin{proof}
  Since $f\in \cD(\cL)$, $f\rvert_{(0,\infty)} \in C_c^{(s)}$.
  Define $F = \supp f\rvert_{(0,\infty)}$, a compact subset of $(0,\infty)$.
  We first note the following: for all $x\in(0,+\infty)$,
  \[
    \left|\cL f(x)\right|
    \leq  \left\|\frac{\d f}{\d s}\right\|_\infty + 2\|f\|_\infty k_h(x,(0,x)) +2\|f\|_\infty q(x),
  \]
  where $q(x)= b-\frac{\mathcal A h(x)}{h(x)}\geq 0$ and $k_h(x,(0,x))$ are locally bounded by Assumption~\ref{assumption1}. This proves the first point.

  If $f$ is non-negative, then
  \begin{align}
    \cL f(x)
    &\geq -\left\|\frac{\d f}{\d s}\right\|_\infty-f(x)k_h(x,(0,x))-q(x)f(x) \notag \\
    & \geq -\left\|\frac{\d f}{\d s}\right\|_\infty-\1_{x\in F} \|f\|_\infty (k_h(x,(0,x))+q(x))\label{eq:lemma2point2}
  \end{align}
  which is bounded below since $F$ is compact and $q(x)$ and $k_h(x,(0,x))$ are locally bounded. This proves the second point of Lemma~\ref{lem:2bis}.

  If $f$ is non-negative and $f(\d)=0$, then 
  \begin{align*}
    \cL f(x)\leq \left\|\frac{\d f}{\d s}\right\|_\infty+\int_{(0,x)} f(y)\,k_h(x,\mathrm dy)\leq \left\|\frac{\d f}{\d s}\right\|_\infty+\|f\|_\infty k_h(x,(0,x))
  \end{align*}
  which is bounded over $x\in(0,M)$, for all $M>0$, according to Assumption~\ref{assumption1}.
\end{proof}

We can now proceed to the proof of Proposition~\ref{tNirr}.

\begin{proof}[Proof of Proposition~\ref{tNirr}]
  We first show that there exists a càdlàg solution of the $(\cL,\cD(\cL))$ martingale problem, and then prove that this solution is unique.

  \textbf{(1) There exists a càdlàg solution of the $(\cL,\cD(\cL))$ martingale problem.}

  Since $s$ is continuous and strictly increasing, there exists a unique flow $\phi:(0,+\infty)\times [0,+\infty)\to (0,+\infty)$ such that $\phi(x,0)=x$ and
  \begin{align}
    \label{eq:flow}
    \frac{\mathrm d}{\mathrm d t} s(\phi(x,t))=1,\ \forall x,t,
  \end{align}
  which is given by $\phi(x,t)=s^{-1}(s(x)+t)$ for all $x\in(0,+\infty)$ and $t\geq 0$. We also set $\phi(\d,t)=\d$ for all $t\geq 0$.
  We observe that $\phi$ is not explosive since it satisfies $s(\phi(x,t))= s(x)+t$ for all $t\geq 0$ and $x\in(0,+\infty)$, while $s(y)\to+\infty$ when $y\to+\infty$. Moreover, for all $f\in\cD(\cL)$, we have
  \begin{align}
    \frac{\mathrm d_+}{\mathrm dt}f(\phi(x,t))&:=\lim_{h\to 0, h>0} \frac{f(\phi(x,t+h))-f(\phi(x,t))}{h}\nonumber\\
    &=\lim_{h\to 0, h>0} \frac{f(s^{-1}(s(x)+t+h))-f(s^{-1}(s(x)+t))}{h}\nonumber\\
    &=\lim_{y\to \phi(x,t), y>\phi(x,t)} \frac{f(y)-f(\phi(x,t))}{s(y)-s(\phi(x,t))}=\frac{\d f}{\d s}(\phi(x,t)).\label{eq:starrevis}
  \end{align}

  Let us consider the 
  piecewise-deterministic Markov process
  (PDMP) $X$ directed by the flow $\phi$ between its jumps and with jump kernel $k_h$ and killing rate $q$, constructed jump after
  jump, similarly as in~\cite{Davis}, with values on
  $(0,+\infty)\cup\{\infty,\d\}$ and up to the time of explosion of
  the number of jumps. Here $\infty$ is the point to which the process
  is sent after explosion of the number of jumps and $\d$ is the
  cemetery point.

  We prove now that the process $X$ is non-explosive, so that it defines a c\`adl\`ag Markov process on $E$. For all $k\geq 2$, we set
  $\tau_k^+=\inf\{t\geq 0,\,X_t\geq k\text{ or }X_{t-}\geq k\}$ and
  $\tau^-_{\nicefrac1k}=\inf\{t\geq 0,\,X_t\leq \nicefrac1k\text{ or
  }X_{t-}\leq \nicefrac1k\}$.
  As pointed out above, we know that the flow $\phi$ does not explode. Since the process only admits negative jumps,
  $X_t\leq \phi(X_0,t)$ almost surely, so that, for
  all $x\in (0,+\infty)$ and all $t\geq 0$, there exists $k_{x,t} \geq 2$ such that
  \begin{align}
    \label{uBound}
    \P_x(\tau^+_{k_{x,t}}\leq t)=0,
  \end{align}
  where $\P_x$ denotes the law of $X$ with initial distribution $\delta_x$ (as usual, we extend this notation to initial distribution $\mu$ by $\P_\mu$ and denote $\E_x$ and $\E_\mu$ the associated expectations).

  According to~\eqref{as:1c_1}, the jump rate of $X$ from $(0,+\infty)$ to $(0,+\infty)$, that is $y\mapsto k_h(y,(0,y))$, is uniformly bounded on $(0,k_{x,t}]$. Since in addition $\d$ is an absorbing point, 
  the process does not undergo an infinity of negative jumps before
  time $t\wedge \tau^+_{k_{x,t}}$, $\mathbb P_x$-almost surely for all $x\in E$. Using the fact that the flow $\phi$ is increasing,
  we deduce that the process does not converge to $0$ before time
  $t\wedge \tau^+_{k_{x,t}}$ $\mathbb P_x$-almost surely for all $x\in E$, that is
  \begin{align}
    \label{lBound}
    \lim_{k\rightarrow+\infty} \P_x(\tau^-_{\nicefrac1k}\leq t\wedge \tau^+_{k_{x,t}} )=0,\ \forall x\in E.
  \end{align}
  Combining both~\eqref{uBound} and~\eqref{lBound}, we deduce that,
  for all initial distribution $\nu$ on $(0,+\infty)\cup\d$,
  \begin{align}
    \label{non-explosion}
    \lim_{k\rightarrow+\infty} \P_\nu(\tau^-_{\nicefrac1k}\wedge \tau^+_k\leq t )=0.
  \end{align}
  This concludes the proof that $X$ defines a non-explosive c\`adl\`ag Markov process on $E$.

  Let us now remark that it satisfies the $(\cL,\cD(\cL))$-local martingale problem. 
  Indeed, for all $f\in \cD(\cL)$, $f$ belongs to the domain of the 
  extended generator of $X$, as proved in Theorem~26.14 in~\cite{Davis}, 
  with the only difference being that, in our case, the flow $\phi$ is 
  not determined by a locally Lipschitz continuous vector field $\chi$, 
  but instead by $s$. The only adaptation to be made in the proof of Theorem~26.14 in~\cite{Davis} 
  to obtain that $f$ is an element of the domain of the extended generator of 
  $X$ is as follows. Denote by
  $J_{i-1}$ and $J_i$ the $i-1^{\text{th}}$ and $i^{\text{th}}$ jump times of $X$;
  then,
  $f(X_{J_i-})-f(X_{J_{i-1}}) = 0$ when $X_{J_{i-1}}=\partial$, 
  and otherwise,
  \begin{align*}
    f(X_{J_i-})-f(X_{J_{i-1}})
    &=
    \int_{0}^{J_i-J_{i-1}}\,\frac{\mathrm d_+}{\mathrm dt}
    f(\phi(X_{J_{i-1}},t))\mathrm d t \\
    &= \int_{0}^{J_i-J_{i-1}}\,\frac{\d f}{\d s}(\phi(X_{J_{i-1}},t))\,\mathrm dt
    = \int_{J_{i-1}}^{J_i} \frac{\d f}{\d s}(X_t)\,\mathrm dt.
  \end{align*}
  This replaces the expression $\int_{J_{i-1}}^{J_i} \mathcal X(X_t)\,\mathrm dt$ in~\cite{Davis}. The rest of the proof is identical.

  Let us now prove that $X$ satisfies the  $(\cL,\cD(\cL))$-martingale problem.
  We have that, for all $x\in E$ and under $\P_x$,
  $M^f_t:=f(X_t)-f(x)-\int_0^t \cL f(X_u)\,\dd u$ is a c\`adl\`ag local martingale. 
  Moreover, since $f$ and $\cL f$ are locally bounded by Lemma~\ref{lem:2bis} point~(i), 
  the sequence $\tau_k=\tau^-_{\nicefrac1k}\wedge \tau^+_k$ is a localization sequence.

  We initially focus on the case where
  $f \in \cD(\cL)$ is non-negative, and set $a=\inf_{E} \cL f$, which is finite
  by Lemma~\ref{lem:2bis} point~\ref{i:2bis:bdd-below}.
  We have, for any fixed $t>0$ and any $k\geq 2$,
  \begin{align*}
    |M^f_{t\wedge \tau_k}|
    \leq 2\|f\|_\infty+\int_0^t \left|\cL f(X_u)\right|\,\dd u
    &\leq  2\|f\|_\infty+|a| t+\int_0^t \left\lvert \cL f(X_u)-a \right\rvert\,\dd u
  \end{align*}
  where, by the monotone convergence theorem and the local martingale property for $M^f$,
  \begin{align*}
    \E_x\left(\int_0^t |\cL f(X_u)-a|\,\dd u\right)&=\E_x\left(\liminf_{k\to+\infty} \int_0^{t\wedge \tau_k} |\cL f(X_u)-a|\,\dd u \right)\\
    &= \liminf_{k\to+\infty}\ \E_x\left(\int_0^{t\wedge \tau_k} |\cL f(X_u)-a|\,\dd u \right)\\
    &= \liminf_{k\to+\infty}\ \E_x\left(\int_0^{t\wedge \tau_k} (\cL f(X_u)-a)\,\dd u \right)\\
    &= \liminf_{k\to+\infty}\ \E_x\left(f(X_{t\wedge \tau_k})-f(x)-M^f_{t\wedge\tau_k} \right)+|a| t.\\
    &\leq 2\|f\|_\infty+|a| t.
  \end{align*}
  Hence, for all $T\geq 0$, $\{|M^f_{t\wedge \tau_k}| : t \le T, k\ge 2\}$ 
  is dominated by an integrable random variable.
  We conclude by \cite[Theorem 51]{Pro-sc} that, for all $x \in E$, under
  $\P_x$, $M^f$ is a martingale.

  Next, we remove the assumption that $f$ is non-negative, and permit any $f\in\cD(\cL)$.
  Let $\varphi\in \cD(\cL)$ such that $\varphi\geq f_+$,
  where $f_+(x) = \max\{f(x),0\}$.
  Then, according to the above result, 
  $M^\varphi$ is a martingale.
  Setting $\psi=\varphi-f$, we have $\psi\geq 0$ and 
  $\psi\in \cD(\cL)$ and hence 
  $M^\psi$ is also a martingale.
  Since $M^f = M^\varphi - M^\psi$, we deduce that $M^f$ is a martingale.

  Finally, we conclude that $X$ defines a non-explosive 
  c\`adl\`ag Markov process on $E$, which satisfies the 
  $(\cL,\cD(\cL))$-martingale problem. 
  In particular, we have shown that 
  $\int_0^t \E_x(|\cL f|(X_u))\,\dd u<+\infty $,
  and we observe that the semigroup of $X$ satisfies~\eqref{eq:Q-gen-X}.

  \medskip

  \textbf{(2) $X$ is the unique càdlàg solution of the $(\cL,\cD(\cL))$ martingale problem.}
  For all $n\geq 2$, we consider the operator 
  $\cL_n$ on $\cD(\cL)$ defined, 
  for all $x \in E$ and $g\in\cD(\cL)$, by 
  \[
    \cL_n g(x) 
    = \1_{x\in (0,+\infty)}
    \left[\frac{\d g}{\d s_n}(x) + \int_{(0,x)\cup \d} [g(y)-g(x)] Q_n(x,\dd y)\right],
  \]
  where $s_n$ is a continuous increasing function on $(0,+\infty)$ and $Q_n$ a kernel such that
  \[
    \begin{cases}
      &s_n(x)=s(x),\quad  x>1/n, \\
      &\lim_{x\downarrow 0}s_n(x) = -\infty,\\
      &Q_n(x,\mathrm dy)
      =\1_{x\in (\nicefrac1n,n)}[k_h(x,\mathrm dy)+q(x)\delta_\d(\mathrm dy)].
    \end{cases}
  \]

  According to Proposition~\ref{prop:martProblem2} in the appendix,
  the solution
  of the martingale problem for $(\cL_n,\cD(\cL))$ is unique.
  In particular, any two solutions of the 
  $D_{E}[0,+\infty)$ martingale problem for
  $\cL_n$ have the same distribution on 
  $D_{E}[0,+\infty)$ (see Corollary~4.4.3 in~\cite{EK-mp}).
  This and Theorem~4.6.1 in~\cite{EK-mp} imply that, for each 
  $n\geq 2$ and all probability measures $\nu$ on $E$, 
  the stopped martingale problem for $(\cL_n,\nu,(\nicefrac1n,n)\cup\{\d\})$ 
  admits a unique solution with sample paths in 
  $D_{E}[0,+\infty)$. Since, for all $g\in\cD(\cL)$,  
  we have $\cL_n g(x)=\cL g(x)$ for all $x\in(\nicefrac1n,n)\cup\{\d\}$, 
  we deduce that the stopped martingale problem for 
  $(\cL,\nu,(\nicefrac1n,n)\cup\{\d\})$ also admits a unique solution 
  with sample paths in $D_{E}[0,+\infty)$. 
  Since $X$ stopped at time 
  $\tau_n\coloneqq\inf\{t\geq 0, X_t\text{ or }X_{t-}\notin (\nicefrac1n,n)\cup\{\d\}\}$ 
  is a c\`adl\`ag solution to this stopped martingale problem, 
  it gives its unique solution in $D_{E}[0,+\infty)$. 
  Since it satisfies in addition
  \[
    \lim_{n\rightarrow+\infty}\P_\nu(\tau_n\leq t)=0,
  \]
  we deduce from Theorem~4.6.3 in~\cite{EK-mp}
  that there is a unique solution to the 
  $D_{E}[0,+\infty)$-martingale problem associated to $\cL$ on $\cD(\cL)$.
\end{proof}

\subsection{Uniqueness of a Markov semigroup generated by \texorpdfstring{$\cL$}{L}}

\label{sec:uniqueness}

This section is devoted to the proof of Proposition~\ref{prop:uniqueQ}, 
that is to the uniqueness of a Markov semigroup $Q$ satisfying \eqref{eq:semigroupQ}. 

In order to do so, we first prove useful technical lemmas. Then, we show that, given such a Markov semigroup $Q$, one can construct a c\`adl\`ag solution to the $(\cL,\cD(\cL))$-martingale problem with semigroup $Q$. The uniqueness of the solution to~\eqref{eq:semigroupQ} then derives from the uniqueness of this martingale problem, proved in Proposition~\ref{tNirr}.

\subsubsection{Technical lemmas}

Let $Q$ be a Markov semigroup solution to~\eqref{eq:semigroupQ}.  The following
lemmas will be useful to prove the non-explosion of a process with semigroup
$Q$. 

Throughout this section, we will apply the operator $\cL$ to functions (such as
$W$ and $f^A_m$ in the following Lemma) which do not lie in $\cD(\cL)$. We note
that, under Assumption~\ref{assumption1},
$\cL f$ is well-defined by its integro-differential action for any
$f \in C^{(s)}$ with $f$ and $\partial f/\partial s$ locally bounded,
whereas $\cD(\cL)$ represents the space on which the equation \eqref{eq:semigroupQ},
defining $Q$, is valid.

\begin{lemma}
  \label{lem:QtW}
  Assume that Assumption~\ref{assumption1} holds true. 
  Let $W\colon(0,+\infty)\cup\{\partial\}\to[0,+\infty)$ be such that $W\rvert_{(0,+\infty)}$ is a non-decreasing 
  function in $C^{(s)}_{\text{loc}}$ 
  with $W(\d) = 0$ and
  such that $\sup_{z>0} \frac{\d W}{\d s}(z)<+\infty$.
  Then, $\int_0^t Q_u \lvert \cL W(x)\rvert \, \dd u < \infty$ and 
  \[
    Q_t W(x)
    \leq W(x) + \int_0^t Q_u \cL W(x)\,\mathrm du,
    \quad t\geq 0\text{ and }x\in(0,+\infty).
  \]
\end{lemma}

\begin{proof}[Proof of Lemma~\ref{lem:QtW}]
  For any $A\geq 1$,
  we consider the $C^{(s)}_{\text{loc}}$ function $W_A:(0,+\infty)\to[0,+\infty)$ defined by
  \[
    W_A(x)=\begin{cases}
      W(x)&\text{ if }x\leq A+1,\\
      W(A+1)&\text{ if }x\geq A+1,
    \end{cases}
  \]  
  and also set $W_A(\d)=0$.
  For any $m\geq 3$, we consider a $C^{(s)}_{\text{loc}}$ function $f_m^A:(0,+\infty)\to[0,+\infty)$ such that
  \begin{align*}
    f_m^A(x)=\begin{cases}
      W(A+1)&\text{ if }x\leq 1/m,\\
      W_A(x)&\text{ if }x\geq 2/m,
    \end{cases}
  \end{align*}
  such that $f_m^A$ which is non-increasing on $(1/m,2/m)$. In particular, $\frac{\d f_m^A}{\d s}(x)\leq \frac{\d W_A}{\d s}(x)\leq \frac{\d W}{\d s}(x)$ for all $x\in(0,+\infty)$. We also set $f_m^A(\d)=0$.

  Since $g_m^A:=W(A+1)\1_E-f_m^A\in \cD(\cL)$
  and $\cL \1_E = 0$,
  we deduce from~\eqref{eq:semigroupQ} that, for all $t\geq 0$ and all $x\in(0,+\infty)$,
  \begin{align*}
    Q_t g_m^A(x)=g_m^A(x)+\int_0^t Q_u \cL g_m^A(x)\,\mathrm du=g_m^A(x)-\int_0^t Q_u \cL f_m^A(x)\mathrm du.
  \end{align*}
  But $Q_t\1_E=\1_E$, and hence, subtracting
  $W(A+1)\1_E(x)$ on both sides of the equation, we deduce that
  \begin{align}
    \label{eq:QtfmA}
    Q_t f_m^A(x)=f_m^A(x)+\int_0^t Q_u \cL f_m^A(x)\,\mathrm du.
  \end{align}
  Set $h_m^A(z)=\int_{(0,z)} [f_m^A(y)-f_m^A(z)]\,k_h(z,\mathrm dy)$ for all $z>0$. We observe that  $h_m^A(z)\leq 0$ for all $z\geq A+1$ and that $h_m^A(z)\leq W(A+1)\sup_{y\in(0,A+1)} k_h(y,(0,y])$ for all $z\leq A+1$, which is finite according to Assumption~\ref{assumption1}.   Using the fact that, for all $z>0$,
  \[
   \frac{\d f_m^A}{\d s}(z)+q(x)(f_m^A(\partial)-f_m^A(z))
   = \frac{\d f_m^A}{\d s}(z)-q(x)f_m^A(z)\leq C_W := \sup_{y>0} \frac{\d W}{\d s}(y),
  \]
  we deduce that
  \begin{align*}
    \sup_{m\geq 1}\sup_{z\in(0,+\infty)} \cL f_m^A(z)<+\infty.
  \end{align*}
  Hence, applying Fatou's Lemma in the integral part of~\eqref{eq:QtfmA}, we deduce that
  \begin{align}
    \label{eq:Fatou1}
    \limsup_{m\to+\infty} Q_t f_m^A(x)&\leq \limsup_{m\to+\infty} f_m^A(x)+  \int_0^t Q_u (\limsup_{m\to+\infty} \cL f_m^A) (x)\,\mathrm du.
  \end{align}
  We have $\limsup_{m\to+\infty} f_m^A(x)=W_A(x)$ and the left hand side is equal to $Q_t W_A(x)$ by dominated convergence (recall that $f_A\leq W(A+1)$). Moreover, for any fixed $z>0$, we deduce from Fatou's Lemma (recall that, when 
  $m\to+\infty$, $f_m^A(y)-f_m^A(z)$ is uniformly bounded from above in $y$ and converges pointwise to $W_A(y)-W_A(z)$, while $k_h(z,\mathrm dy)$ has finite mass) that
  \begin{align*}
    \limsup_{m\to+\infty} \int_{(0,z)} [f_m^A(y)-f_m^A(z)]\,k_h(z,\mathrm dy)&\leq \int_{(0,z)} [W_A(y)-W_A(z)]\,k_h(z,\mathrm dy),
  \end{align*}
  while $\partial f_m^A/\partial s(z)$ converges pointwisely toward $\partial W^A/\partial s(z)$, so that 
  \[
    \limsup_{m\to+\infty} \cL f_m^A(z)\leq \cL W_A(z).
  \]
  This and~\eqref{eq:Fatou1} thus entail that, for all $A\geq 2$,
  \begin{align*}
    Q_t W_A(x)\leq W_A(x)+\int_0^t Q_u \cL W_A(x)\,\mathrm du.
  \end{align*}
  Since $\cL W_A\leq C_W$, we can use again Fatou's Lemma, and deduce
  \begin{align*}
    \limsup_{A\to+\infty} Q_t W_A(x) 
    \leq  \limsup_{A\to+\infty} W_A(x) 
    +  \int_0^t Q_u (\limsup_{A\to+\infty} \cL W_A) (x)\,\mathrm du.
  \end{align*}
  On the one hand, $\limsup_{A\to+\infty} W_A(x)=W(x)$ and, by monotone convergence, we obtain that $\limsup_{A\to+\infty} Q_t W_A(x)=Q_tW(x)$. On the other hand, using the monotone convergence theorem (note that $W_A(y)$ is increasing in $A$, for any fixed $y$), we deduce that, for all $z>0$,
  \begin{align*}
    \limsup_{A\to+\infty} \int_{(0,z)} [W_A(y)-W_A(z)]\,k_h(z,\mathrm dy) = \int_{(0,z)} [W(y)-W(z)]\,k_h(z,\mathrm dy)
  \end{align*}
  and hence that $\limsup_{A\to+\infty} \cL W_A(z)= \cL W(z)$.
  This implies that
  \begin{align*}
    Q_t W(x)\leq W(x)+ \int_0^t Q_u \cL W(x)\,\mathrm du.
  \end{align*}

  Since $\cL W$ is bounded from above by $C_W$, this implies
  that $\int_0^t Q_u(\cL W)_-(x)\mathrm du<+\infty$ and hence that $\int_0^t
  Q_u|\cL W|(x)\mathrm du<+\infty$. This concludes the proof of
  Lemma~\ref{lem:QtW}.
\end{proof}

\begin{lemma}
  \label{lem:defofnu}
  We define the function $p\colon E\to[0,+\infty]$ by 
  $p(z)=k_h(z,(0,1))$ and $p(\d)=0$.
  If Assumption~\ref{assumption1} holds true, then $p(E)\subset [0,+\infty)$ and, for all $x\in(0,+\infty)$, 
  \[
    \int_0^t Q_u p(x)\,\mathrm du <+\infty.
  \]
\end{lemma}

\begin{proof}[Proof of Lemma~\ref{lem:defofnu}]
  We first observe that $p(z)<+\infty$ for all $z\in(0,+\infty)$ 
  according to \eqref{as:1c_1}. 
  Let $W$ be a $C^{(s)}_{\text{loc}}$ non-decreasing function 
  with $W(\d)=0$,
  such that 
  $C_W:=\sup_{z>0} \frac{\d W}{\d s}(z)<+\infty$ and
  \begin{align*}
    W(x)=\begin{cases}
      0&\text{ if }x\leq 1,\\
      1&\text{ if }x\geq 2.
    \end{cases}
  \end{align*}
  For all $z>0$, we have
  \begin{align*}
    \cL W(z)
    &\leq C_W + \int_{(0,z)} [W(y) - W(z)]\,k_h(z,dy) \\
    &\leq  C_W+
    \begin{cases}
      0&\text{ if }z\leq 2\\
      -\int_{(0,z)} \1_{y\leq 1} \,k_h(z,\mathrm dy)&\text{ if }z\geq 2.
    \end{cases}
  \end{align*}
  Hence 
  \begin{align*}
    (\cL W)_-(z)&\geq \1_{z\geq 2}\int_{(0,z)}\1_{y\leq 1}k_h(z,\mathrm dy)-C_W\\
    &= p(z)\1_{z\geq 2}-C_W\geq p(z)-\sup_{r\in(0,2)} k_h(r,(0,r))-C_W,
  \end{align*}
  where $\sup_{r\in(0,2)} k_h(r,(0,r))<+\infty$ by Assumption~\ref{assumption1}. Hence
  \[
    \int_0^t Q_u p(x)\,\mathrm du 
    \leq \int_0^t Q_u (\cL W)_-(x)\,\mathrm du+ t\, \left(\sup_{r\in(0,2)} k_h(r,(0,r))+C_W\right).
  \]
  According to Lemma~\ref{lem:QtW}, we have $\int_0^t Q_u(\cL W)_-(x)\,\mathrm du<+\infty$. Hence we obtain
  \begin{align*}
    \int_0^t Q_u p(x)\,\mathrm du < +\infty.
  \end{align*}
\end{proof}

\begin{lemma}
  \label{lem:QtWbis}Assume that Assumption~\ref{assumption1} holds true.
  Let $W\colon E\to[0,+\infty)$ be a $C^{(s)}_{\text{loc}}$
  non-increasing function 
  such that $W(x)=0$ for all $x\geq 1$ and $W(\d)=0$. 
  Assume that $p_W(x)<+\infty$ and $\int_0^t Q_u p_W(x)\,\mathrm du<+\infty$ for all $t>0$ and $x\in(0,+\infty)$, 
  where $p_W(x)=\int_{(0,x)} W(y) k_h(x,\mathrm dy)$. 
  Then $\int_0^t Q_u |\cL| W(x)\,\mathrm du<+\infty$ and
  \[
    Q_t W(x)
    \leq W(x)+\int_0^t Q_u \cL W(x)\,\mathrm du,\ \forall x\in(0,+\infty)\text{ and }t\geq 0.
  \]
\end{lemma}

\begin{proof}
  For all $A\geq 2$, let $W_A:(0,+\infty)\to[0,+\infty)$ be the non-increasing 
  $C^{(s)}_{\text{loc}}$
  function defined as
  \begin{align*}
    W_A(x)=\begin{cases}
      W(1/A)&\text{ if }x\leq 1/A,\\
      W(x)&\text{ if }x\geq 1/A.
    \end{cases}
  \end{align*}
  We also set $W_A(\d)=0$.
  For all $m\geq 2$, let $m'>0$ be such that $s(m')=s(m)+W(1/A)$
  and let $f_m^A\colon (0,+\infty) \to[0,+\infty)$ be a $C^{(s)}_{\text{loc}}$ function such that 
  \begin{align*}
    f_m^A(x)=\begin{cases}
      W_A(x)&\text{ if }x\leq m,\\
      W(1/A)&\text{ if }x\geq m',
    \end{cases}
  \end{align*}
  such that $f_m^A$ is non-decreasing on $(1,+\infty)$ and such that $\frac{\d f_m^A}{\d s}(x)\leq 1$ for all $x\in(0,+\infty)$. We set $f_m^A(\d)=0$. Proceeding as in the proof of Lemma~\ref{lem:QtW}, we have
  \begin{align}
    \label{eq:QtfmAbis}
    Q_t f_m^A(x)=f_m^A(x)+\int_0^t Q_u \cL f_m^A(x)\,\mathrm du,\qquad\forall t\geq 0\text{ and }x\in(0,+\infty).
  \end{align}
  Set $h_m^A(z)=\int_{(0,z)} [f_m^A(y)-f_m^A(z)]\,k_h(z,\mathrm dy)$ for all $z>0$. We have, for all $0<y\leq z$, 
  \[
    f_m^A(y)-f_m^A(z) \leq W(1/A)\mathbf 1_{y<1}
  \]
  and hence
  \begin{align*}
    h_m^A(z)
    &\leq W(1/A)\int_{(0,z)}\1_{y<1}k_h(z,\mathrm dy)
    \leq W(1/A)p(z),
  \end{align*}
  where $p$ is defined in the previous lemma.
  Since $\frac{\d f_m^A}{\d s}(z)\leq 1$ for all $z>0$, we deduce that $\cL f_m^A(z)\leq 1+W(1/A)p(z)$. Since $\int_0^t Q_u (1+W(1/A)p)(x)\,\mathrm du <+\infty$ according to Lemma~\ref{lem:defofnu}, we deduce using Fatou's Lemma in~\eqref{eq:QtfmAbis}, that
  \begin{align*}
    \limsup_{m\to+\infty} Q_t f_m^A(x)&\leq \limsup_{m\to+\infty} f_m^A(x)+  \int_0^t Q_u (\limsup_{m\to+\infty} \cL f_m^A) (x)\,\mathrm du.
  \end{align*}
  As in the proof of Lemma~\ref{lem:QtW}, this entails that
  \[
    Q_t W_A(x)\leq W_A(x)+\int_0^t Q_u \cL W_A(x)\,\mathrm du.
  \]
  Now we observe that, for all $z>0$, for all $A\geq 2$,
  \[
    \cL W_A(z)\leq \int_{(0,z)} W(y)\,k_h(z,\mathrm dy) = p_W(z).
  \]
  Since $p_W$ is integrable by assumption, we can apply again Fatou's Lemma to deduce that
  \[
    \limsup_{A\to+\infty} Q_t W_A(x)\leq \limsup_{A\to+\infty}W_A(x)+\int_0^t Q_u (\limsup_{A\to+\infty}\cL W_A)(x)\,\mathrm du.
  \]
  As in the proof of Lemma~\ref{lem:QtW}, this entails that
  \[
    Q_t W(x)\leq W(x)+\int_0^t Q_u \cL W(x)\,\mathrm du.
  \]
  In addition, $\int_0^t Q_u(\cL W)_+(x)\,\mathrm du\leq \int_0^t Q_u p_W(x)\,\mathrm du<+\infty$, and hence $\int_0^t Q_u (\cL W)_-(x)\,\mathrm du<+\infty$, which concludes the proof.
\end{proof}

\subsubsection{Representation of the semigroup \texorpdfstring{$Q$}{Q} by a càdlàg Markov process}

In this section, $Q$ is a Markov semigroup satisfying~\eqref{eq:semigroupQ}.
In Lemma~\ref{lem:continuity}, we  prove the continuity and the non-explosion
of any process $(Z_t)_{t\in F}$ with semigroup $Q$, where $F\subset[0,+\infty)$
contains $\Q_+=[0,+\infty)\cap \Q$ and is countable.

\begin{lemma}
  \label{lem:continuity}
  Assume that Assumption~\ref{assumption1} holds true.  Let $F\supset \Q_+$ be
  a countable subset of $[0,+\infty)$ and let $(Z_t)_{t\in F}$ be a Markov
  process on $E$ with semigroup $Q$, defined on the probability space
  $\Omega=E^{F}$.  Then, almost surely (for any starting distribution), the
  process $(Z_t)_{t\in F}$ is continuous at any time $t\in F$ and, for all
  $T>0$, $\sup_{t\in F\cap [0,T]} \1_{Z_t\neq \partial}/Z_t<+\infty$ and
  $\sup_{t\in F\cap[0,T]} \1_{Z_t\neq \partial}Z_t<+\infty$.
\end{lemma}

\begin{proof} 
  First note that the existence of $(Z_t)_{t\in F}$ is guaranteed
  by the Kolmogorov extension theorem.  In order to simplify the expressions,
  we consider the case $F=\Q_+$.  We denote by $\P^Z_x$ (resp.\ $\P^Z_\mu$) the
  law of $Z$ with initial measure $\delta_x$ (resp.\ $\mu$), with the
  associated expectations $\E^Z_x$ and $\E^Z_\mu$. We first prove that $Z$ is
  right-continuous almost surely, then that it is left-continuous almost
  surely, and conclude by proving that, on any finite time horizon, the
  trajectories of the process are almost surely bounded away from $0$ and
  $+\infty$.

  \textbf{(1) The process $(Z_t)_{t\in\Q_+}$ is right-continuous almost surely.}
  Let $x\in(0,+\infty)$ and $f:E\to[0,+\infty)$ such that $f\rvert_{(0,+\infty)} \in C_c^{(s)}$
  with $f(\d)=0$ and such that $f$ is maximal at $x$. 
  Fix $\delta>0$ a positive rational number. 
  For all $n\geq 1$, let $M^{(n)}_0=0$ and, for all $k\geq 0$,
  \[
    M^{(n)}_{k+1}-M^{(n)}_k
    = f(Z_{\delta (k+1)/n})
    -f(Z_{\delta k/n})
    -\int_{0}^{\delta /n} Q_u \cL f(Z_{\delta k/n})\,\dd u.
  \]
  The process $M^{(n)}$ is a discrete time martingale and, using Doob's inequality, we deduce that, for all $\varepsilon>0$, 
  \begin{align}
    \label{eq:maj1}
    \P^Z_x\left(\sup_{k\in\{0,\ldots,n\}} |M^{(n)}_k|>\varepsilon\right)\leq  \frac{\E^Z_x\big(|M^{(n)}_n|\big)}{\varepsilon}.
  \end{align}
  But $M^{(n)}_k=f(Z_{\delta k/n})-f(x)-\sum_{l=0}^{k-1}\int_0^{\delta /n} Q_u \cL f(Z_{\delta l/n})\,\dd u$, so that
  \[
    |M^{(n)}_n|
    \leq f(x)-f(Z_{\delta})+\sum_{l=0}^{n-1}\int_0^{\delta /n} Q_u |\cL f|(Z_{\delta l/n})\,\dd u,
  \]
  since the maximum of $f$ is attained at $x$.
  Taking the expectation on both sides of the inequality, we obtain
  \begin{align*}
    \E^Z_x(|M^{(n)}_n|)
    &\leq f(x)-Q_\delta f(x)+ \sum_{l=0}^{n-1}\int_0^{\delta /n} Q_{u+\delta l/n} |\cL f|(x)\,\dd u\\
    &\leq Q_0f(x)-Q_\delta f(x)+ \int_0^{\delta} Q_{u} |\cL f|(x)\,\dd u.
  \end{align*}
  We also obtain that
  \[
    |M^{(n)}_k|\geq |f(Z_{\delta k/n})-f(x)|-\sum_{l=0}^{n-1}\int_0^{\delta /n} Q_u |\cL f|(Z_{\delta l/n})\,\dd u,
  \]
  where 
  \begin{align*}
    \P\left(\sum_{l=0}^{n-1}\int_0^{\delta /n} Q_u |\cL f|(Z_{\delta l/n})\,\dd u>\varepsilon\right)&\leq \frac{1}{\varepsilon}\E\left(\sum_{l=0}^{n-1}\int_0^{\delta /n} Q_u |\cL f|(Z_{\delta l/n})\,\dd u\right)\\
    &\leq \frac{1}{\varepsilon}\int_0^\delta Q_{u} |\cL f|(x)\,\dd u.
  \end{align*}
  Hence~\eqref{eq:maj1} implies that
  \begin{align*}
    &\!\!\P^Z_x\left(\sup_{k\in\{0,\ldots,n\}} |f(Z_{\delta k/n})-f(x)|>2 \varepsilon\right) \\
    &\leq \P^Z_x\left(\sup_{k\in\{0,\ldots,n\}} |M^n_k|>\varepsilon\right)+\P^Z_x\left(\sum_{l=0}^{n-1}\int_0^{\delta /n} Q_u |\cL f|(Z_{\delta l/n})\,\dd u>\varepsilon\right)\\
    &\leq \frac{Q_0f(x)-Q_\delta f(x)+ 2 \int_0^{\delta} Q_{u} |\cL f|(x)\,\dd u}{\varepsilon}.
  \end{align*}
  Setting $h_{x}(\delta)=Q_0f(x)-Q_\delta f(x)+ 2 \int_0^{\delta} Q_{u} |\cL f|(x)\,\dd u$, this implies in particular that, for all $n\geq 1$, 
  \[
    \P^Z_x\left(\sup_{k\in\{0,\ldots,n!\}} |f(Z_{\delta k/n!})-f(x)|>2\varepsilon\right)\leq \frac{h_{x}(\delta)}{\varepsilon}.
  \]
  But, almost surely,
  \[
    \sup_{k\in\{0,\ldots,n!\}} |f(Z_{\delta k/n!})-f(x)|\leq \sup_{k\in\{0,\ldots,(n+1)!\}} |f(Z_{\delta k/(n+1)!})-f(x)|
  \]
  and
  hence we can take the limit when $n\to+\infty$ in the penultimate inequality, which leads to
  \begin{align*}
    \P^Z_x\!\!\left(
      \sup_{\substack{n\geq 1,\\k\in\{0,\ldots,n!\}}} |f(Z_{\delta k/n!})-f(x)| > 2\varepsilon
    \right)
    &= \P^Z_x\!\!\left(
      \bigcup_{n\geq 1}\{\sup_{k\in\{0,\ldots,n!\}} |f(Z_{\delta k/n!})-f(x)|>2\varepsilon \}
    \right)\\
    &\leq 1\wedge \frac{h_{x}(\delta)}{\varepsilon}.
  \end{align*}
  Since $\{k/n! : n\geq 1,\,0\leq k\leq n!\}=[0,1]\cap \Q$,
  we deduce that
  \begin{align}
    \label{eq:step1-0}
    \P^Z_x\left(\sup_{q\in[0,\delta]\cap \Q} |f(Z_{q})-f(x)|>2\varepsilon\right)\leq 1\wedge \frac{h_{x}(\delta)}{\varepsilon}.
  \end{align}
  Note that $h_x(\delta)\to 0$ when $\delta\to0$,
  since $Q_t f(x)$ is continuous in $t$ 
  by \eqref{eq:semigroupQ}
  and $Q_{u} |\cL f|(x)$ is integrable over $[0,t]$. We deduce that
  \begin{equation}
    \label{eq:step1}
    \P^Z_x\left(\sup_{q\in[0,\delta]\cap \Q} |f(Z_{q})-f(x)|>2\varepsilon\right) \xrightarrow[\delta\to 0]{} 0,
  \end{equation}
  Since this is true for all functions $f\in C_c^{(s)}$ such that $f$ is maximal at $x$, this implies that $(Z_t)_{t\in\Q}$ is (right)-continuous at time $t=0$, $\P_x$-almost surely. In particular
  \[
    \P^Z_x\left(\sup_{q\in[0,\delta]\cap \Q} |Z_{q}-x|>\varepsilon\right)\xrightarrow[\delta\to 0]{} 0,\ \forall x\in (0,+\infty).
  \]
  For $x=\d$, we have, for all $t\geq 0$, $Q_t \1_\d(x)=Q_0\1_\d(x)=1$, so that $Z_t=\d$ $\P_\d$-almost surely, which of course implies the right-continuity of $(Z)_{t\in\Q_+}$ $\P_\d$-almost surely. Hence the last convergence also holds true under $\P_\d$ (taking for instance $|y-\d|=+\infty$ for all $y\in(0,+\infty)$).

  Now, for any probability measure $\mu$ on $E$, integrating with respect to $\mu(\dd x)$ the last convergence and using the dominated convergence theorem, we deduce that
  \[
    \P^Z_\mu\left(\sup_{q\in[0,\delta]\cap \Q} |Z_{q}-Z_0|>\varepsilon\right)\xrightarrow[\delta\to 0]{} 0,
  \]
  which implies that $Z$ is continuous at time $0$, $\P_\mu$-almost surely.

  Finally, fixing $t\in \Q_+$ and using the Markov property at time $t$, we deduce that the process is right continuous at time $t\in\Q_+$ almost surely. This implies that $Z$ is right-continuous at any time $t\in\Q_+$, $\P^Z_x$-almost surely for all $x\in E$.

  \medskip
  \textbf{(2) The process $(Z_t)_{t\in\Q_+}$ is left-continuous almost surely.}

  Fix $\varepsilon>0$.
  For each $x\in(2\varepsilon,1/\varepsilon)$, let
  $f_{x,\varepsilon}\in C_c^{(s)}$
  be a function with support in $(\varepsilon/2,1/\varepsilon+2\varepsilon)$
  such that $f_{x,\varepsilon}(y)\leq \1_{|y-x|<\varepsilon}$
  and $0\leq f_{x,\varepsilon}(y)\leq f_{x,\varepsilon}(x)=1$
  for $y\in (0,+\infty)\cup\{\d\}$. 
  The collection of functions $f_{x,\varepsilon}$ can be chosen such that
  $f_{x,\varepsilon}$ and $\partial f_{x,\varepsilon}/\partial s$ are bounded
  uniformly in $x$.
  Define
  $h_{x,\varepsilon}(\delta)
  = Q_0f_{x,\varepsilon}(x)-Q_\delta f_{x,\varepsilon}(x)+ 2\int_0^{\delta} Q_{u} |\cL f_{x,\varepsilon}|(x)\,\dd u$.
  By applying~\eqref{eq:step1-0}, we obtain
  \begin{equation}
    \P^Z_x\left(\sup_{q\in[0,\delta]\cap \Q} |Z_{q}-x|>\varepsilon\right)\leq  
    \P^Z_x\left(\sup_{q\in[0,\delta]\cap \Q} |f_{x,\varepsilon}(Z_{q})-f_{x,\varepsilon}(x)|\geq 1\right)\leq 2\sup_{x\in (2\varepsilon,1/\varepsilon)} h_{x,\varepsilon}(\delta),\label{eq:lasteq}
  \end{equation}
  for all $x\in(2\varepsilon,1/\varepsilon)\cup\{\d\}$.
  (The case $x=\d$ is immediate, since we observed in step~1 that $\d$ is absorbing.)

  Using the fact that $f_{x,\varepsilon}$ is maximal at $x$, we deduce that $Q_0f_{x,\varepsilon}(x)-Q_\delta f_{x,\varepsilon}(x)$ is non-negative, hence
  \begin{align*}
    h_{x,\varepsilon}(\delta)
    &= Q_0f_{x,\varepsilon}(x)-Q_\delta f_{x,\varepsilon}(x)+ 2\int_0^{\delta} Q_{u} |\cL f_{x,\varepsilon}|(x)\,\dd u\\
    &\leq 2(Q_0f_{x,\varepsilon}(x)-Q_\delta f_{x,\varepsilon}(x))+ 2\int_0^{\delta} Q_{u} |\cL f_{x,\varepsilon}|(x)\,\dd u\\
    &=-2 \int_0^{\delta} Q_{u} \cL f_{x,\varepsilon}(x)\,\dd u+ 2\int_0^{\delta} Q_{u} |\cL f_{x,\varepsilon}|(x)\,\dd u=4\int_0^\delta Q_u(\cL f_{x,\varepsilon})_-(x)\,\dd u,
  \end{align*}
  where we used~\eqref{eq:semigroupQ} for the penultimate equality.
  We observe that $(\cL f_{x,\varepsilon})_-(z)$ is 
  bounded in $z\in (0,+\infty)\cup\{\d\}$ according to 
  Lemma~\ref{lem:2bis} point \ref{i:2bis:bdd-below}, uniformly in 
  $x\in (2\varepsilon,1/\varepsilon)$ 
  according to~\eqref{eq:lemma2point2} in its proof 
  (for this last claim, we simply observe that $\|f_{x,\varepsilon}\|_\infty$ 
  and $\|\d f_{x,\varepsilon}/\d s\|_\infty$ are bounded in $x$ by assumption and that the union of the supports of these functions is included in a compact subset of $(0,+\infty)$). Hence $C_\varepsilon(\delta):=2\sup_{x\in (2\varepsilon,1/\varepsilon)} h_{x,\varepsilon}(\delta)$ goes to $0$ when $\delta\to0$.

  Fix $x\in E$, a positive time $t\in \Q_+$ and $\delta\in [0,t]\cap\Q$.
  Note that for $q \in [0,\delta] \cap\Q$,
  $\lvert Z_t - Z_{t-q}\rvert \le \lvert Z_t - Z_{t-\delta} \rvert
  + \lvert Z_{t-\delta} - Z_{t-q}\rvert$.
  Taking we can conclude 
  (using the preceding remark in the first line, the Markov property
  in the second and \eqref{eq:lasteq} in the third) that,
  for $x\in (0,+\infty)$ and $\varepsilon'\in(0,\varepsilon/2]$,
  \begin{align}
    \P^Z_x\left(\sup_{q\in[0,\delta]\cap \Q} |Z_t-Z_{t-q}|>\varepsilon\right)
    &\leq \P^Z_x\left(\sup_{q\in[0,\delta]\cap \Q} |Z_{t-q}-Z_{t-\delta}|>\varepsilon'\right)\nonumber\\
    &=\E^Z_x\left(\P^Z_{Z_{t-\delta}}\left(\sup_{q\in[0,\delta]\cap \Q} |Z_{\delta -q}-Z_{0}|>\varepsilon'\right)\right)\nonumber\\
    &\leq C_{\varepsilon'}(\delta)+\P^Z_x(Z_{t-\delta}\notin (2\varepsilon',1/\varepsilon')\cup\{\d\})\label{eq:num}.
  \end{align}  
  But 
  \begin{align*}
    \P^Z_x(Z_{t-\delta}\notin (2\varepsilon',1/\varepsilon')\cup\{\d\})
    = 1-Q_{t-\delta}(\1_{(2\varepsilon',1/\varepsilon')\cup\{\d\}})(x)\leq 1-Q_{t-\delta}g_{\varepsilon'}(x),
  \end{align*}
  where $g_{\varepsilon'}$ is any non-negative function in $\cD(\cL)$ bounded by $1$, equal to $1$ on $(3\varepsilon',\nicefrac1{2\varepsilon'})\cup\{\d\}$ and vanishing outside $(2\varepsilon',1/\varepsilon')\cup\{\d\}$. 
  Now, for all $\eta>0$, there exists $\varepsilon'>0$ such that $1-Q_tg_{\varepsilon'}(x)\leq \eta/2$ (by dominated convergence theorem and the fact that $\1_E\geq g_{\varepsilon'}\to \1_E$ pointwisely, with $Q_t\1_E=\1_E$) and $\delta'>0$ such that, for all $\delta\in(0,\delta')$, $|Q_tg_{\varepsilon'}(x)-Q_{t-\delta}g_{\varepsilon'}(x)|\leq \eta/2$ (by continuity of $u\mapsto Q_ug_{\varepsilon'}$ at time $t$).
  In particular, for all  $\delta\in (0,\delta')$,
  \[
    \P^Z_x(Z_{t-\delta}\notin (2\varepsilon',1/\varepsilon')\cup\{\d\})
    \leq \eta,
  \]
  Hence, we deduce from~\eqref{eq:num} that
  \begin{align}
    \label{eq:maj3}
    \P^Z_x\left(\sup_{q\in[0,\delta]\cap \Q} |Z_t-Z_{t-q}|>\varepsilon\right)&\xrightarrow[\delta\to 0]{} 0,
  \end{align}
    so that $Z$ is $\P^Z_x$-almost surely left continuous at time $t$.

  The extension to non-Dirac initial distribution can be done as in Step~1, and this concludes the proof of the first part of Lemma~\ref{lem:continuity}.

  \medskip
  \textbf{(3) The trajectories of the process $(Z_t)_{t\in [0,T]\cap \Q_+}$ are bounded away from  $0$ and $+\infty$.}

  Fix $T>0$. We first show that, for all $x\in(0,+\infty)\cup\{\d\}$, $Z$ is $\P^Z_x$-almost surely bounded from above. 
  In order to do so, fix $x\in(0,+\infty)$ (the result is trivial for $x=\d$). 
  Let $W_1$  be a $C^{(s)}_{\text{loc}}$ non-decreasing function such that $C_1:=\sup_{z>0}\d W_1/\d s(z)<+\infty$ and $\lim_{m\to+\infty} W_1(m)=+\infty$ (such a function exists since $\lim_{z\to+\infty} s(z)=+\infty$ by assumption) and set $W_1(\partial)=0$.
  According to Lemma~\ref{lem:QtW} and using the fact that  $\cL W_1\leq C_1$, we obtain that, for all $n\geq 1$, 
  \[
    M^{(n)}_k
    = W_1(Z_{Tk/n})-C_1 \,Tk/n
  \]
  defines a super-martingale.  Hence, for any $m>0$, defining the stopping time $\sigma^n_m=\inf\{lT/n,\ l\in\Z_+,\ Z_{lT/n}> m\}$ and using the optional sampling theorem, we deduce that
  \begin{align*}
    \E^Z_x(W_1(Z_{\sigma^n_m\wedge T}))&\leq  W_1(x)+ C_1 T.
  \end{align*}
  Since $W_1(Z_{\sigma^n_m\wedge T})\geq W_1(m)$ on the event $\sigma^n_m\leq T$, we deduce that
  \begin{align*}
    \P^Z_x(\sigma^n_m\leq T)\leq \frac{W_1(x)+C_1T}{W_1(m)}.
  \end{align*}
  Since $(\sigma^{n!}_m)_n$ is almost surely non-increasing and converges toward $\sigma_m=\inf\{u\in \Q_+,\ Z_{u}>m\}$, we deduce that
  \begin{align*}
    \P^Z_x(\sigma_m\leq T)\leq \frac{W_1(x)+C_1 T}{W_1(m)}.
  \end{align*}
  Using now that $(\sigma_m)_m$ is almost surely non-decreasing, we deduce that
  \begin{align}
    \label{eq:Z-non-explo1}
    \P^Z_x\left(\sup_{u\in[0,T]\cap \Q_+}\1_{Z_u\neq \partial}Z_u=+\infty\right)=\P^Z_x\left(\lim_{m\to+\infty}\sigma_m\leq T\right)=0.
  \end{align}

  We prove now that $Z$ is almost surely bounded away from $0$, starting from $x\in(0,+\infty)$. We consider the non-negative measure $\nu$ on $(0,1)$ defined by 
  \[
    \nu(A):=\int_0^t Q_u p_A(x)\,\dd u.
  \]
  where $p_A(z)=\int_0^z \1_A(y)\,k_h(z,\dd y)$ for all measurable $A$ subset of $(0,1)$.
  This is a finite measure according to Lemma~\ref{lem:defofnu}. 
  Hence there exists a non-increasing $C^{(s)}_{\text{loc}}$ function 
  $W_2:(0,+\infty)\to(0,+\infty)$ such that $W_2(z)\to+\infty$ when $z\to 0$ 
  and $W_2(z)=0$ for all $z\geq 1$, and such that $\nu(W_2)<+\infty$;
  see Lemma~\ref{l:W2} below.

  According to Lemma~\ref{lem:QtWbis} and using the fact that 
  $\int_0^t Q_u \cL W_2\,\dd u \leq \nu(W_2)$ (with $W_2(\d):=0$),
  we have that, for all $n\geq 1$, 
  \[
    N^{(n)}_k
    = W_2(Z_{Tk/n})-\nu(W_2) \,Tk/n
  \]
  defines a super-martingale.
  Defining the stopping time 
  $\sigma^n_{1/m}=\inf\{lT/n: l\in\Z_+,\ Z_{lT/n}< 1/m\}$ 
  and using the same method used to obtain~\eqref{eq:Z-non-explo1}, we deduce that
  \begin{align*}
    \P^Z_x\left(\sup_{u\in[0,T]\cap \Q_+}\1_{Z_u\neq\partial} /Z_u=+\infty\right)=0.
  \end{align*}
  This and equation~\eqref{eq:Z-non-explo1} concludes the proof of Lemma~\ref{lem:continuity}.
\end{proof}

\begin{lemma}\label{l:W2}
  Let $\nu$ be a finite measure on $(0,1)$. Then, there exists a
  non-increasing $C^{(s)}_{\text{loc}}$ function $W_2$ such that $W_2(x) \to \infty$ when
  $x\to 0$, $W_2(x) = 0$ for $x>1$ and $\nu(W_2) <+\infty$.
\end{lemma}
\begin{proof}
  Let $y_n = 2^{n-1} - 1$ for $n\ge 1$.
  Let $(x_n)$ be a decreasing sequence of numbers in $(0,1)$ such that
  $\nu(0,x_n) < 3^{-n}$ for $n\ge 1$,
  which exists because $\nu((0,1))<+\infty$.
  Then
  \[ A := \sum_{n\ge 1} y_{n+1} \nu[x_{n+1},x_n) 
    \le \sum_{n\ge 1} 2^{n}3^{-n} 
    < \infty. 
  \]
  Now let $W_2$ be defined
  by
  \[
    W_2(x) =y_{n+1} + \frac{s(x)-s(x_{n+1})}{s(x_n)-s(x_{n+1})}(y_n-y_{n+1}),
    \qquad x \in [x_{n+1},x_n),
  \]
  so that $W_2(x) \in (y_n,y_{n+1}]$ when $x\in [x_{n+1},x_n)$.
  Let $W_2(x) = 0$ for $x\ge 1$.
  Then $W_2$ is a positive, non-increasing, continuous, and admits a 
  right derivative with respect to $s$ given by
  \[
    \frac{\d W_2}{\d s}(x)=\frac{y_n-y_{n+1}}{s(x_n)-s(x_{n+1})}\leq 0,
    \qquad x \in [x_{n+1},x_n),
  \]
  and, for all $x\geq 1$, by $\frac{\d W_2}{\d s}(x)=0$.
  Moreover, we have
  \[
    \int_{(0,1)} W_2(x)\,\nu(\mathrm dx)\leq \sum_{n\in\N} y_{n+1}\nu\left(\left[x_{n+1},x_n\right)\right)<+\infty,
  \]
  which proves the lemma.
\end{proof}

We state now the uniqueness of the Markov semigroup, so that the proof of the following lemma concludes the proof of Proposition~\ref{prop:uniqueQ}. In order to do so, we show that $(Z_t)_{t\in\Q_+}$ (as in the proof of the preceeding lemma) can be extended to a c\`adl\`ag process $(Y_t)_{t\in[0,+\infty)}$ with values in $E$, which appears to be solution to the $(\cL,\cD(\cL))$-martingale problem. The conclusion is then obtained from Proposition~\ref{tNirr}.
\begin{lemma}
  \label{lem:uniqueness}
  Assume that Assumption~\ref{assumption1} holds true and that
  $Q$ is a semigroup satisfying \eqref{eq:semigroupQ}.
  Then $Q_t f(x)=\E_x(f(X_t))$ for all bounded measurable functions $f$ on $E$, 
  where $X$ is the unique c\`adl\`ag solution to the martingale 
  problem~$(\cL,\cD(\cL))$. 
  Moreover, 
  $Q_t \1_{(0,\infty)}(x) = \1_{(0,\infty)}(x) - \int_0^t Q_u q(x)\, \dd u$,
  for all $x \in E$.
\end{lemma}

\begin{proof}

  Let $(Z_t)_{t\in\Q_+}$ be as in the proof of Lemma~\ref{lem:continuity}.  In a first step, we show that, for any sufficiently regular function $f$, $(f(Z_t))_{t\in\Q_+}$ admits only finitely many upcrossings over non-empty open intervals. In a second step, we use this to deduce that $Z$ can be extended to a c\`adl\`ag Markov process $(Y_t)_{t\in[0,+\infty)}$ with semigroup $Q$ and taking its values in the one point compactification of $E$. Finally, we prove that $Y$ takes its values in $E$ and that it satisfies the $(\cL,\cD(\cL))$-martingale problem.

  \medskip
  \textbf{(1) Finiteness of the number of upcrossings.} 
  Let $x\in(0,+\infty)$ and $f$ be a non-negative function in $C_c^{(s)}$, extended to $\d$ with $f(\d)=0$. Our aim is to prove that, for any $a<b\in\R$, the number of upcrossings through $(a,b)$ of
  $(f(Z_t)-f(x))_{t\in\Q_+}$ is finite $\P^Z_x$-almost surely on any finite time horizon.

  Fix $a<b\in \R$ and $\delta\in(0,\frac{b-a}{1+4 c})\cap \Q$, where $c:=\sup (\cL f)_-$ is finite according to Lemma~\ref{lem:2bis} point (ii). For all $n\geq 1$, let $M^{(n)}_0=0$ and
  \[
    M^{(n)}_{k+1}-M^{(n)}_k= f(Z_{\delta (k+1)/n})-f(Z_{\delta k/n})-\int_{0}^{\delta/n} Q_u \cL f(Z_{\delta k/n})\,\dd u.
  \]
  The process $M^{(n)}$ is a discrete time martingale. Hence, setting $N^{(n)}_0=0$ and
  \[
    N^{(n)}_{k+1}-N^{(n)}_{k}=f(Z_{\delta (k+1)/n})-f(Z_{\delta k/n})+\frac{c\delta}{n}=M^{(n)}_{k+1}-M^{(n)}_k+\int_{0}^{\delta/n} Q_u \cL f(Z_{\delta k/n})\,\dd u+\frac{c\delta}{n}
  \]
  defines a sub-martingale. In particular, using Lemma~2.5 p.57 in~\cite{EK-mp}, we have (here $U^{(n)}(a,b)$ denotes the number of upcrossings through the interval $(a,b)$ during the $n$ first steps of the sub-martingale $N^{(n)}$):
  \[
    \E^Z_x(U^{(n)}(a,b))\leq \frac{\E^Z_x((N^{(n)}_n-a)_+)}{b-a}\leq \frac{ \|f\|_\infty+c\delta+|a|}{b-a},
  \]
  since $N^{(n)}_n=f(Z_{\delta})-f(x)+c\delta$.
  In addition, the number of up-crossing through $(a,b)$ of $(f(Z_{\delta k/n})-f(x))_{k\in\{0,\ldots,n\}}$, denoted by $V^{(n)}(a,b,\delta)$ from now on, is bounded from above by the number of up-crossing through $(a+c\delta,b-c\delta)$ of $(N^{(n)}_k)_{k\in\{0,\ldots,n\}}$. Hence
  \[
    \E^Z_x( V^{(n)}(a,b,\delta))\leq \frac{ \|f\|_\infty+2c\delta+|a|}{b-a-2c\delta}.
  \]
  Since, for all $n\geq1$, $(f(Z_{k/n!})-f(x))_{k\in\{0,n!\}}$ is a sub-process of $(f(Z_{k/(n+1)!})-f(x))_{k\in\{0,(n+1)!\}}$, we have $V^{(n!)}(a,b,\delta)\leq V^{((n+1)!)}(a,b,\delta)$ almost surely and hence
  \[
    \E^Z_x\left(\sup_{n\geq 1} V^{(n)}(a,b,\delta)\right)\leq \frac{ \|f\|_\infty+2c\delta+|a|}{b-a-2c\delta}.
  \]
  But $\sup_{n\geq 1} V^{(n)}(a,b,\delta)$ is exactly the number of upcrossings through $(a,b)$ of
  $(f(Z_t)-f(x))_{t\in\Q_+\cap[0,\delta]}$
  and hence, denoting by $V(a,b,\delta)$ this number, we have
  \[
    \E^Z_x\left(V(a,b,\delta)\right)\leq \frac{ \|f\|_\infty+2c\delta+|a|}{b-a-2c\delta}.
  \]
  Hence
  \[
    \E^Z_x\left(V(a-f(x),b-f(x),\delta)\right)\leq \frac{ \|f\|_\infty+2c\delta+|a-f(x)|}{b-a-2c\delta}\leq \frac{ 2\|f\|_\infty+2c\delta+|a|}{b-a-2c\delta},
  \]
  and, since the upcrossings through $(a-f(x),b-f(x))$ by $(f(Z_t)-f(x))_{t\in\Q_+\cap[0,\delta]}$ is exactly the number $V'(a,b,\delta)$ of upcrossings of $(a,b)$ by $(f(Z_t))_{t\in\Q_+\cap[0,\delta]}$, we deduce that
  \[
    \E^Z_x\left( V'(a,b,\delta)\right)\leq \frac{ 2\|f\|_\infty+2c\delta+|a|}{b-a-2c\delta}.
  \]
  We conclude that the number of upcrossings $V'(a,b,\delta)$ is finite $\P^Z_x$-almost surely. 
  Since this is true for all initial distribution, using the Markov property at times  $\delta$, $2\delta$, ..., we obtain that, for all $T\in\Q_+$, the number of upcrossings $V'(a,b,T)$ is finite almost surely. Since this is true for all $a<b\in\R$, this in turn implies that $V(a,b,T)$ is finite $\P^Z_x$-almost surely.

  \medskip
  \textbf{(2) Construction of a c\`adl\`ag representation of $(Q_t)_{t\in[0,+\infty)}$ in $E\cup\{\Delta\}$.}
  Now, using Problem~9(a), p.~90 in \cite{EK-mp}, we deduce that, for all non-negative functions $f\in C_c^{(s)}(0,+\infty)$ extended to $\d$ with $f(\d)=0$, $\P^Z_x$-almost surely, for all $t\in[0,+\infty)$,
  \begin{equation}
    \label{eq:Step2eq1}
    \lim_{u\in\Q_+,u>t,u\to t} f(Z_u)\quad \text{ and }\quad \lim_{u\in\Q_+,u<t,u\to t} f(Z_u)
  \end{equation}
  both exist. Moreover $\d$ is an absorbing point for $Z$, so that $(\1_\d(Z_t))_{t\in\Q_+}$ is increasing, taking its values in $\{0,1\}$, and hence the above limits also exist for  $f=\1_\d$.

  As a consequence, there exists a countable family $\mathcal{H}$ of continuous functions $f$ that separates points in $E$ and such that the above limits exist (recall that $\1_\d$ is continuous since $\d$ is an isolated point). We deduce that, $\P^Z_x$-almost surely, for all $t\in[0,+\infty)$,
  \[
    \lim_{u\in\Q_+,u>t, u\to t} Z_u\quad \text{ and }\quad \lim_{u\in\Q_+,u<t,u\to t} Z_u
  \]
  also exist in $(0,+\infty)\cup\{\d,\Delta\}$, where $\Delta$ is a compactification point for $(0,+\infty)$ (and hence for $(0,+\infty)\cup\{\d\}$).
  Indeed, let $Z_{t_+}$ and $Z_{t_+}'$ be two accumulation points in 
  $(0,+\infty)\cup\{\d,\Delta\}$ of $(Z_{u})_{u\in\Q_+,u\geq t}$ at $t\in[0,+\infty)$. 
  On the one hand, if $Z_{t_+}\in(0,+\infty)\cup\{\d\}$ and 
  $Z_{t_+}'\in (0,+\infty)\cup\{\d\}$ are different, 
  then there exists a function $f\in \mathcal{H}$ 
  such that $f(Z_{t_+})\neq f(Z_{t_+}')$. 
  Since $f$ is continuous, then this contradicts~\eqref{eq:Step2eq1}. 
  On the other hand, if $Z_{t_+}\in (0,+\infty)\cup\{\d\}$ and $Z_{t_+}'=\Delta$, 
  then one chooses any function $f\in\mathcal{H}$ 
  such that $f(Z_{t_+})>0$ with compact support, 
  and observe that $f$ extended by $0$ at $\Delta$ is continuous, 
  so that $f(Z_{t_+})\neq 0=f(Z_{t_+}')$ also contradicts~\eqref{eq:Step2eq1}. 
  This implies that, almost surely, for all $t\in(0,+\infty)$, 
  the accumulation point in $(0,+\infty)\cup\{\d,\Delta\}$ of 
  $(Z_{t+u})_{t+u\in\Q_+}$ at $t\in(0,+\infty)$ is unique, 
  which implies the existence of the first limit. 
  The existence of the second limit is proved similarly.

  We deduce that $Z$ satisfies almost surely the assumptions of Lemma~2.8, p.~58 in~\cite{EK-mp} and hence we can define the c\`adl\`ag random process $(Y_t)_{t\in\R_+}$ with values in $E\cup\{\Delta\}$ as
  \[
    Y_t:=\lim_{u\in \Q,u>t,u\to t} Z_u,\quad\ \P_x^Z\text{-almost surely}.
  \]
  Since $(Z_t)_{t\in\Q_+}$ is (right)-continuous according to Lemma~\ref{lem:continuity}, we deduce that $Y_t=Z_t$ for all $t\in\Q_+$ (in particular, $Y_t\in E$ $\PP_x^Z$-almost surely, for all $t\in\Q_+$).

  Let us now show that, for all $t\ge 0$, $\delta_x Q_t$ is the law of $Y_t$ 
  under $\P^Z_x$. We have, for all $f\in\cD(\cL)$ extended to $E\cup \{\Delta\}$ by $f(\Delta)=0$,
  \begin{align*}
    \E^Z_x(f(Y_t))
    & = \E^Z_x\left(\lim_{u>t,u\in\Q,u\to t} f(Z_u)\right) \\
    & = \lim_{u>t,u\in\Q,u\to t}\E^Z_x(f(Z_u))
    = \lim_{u>t,u\in\Q,u\to t} Q_u f(x)
    = Q_t f(x),
  \end{align*}
  since $Q_u f(x)$ is continuous in $u$ for all $f\in \cD(\cL)$
  by \eqref{eq:semigroupQ}.
  Since $C_c^{(s)}\subset {\cal D}({\cal L})$ and $\1_\d\in{\cal D}({\cal L})$, we deduce that $\P^Z_x(Y_t\in A)=\delta_xQ_t\1_A$ for all measurable $A\subset (0,+\infty)\cup\{\d\}$. Since $\delta_x Q_t\1_E=1$, we conclude that $\P^Z_x(Y_t\in E)=1$ and that $\delta_x Q_t$ is the law of $Y_t$ under $\P^Z_x$, for all $t\in[0,+\infty)$.

  Let us now prove that $Y$ is a Markov process with respect to its natural
  filtration $(\cF^0_t)_{t\geq 0}$. Fix $u_0\leq t_0\in[0,+\infty)$ and
  consider the Markov process $(Z'_t)_{t\in\Q_+\cup\{u_0,t_0\}}$ with semigroup
  $(Q_t)_{t\in\Q_+\cup\{u_0,t_0\}}$. Then $(Z'_t)_{t\in\Q_+}$ under $\P^{Z'}_x$
  has the same law as $(Z_t)_{t\in\Q_+}$ under $\P^{Z}_x$. Since $Z'$ and $Y$
  are right-continuous at times $u_0,t_0$ almost-surely (according to
  Lemma~\ref{lem:continuity} for $Z'$), we deduce that
  $(Z'_{u_0},Z'_{t_0},(Z'_t)_{t\in\Q_+})$ under $\P^{Z'}_x$ and
  $(Y_{u_0},Y_{t_0},(Z_t)_{t\in\Q_+})$ under $\P^Z_x$ have the same law, for
  all $x\in E$.
  Hence, 
  for all bounded measurable functions $f:E\to\R$ and $g:E\to\R$,
  \begin{align*}
    \E^Z_x(f(Y_{u_0})g(Y_{t_0}))&=\E^{Z'}_x(f(Z'_{u_0})g(Z'_{t_0}))\\
    &= \E^{Z'}_x(f(Z'_{u_0})Q_{t_0-u_0}g(Z'_{u_0}))\\
    &= \E^Z_x(f(Y_{u_0})Q_{t_0-u_0}g(Y_{u_0})).
  \end{align*}
  The same line of arguments applies for any finite family of times $u_1\leq\ldots\leq u_k\leq u_0\leq t_0$, which implies that, for all $0\leq u\leq t$,
  \begin{align*}
    \E^Z_x(f(Y_{t})\mid \sigma(Y_v,\,v\leq u))&=Q_{t-u} f(Y_u),\ \P^Z_x\text{-almost surely}.
  \end{align*}
  We conclude that $Y$ is indeed a Markov process, with values in $E\cup\{\Delta\}$.

  \medskip

  \textbf{(3) The càdlàg representation is a solution to the martingale problem in $E$.}
  We observe that, for all $t\geq u\geq 0$ and all $f\in\cD(\cL)$, and setting $\cL f(\Delta)=0$,
  \begin{multline*}
    \E^Z_x\left(f(Y_t)-\int_0^t \cL f(Y_v)\,dv\mid \cF^0_u\right)
    \\
    =Q_{t-u} f(Y_u)-\int_0^u \cL f(Y_v)\,\dd v - \E^Z_x\left(\int_u^t \cL f(Y_v)\,\dd v\mid\cF^0_u\right),
  \end{multline*}
  where $\mathcal F$ is the natural filtration of $Y$.
  But 
  \[
    Q_{t-u} f(Y_u)=f(Y_u)+\int_0^{t-u} Q_v\,\cL f(Y_u)\,\dd v
  \]
  and
  \[
    \E^Z_x\left(\int_u^t \cL f(Y_v)\,\dd v\mid\cF^0_u\right)=\int_u^t Q_{v-u}\,\cL f(Y_u)\,\dd v
  \]
  (using the fact that $\int_u^t Q_{v-u}\,|\cL f|(Y_u)\,\dd v$ is finite, which allows the use of Fubini's theorem). 
  Hence $f(Y_t)-\int_0^t \cL f(Y_v)\,\dd v$ defines a martingale. We deduce that $Y$ is a càdlàg solution to the martingale problem associated to $\cL$ on $E\cup\{\Delta\}$.

  But, according to Lemma~\ref{lem:continuity}, $Z$ is bounded away from $0$ and $+\infty$ almost surely, so that $Y$ (whose values are in the adherence of the values taken by $Z$ almost surely) is also bounded away from $0$ and $+\infty$ almost surely. This implies that $Y$ never reaches $\Delta$ and hence that $Y$ takes its values in $E$, $\P^Z_x$-almost surely for all $x\in E$. This entails that $Y$ is a càdlàg solution to the martingale problem in~$E$.

  \medskip

  We conclude the proof of the first part of Lemma~\ref{lem:uniqueness} by observing that Proposition~\ref{tNirr} states that the càdlàg solution to the martingale problem $(\cL,\cD(\cL))$ is unique.

  \medskip

  In order to obtain the last claim of Lemma~\ref{lem:uniqueness}, observe that $\1_\d\in\cD(\cL)$ and that $Q_t\1_E=\1_E$, so that
  \begin{align*}
    \delta_x Q_t \1_{(0,+\infty)}
    &= \delta_x Q_t \1_E-\delta_x Q_t\1_\d \\
    &= \1_E(x)-\1_\d(x)-\int_0^t Q_u\cL\1_\d(x)\,\dd u=\1_{(0,+\infty)}(x)-\int_0^t Q_u q(x)\,\dd u.
  \end{align*}
  This concludes the proof of Lemma~\ref{lem:uniqueness}.
\end{proof}

\subsection{Conclusion of the proof of Theorem~\ref{thm:semigroup}}
\label{sec:endProofTh1}
For the existence, we set $T_t f(x)= e^{bt} h(x) Q_t(f/h)(x)$ for all $f\in\cD(\cA)$ with the convention $f/h(\d):=0$, where $Q$ is the semigroup of Proposition~\ref{prop:uniqueQ}.
For all $f\in\cD(\cA)$, the function $g=f/h$ is in $\cD(\cL)$ or $g=\1_{(0,+\infty)}$, and hence, for all $x\in(0,+\infty)$, if $g\in\cD(\cL)$, then
\begin{align*}
  \partial_t T_t f(x)&=\partial_t [e^{bt} h(x) Q_tg(x)]= be^{bt}h(x)Q_tg(x)+e^{bt}h(x)Q_t \cL g(x)\\
  &=be^{bt}h(x)Q_tg(x)+e^{bt}h(x)Q_t\left(\frac{\cA (hg)}{h}-bg\right)(x)\\
  &=e^{bt}h(x)Q_t\left(\frac{\cA f}{h}\right)(x)=T_t\cA f(x),
\end{align*}
understanding differentiation here in the sense of density with respect
to Lebesgue measure; if $g=\1_{(0,+\infty)}$, then the same computation holds true according to the last property of Lemma~\ref{lem:uniqueness}.
The fact that $T_t B\subset B$ is a straightforward consequence of the fact that $Q_t \1_{(0,+\infty)}\leq \1_{(0,+\infty)}$. 

Let us now check the uniqueness.
Assume that $T$ is a semigroup which solves the above equation
for $f\in \cD(\cA)$.
Then $h\in\cD(\cA)$ and hence the semigroup defined by $\delta_x R_t:=\frac{e^{-bt} \delta_x T_t(\cdot\,h)}{h(x)}$ ($x\in(0,+\infty)$) satisfies, for all $x\in (0,+\infty)$,
\begin{align*}
  R_t \1_{(0,+\infty)}(x)&=\frac{e^{-bt} T_th(x)}{h(x)}=1-b\,\int_0^t \frac{e^{-bu} T_u h(x)}{h(x)}\,\dd u+\int_0^t e^{-b u} \frac{T_u\cA h(x)}{h(x)}\,\dd u\\
  &=1+\int_0^t R_u\cL \1_{(0,+\infty)}(x)\,\dd u\leq 1.
\end{align*}
Hence $(R_t)_{t\geq 0}$ is a sub-Markov semigroup on the set of bounded measurable functions on $(0,+\infty)$. As usual, we extend $R$ as a Markov semigroup on the set of bounded measurable functions on $E=(0,+\infty)\cup\{\d\}$, by setting $R_t \1_{\d}(x)= 1-R_t\1_{(0,+\infty)}(x)$ for all $x\in (0,+\infty)$ and $R_t f(\d)=f(\d)$ for all bounded measurable functions $f$ on $E$. For all $f \in C^{(s)}_c$, $fh\in\cD(\cA)$ and hence, for all $x\in (0,+\infty)$,
\begin{align*}
  R_t f(x)&=\frac{e^{-bt} T_t(f\,h)}{h(x)}=f(x)-b\,\int_0^t \frac{e^{-b u} T_u(f\,h)}{h(x)}\,\dd u+\int_0^t e^{-bu} \frac{T_u\cA(f\,h)(x)}{h(x)}\,\dd u\\
  &=f(x)+\int_0^t R_u\cL f(x)\,\dd u,
\end{align*}
while $R_t f(\d)=f(\d)=f(\d)+\int_0^t R_u\cL f(\d)\,\dd u$.
For all $x\in (0,+\infty)$, we have
\begin{align*}
  R_t \1_{(0,+\infty)}(x)&=1-\int_0^t R_u q(x)\,\dd u
\end{align*}
and hence
\[
  R_t \1_\d(x)=\int_0^t  R_u q(x)\,\dd u=\int_0^t  R_u \cL\1_\d(x)\,\dd u,
\]
while $R_t\1_\d(\d)=\1_\d(\d)=\1_\d(\d)+\int_0^t R_u\cL \1_\d(\d)\,\dd u$.
Using Lemma~\ref{lem:uniqueness}, we deduce that $R_t=Q_t$ and hence that $T_t f(x)=e^{bt} h(x) Q_t(f/h)(x)$. This concludes the proof of Theorem~\ref{thm:semigroup}.

\subsection{Proof of Corollary~\ref{cor:1}}
\label{sec:proofcor1}

Fix $x\in(0,+\infty)$. 
Assume first that $f\geq 0$ and set $\varphi=f/h$ and 
let $(\varphi_m)_{m\geq 0}$ be a non-decreasing sequence of functions in $C_c^{(s)}$ such that $\varphi_m(x)=\varphi(x)$ for all $x\in(1/m,m)$. We also set $\varphi_m(\partial)=\varphi(\partial)=0$. Then, for all $m>k\geq 1$, since $\varphi_m\in\cD(\cL)$ and $\tau_k$ (defined in the first step of the proof of Proposition~\ref{tNirr}) is a stopping time, for all $t\geq 0$, and all $x\in (1/k,k)$, we have
\begin{align*}
  \mathbb E_x(\varphi_m(X_{t\wedge\tau_k}))
  =\varphi_m(x)
  +\mathbb E_x\!\left(\int_0^{t\wedge\tau_k} \cL \varphi_m(X_u)\,\mathrm du\!\right)
  =\varphi(x)+\mathbb E_x\!\left(\int_0^{t\wedge\tau_k} \cL \varphi_m(X_u)\,\mathrm du\!\right).
\end{align*}
But, almost surely, for all $u<\tau_k$, we have $X_u\in (1/k,k)\subset (1/m,m)$ and hence
\begin{align*}
  \cL \varphi_m(X_u)
  &=\frac{\partial \varphi_m}{\partial s}(X_u) \\
  & \quad {} + \int_{(0,x)} \varphi_m(y)\,k_h(X_u,\mathrm dy)-\varphi_m(X_u)k_h(X_u,(0,X_u))-q(X_u)\varphi_m(X_u)\\
  &=\frac{\partial \varphi}{\partial s}(X_u)+\int_{(0,x)} \varphi_m(y)\,k_h(X_u,\mathrm dy)-\varphi(X_u)k_h(X_u,(0,X_u))-q(X_u)\varphi(X_u)\\
  &\nearrow \frac{\partial \varphi}{\partial s}(X_u)+\int_{(0,x)} \varphi(y)\,k_h(X_u,\mathrm dy)-\varphi(X_u)k_h(X_u,(0,X_u))-q(X_u)\varphi(X_u)\\
  &=\mathcal L\varphi(X_u)\text{ when $m\to+\infty$.}
\end{align*}
The monotone convergence theorem (taking into account $\mathbb E_x\left(\int_0^{t\wedge\tau_k} |\cL \varphi_m(X_u)|\,\mathrm du\right)<+\infty$ for all $m\geq 1$), we deduce that
\[
  \mathbb E_x\left(\int_0^{t\wedge\tau_k}\cL \varphi_m(X_u)\,\mathrm du\right)\xrightarrow[m\to+\infty]{} \mathbb E_x\left(\int_0^{t\wedge\tau_k} \cL \varphi(X_u)\,\mathrm du\right).
\]
Since $\varphi=f/h$ is bounded, by the dominated convergence theorem, we also deduce that
\[
  \mathbb E_x(\varphi_m(X_{t\wedge\tau_k}))\xrightarrow[m\to+\infty]{} \mathbb E_x(\varphi(X_{t\wedge\tau_k}))
\]
and hence 
\[
  \mathbb E_x(\varphi(X_{t\wedge\tau_k}))=\varphi(x)+\mathbb E_x\left(\int_0^{t\wedge\tau_k} \cL \varphi(X_u)\,\mathrm du\right).
\]

Assume first that $\mathcal A f/h$ is lower bounded by $-a$, where $a>0$. Then 
\[
  \mathbb E_x(\varphi(X_{t\wedge\tau_k}))+a\mathbb E_x(t\wedge \tau_k)=\varphi(x)+\mathbb E_x\left(\int_0^{t\wedge\tau_k} (\cL \varphi(X_u)+a)\,\mathrm du\right),
\]
where $\cL \varphi(X_u)+a=\mathcal A f(X_u)/h(X_u)+a\geq 0$, so that, by dominated convergence on the left hand side, and by monotone convergence in the right-hand side, we obtain by letting $k\to+\infty$
\[
  \mathbb E_x(\varphi(X_t))+a\mathbb E_x(t)=\varphi(x)+\mathbb E_x\left(\int_0^{t} (\cL \varphi(X_u)+a)\,\mathrm du\right)
\]
and hence that
\begin{align}
  \label{eq:eqforE}
  \mathbb E_x\left(\int_0^{t} |\cL \varphi(X_u)|\,\mathrm du\right)<+\infty\quad\text{ and }\quad\mathbb E_x(\varphi(X_t))=\varphi(x)+\mathbb E_x\left(\int_0^{t} \cL \varphi(X_u)\,\mathrm du\right).
\end{align}

Assume now instead that $\mathcal A f/h$ is upper bounded by $a>0$. Then 
\[
  \mathbb E_x(\varphi(X_{t\wedge\tau_k}))-a\mathbb E_x(t\wedge \tau_k)=\varphi(x)-\mathbb E_x\left(\int_0^{t\wedge\tau_k} (-\cL \varphi(X_u)+a)\,\mathrm du\right),
\]
where $-\cL \varphi(X_u)+a=-\mathcal A f(X_u)/h(X_u)+a\geq 0$. As above, this entails that~\eqref{eq:eqforE} holds true.

In both cases, we deduce from Fubini's theorem that
\begin{align*}
  \int_0^{t} Q_u|\cL \varphi|(x)\,\mathrm du<+\infty\quad\text{ and }\quad Q_t\varphi(x)=\varphi(x)+\int_0^{t} Q_u\cL \varphi(x)\,\mathrm du.
\end{align*}
Replacing $Q$, $\mathcal L$ and $\varphi$ by their respective expressions of $T$, $\mathcal A$ and $f$,
this concludes the proof of Corollary~\ref{cor:1}.

\subsection{Proof of Corollary~\ref{cor:consistence}}
\label{sec:cor2}

We observe that Assumption~\ref{assumption1} is clearly satisfied with $h=h_1+h_2$, and hence, according to Theorem~\ref{thm:semigroup}, there exists $T$ a solution to~\eqref{eq:gfint}. In addition, $\mathcal A h_1/h$ is upper bounded by $\mathcal A h_1/h_1$ and hence is upper bounded. By Corollary~\ref{cor:1}, we deduce that
\[
  \int_0^t T_u |\cA h_1|\,\mathrm du <+\infty \quad\text{ and }T_t h_1(x)=h_1(x)+\int_0^t T_u \cA h_1\,\mathrm du.
\]
Since in addition~\eqref{eq:gfint} holds true for all $f\in C_c^{(s)}$, we deduce from the uniqueness part of Theorem~\ref{thm:semigroup} that $T=T^1$. Similarly, $T=T^2$ which concludes the proof.

\section{Long time asymptotics of the solution to the growth-fragmentation equation}
\label{sec:long-time}

In this section, we focus on the existence of leading eigenelements and a spectral gap for the semigroup $T$ solution to~\eqref{eq:gfint} acting on the Banach space $B$. Our approach will be to leverage the representation of $T$ as the $h$-transform of the semigroup $Q$ of an absorbed Markov process evolving on $E=(0,+\infty)\cup\{\d\}$, as given in section~\ref{sec:ExistenceUniqueness}. More precisely, we will make use of the results developed in~\cite{CV-qsd} for the study of quasi-stationary distributions. 

At this stage, we require Assumptions~\ref{assumptionIrr} and~\ref{assumptionDoebSec},
which appeared in the introduction. These can be interpreted, respectively, as an irreducibility and
local Doeblin condition for the càdlàg Markov process with
semigroup $Q$ defined in Proposition~\ref{tNirr}.

Under Assumptions~\ref{assumption1} and~\ref{assumptionIrr}, the semigroup $T$
from Theorem~\ref{thm:semigroup}, the semigroup $Q$ from
Proposition~\ref{prop:uniqueQ} and the Markov process $X$ from
Proposition~\ref{prop:Markov} below are well defined, and we have the following irreducibility
result, which is
proved in section~\ref{sec:proofMarkov}. 

Denote by $\P_x$ the law of $X$ with
initial distribution $\delta_x$ for $x\in(0,+\infty)$, and 
$\PP_\mu = \int \PP_x \, \mu(\dd x)$ for a distribution $\mu$ on $(0,+\infty)$.  
Let $\cF = (\cF_t)_{t\ge 0}$ be the completion of the natural filtration with
respect to sets which are null for every $\PP_\mu$ (see \cite[\S 25]{Davis}).
Moreover, define $H_y=\inf\{t\geq 0,\ X_t=y\}$, the hitting time of $y$ by $X$.
\begin{proposition}
  \label{prop:Markov}
  Assume that Assumption~\ref{assumption1} holds true. Let $X$ be the unique càdlàg solution of the martingale
  problem $(\cL,\cD(\cL))$. Then  $X$ is a strong Markov process with respect to
  $\cF$.
  If in addition
  Assumption~\ref{assumptionIrr} holds, then $X$ is irreducible in
  $(0,+\infty)$, in the sense
  that, for all $l<r\in(0,+\infty)$, there exists $t_0>0$ such that
  \[
    \inf_{x,y\in[l,r]} \P_x(H_y\leq t_0)> 0.
  \]
\end{proposition}

Some natural examples of growth-fragmentation models may not lead to a process irreducible in
$(0,+\infty)$. For instance, if a cell can grow from small sizes to some critical
$m>0$, and subsequent divisions always lead to cells of size at least $m$, then
one may expect to fruitfully study a version of the process $X$ above, restricted
either to $(m,+\infty)$ or to $[m,+\infty)$, and find a similar irreducibility result.
Practically, this can be ensured by replacing the role of $0$
with $m$ throughout our assumptions, but we will not discuss this in detail.

We turn next to Assumption~\ref{assumptionDoebSec}. This is sufficient to obtain a local
Doeblin condition for $X$.
Assumption~\ref{assumptionDoebBoth} is adapted from a general, multi-dimensional result developed in~\cite[Proposition~1]{PichorRudnicki2000}. Its application leverages a simple 
change of variable argument, as detailed in section~\ref{sec:proofprop:doeblineTer}.
Assumption~\ref{assumptionDoebF} places regularity conditions on a lower bound of the division kernel, inspired by the concept of $T$-chain, as introduced in~\cite[Chapter~6]{MT-book}. Its application use irreducibility arguments, as detailed in Section~\ref{sec:proofprop:doeblinFeller}.

Other approaches to the local Doeblin condition typically revolve around coupling;
for instance, see~\cite{CMP10} for an approach to the TCP process and~\cite{CGY-spectral} 
for results applied to the mitosis kernel  $k(x,\dd y) = 2K(x)\delta_{x/2}(\dd y)$.
The equal mitosis kernel is also considered in section~6.3.3 of~\cite{RudnickiTyran-Kaminska2017},
using a similar approach to us; see Remark~\ref{rem:equamitosi} in section~\ref{s:app} below.

\begin{proposition}
  \label{prop:doeblin} Assume that Assumptions~\ref{assumption1},~\ref{assumptionIrr} and~\ref{assumptionDoebSec} hold true.
  Then there exists a probability measure $\upsilon$ on $(0,+\infty)$ such that, for any
  compactly contained interval $L\subset(0,+\infty)$, there exists $t_L>0$ such that, for all $t\geq t_L$ and all $x\in L$,
  \begin{align}
    \label{eq:doeblininprop}
    \P_x(X_t\in \cdot)\geq c_{L,t} \upsilon(\cdot),
  \end{align}
  where $c_{L,t}>0$ only depends on $L$ and $t$ and is non-increasing in $t$.
\end{proposition}

If Assumptions~\ref{assumption1},~\ref{assumptionIrr} and the Doeblin condition~\eqref{eq:doeblininprop} hold true, we can introduce the growth coefficient of $T$, defined by
\[
  \lambda_0:=\inf\{\lambda\in\R,\ \liminf_{t\to+\infty}\,e^{\lambda t} T_t \1_L(x)=+\infty\},
\]
with arbitrary $x\in(0,+\infty)$ and non-empty, open, compactly contained interval $L\subset (0,+\infty)$.  One easily checks, using the relationship between $T$ and the semigroup of $X$, that $\lambda_0=\lambda_0^X-b$, where
\begin{align}
  \label{eq:lambda0X}
  \lambda_0^X:=\inf\{\lambda\in\R,\ \liminf_{t\to+\infty}\,e^{\lambda t} \P_x(X_t\in L)=+\infty\},
\end{align}
The fact that $\lambda_0^X$ (and hence $\lambda_0$) does not depend on $x$ nor $L$ is a well known consequence of the irreducibility property and the Doeblin condition~\eqref{eq:doeblininprop}.

Our aim is to apply Theorem~3.5 in~\cite{CV-qsd} to $X$. This requires a Foster-Lyapunov type condition, which will be obtained using the following assumption, where we recall that $C^{(s)}_{\text{loc}}$ denotes the set of functions with a locally bounded derivative with respect to $s$.
(In fact, one may consider situations where $\psi$ is only $s$-absolutely continuous, as defined in the appendix).

\begin{assumption}\label{assumptionLyapGeneral}
  There exist a positive function $\psi\in C^{(s)}_{\text{loc}}$,
  a constant $\lambda_1>\lambda_0$ and a compact interval $L\subset(0,+\infty)$ such that $\inf_{x\in (0,+\infty)} \psi/h>0$ and 
  \begin{equation*}
    \cA \psi(x) \leq  -\lambda_1\psi(x)+C\1_L(x),\quad\forall x\in (0,+\infty),
  \end{equation*}
  for some constant $C>0$.
\end{assumption}

We emphasis that, in most cases, taking $h=\psi$ is the most natural choice, in which case the requirement $\inf_{x\in (0,+\infty)} \psi/h>0$ of the last assumption is trivial.

\medskip We can now state the main result of this section.
It is proved in section~\ref{sec:proofspectgap}.

\begin{theorem}
  \label{thm:spectgap}
  Assume that Assumptions~\ref{assumption1},~\ref{assumptionIrr},~\ref{assumptionDoebSec} and~\ref{assumptionLyapGeneral} hold true.
  Then there exist a unique positive measure $m$ on $(0,+\infty)$ 
  and a unique function $\varphi:(0,+\infty)\to (0,+\infty)$ such that $m(\psi)=1$ and $\|\varphi/\psi\|_\infty<+\infty$ and such that, for all $t\geq 0$, $mT_t=e^{\lambda_0 t}m$ and $T_t\varphi=e^{\lambda_0 t}\varphi$. Moreover, for all $f:(0,+\infty)\to\R$ such that $|f|\leq \psi$, we have
  \begin{align*}
    \left|e^{\lambda_0 t} T_t f(x)-\varphi(x)m(f)\right|\leq Ce^{-\gamma t} \psi(x).
  \end{align*}
  for some constants $C,\gamma>0$.
\end{theorem}

We call $\lambda_0$ the \emph{growth coefficient} of $T$.
Theorem~\ref{thm:spectgap} entails the existence of a spectral gap for the semigroup of $(T_t)_{t\geq 0}$ acting on the Banach space $L^\infty(\psi):=\{f:(0,+\infty)\to\R,\ \|f/\psi\|_\infty<+\infty\}$, endowed with the norm $f\mapsto \|f/\psi\|_\infty$. Conversely, if the convergence of Theorem~\ref{thm:spectgap} holds true, then $(T_t)_{t\geq 0}$ also satisfies Lyapunov type conditions and Doeblin type conditions (we refer the reader to~\cite{BCGM-Harris} and~\cite{ChampagnatVillemonais2020} for such converse properties) and it is thus expected that Theorem~\ref{thm:spectgap} covers most situations where a spectral gap exists in some $L^\infty(\psi)$.
However it is clear that our result does not apply in situations with no spectral gap. 
While a similar approach may be used in this situation, the main limitation is that the theory of quasi-stationary distributions for sub-Markov semigroup without spectral gap is limited and is still as of this day an active area of research. 

In practice, checking Assumption~\ref{assumptionLyapGeneral} requires to find a upper bound on $\lambda_0$ and to find a Lyapunov function $\psi$. We first relate $\lambda_0$ to an apparently lower quantity. This result 
 is proved in section~\ref{sec:proofprop:alterlambda}.

\begin{proposition}
  \label{prop:alterlambda}
  If Assumptions~\ref{assumption1},~\ref{assumptionIrr} and~\ref{assumptionDoebSec}
  hold true, then
  \begin{align*}
    \lambda_0=\inf\{\lambda\in\mathbb R,\ \int_0^\infty e^{\lambda t} T_t \1_L(x)\,\mathrm dt=+\infty\}
  \end{align*}
  for any $x\in (0,+\infty)$ and any non-empty open compactly embedded subset $L\subset(0,+\infty)$.
\end{proposition}

Making use of a second Lyapunov-type function $\xi$, the following result provides a criterion to find upper bounds for $\lambda_0$, proved in Section~\ref{sec:lambda0upperbound} (the proof adapts easily to situations where $\xi$ is only $s$-absolutely continuous).
Theorem~\ref{thm:spectgap} together with part \ref{i:lambda0upperbound:b}
of this result provides
Theorem~\ref{thm:spectgap-intro} in the introduction.

\begin{proposition}
  \label{prop:lambda0upperbound}
  Assume that Assumptions~\ref{assumption1}, \ref{assumptionIrr} 
  and~\ref{assumptionDoebSec}
  hold true, and that:
  \begin{enumerate}
    \item \relax \label{i:lambda0upperbound:1}
      There exist a positive function $\psi\in C^{(s)}_{\text{loc}}$, a constant $\lambda_1\in\R$ and a compact interval $L\subset(0,+\infty)$ such that $\inf_{x\in (0,+\infty)} \psi/h>0$ and 
      \begin{equation}\label{e:lambda0upperbound:Lyap}
        \cA \psi(x) \leq  -\lambda_1\psi(x)+C\1_L(x),\quad\forall x\in (0,+\infty),
      \end{equation}
      for some constant $C>0$.
    \item \relax \label{i:lambda0upperbound:2}
      There exists a positive function  $\xi\in C^{(s)}_{\text{loc}}$  such that $\|\xi/h\|_\infty<+\infty$,
      $
      \frac{\xi(x)}{\psi(x)}\xrightarrow[x\to 0,+\infty]{} 0,
      $
      and such that there exists $\lambda_2\in\mathbb R$ such that
      \begin{equation}
        \label{e:lambda0upperbound:Lyap2}
        \cA \xi(x)\geq -\lambda_2 \xi(x),\ \forall x\in (0,+\infty).
      \end{equation}
  \end{enumerate}
  The following hold:
  \begin{enumerate}[label=(\alph*)]
    \item 
      \relax \label{i:lambda0upperbound:a}
      if $\lambda_2<\lambda_1$, then $\lambda_0\leq \lambda_2$;
    \item 
      \relax \label{i:lambda0upperbound:b}
      if $\lambda_2\leq \lambda_1$ and $\sup_{x\in(0,M)} \int_{(0,x)} \frac{\xi(y)}{\xi(x)} k(x,\mathrm dy)<+\infty$ for all $M>0$, then $\lambda_0\leq  \lambda_2$, with strict inequality if $x\in(0,+\infty)\mapsto \frac{\cA \xi(x)}{\xi(x)}$ is not constant.
  \end{enumerate}
\end{proposition}

While finding Lyapunov functions can be tricky, we show in the next section that exponentials of $s$ or of $\int_1^\cdot K(y)s(\mathrm dy)$ cover several situations and allow to recover and improve on several results in the literature.

\begin{remark}
  We emphasize that $\lambda_0$ may be characterized by other means than its definition. 
  For instance, in \cite[Proposition~3.3]{BW-gfe} it is shown that 
  $\lambda_0=-\inf\{q\in\R,\ L_{x_0,x_0}(q)<1\}$, 
  where $L_{x_0,x_0}$ is 
  defined in terms of a multiplicative functional of an auxiliary Markov process
  evaluated at the return time to $x_0$.
  In particular, \cite[Proposition~3.4]{BW-gfe} provides a upper bound for $\lambda_0$. We also refer the reader to~\cite[Section~2.2]{Cavalli2019} for a situation where the mass conservation does not hold. 
\end{remark}

\subsection{Applications}

\label{s:app}

In this section, we apply the results of sections~\ref{sec:ExistenceUniqueness}
and~\ref{sec:long-time} to different situations, focusing on
Assumptions~\ref{assumption1} and~\ref{assumptionLyapGeneral}, since
Assumptions~\ref{assumptionIrr} and~\ref{assumptionDoebSec}
are already explicit (see also
Remark~\ref{rem:onDoeblin} below). In subsection~\ref{sec:entrance}, we
provide a sufficient criterion for Assumptions~\ref{assumption1}
and~\ref{assumptionLyapGeneral} in the situation where $s(0+)>-\infty$.  In
subsection~\ref{sec:pseudo-entrance}, we consider the situation where
$\int_{(0,1)} K(y)\,s(\mathrm dy) <+\infty$ and where mass conservation holds
true.  The last two subsections are dedicated to the study of  near-critical
cases, the critical case being when $K$ is constant and $s(x)=\ln x$, in which
case it is well known that the conclusions of Theorem~\ref{thm:spectgap} do not
hold true. In subsection~\ref{sec:criticallnx}, we study the case  $s(x) \approx \ln x$
and $K$ is not constant. In subsection~\ref{sec:criticalKcst}, we study the
case $s(x)\neq \ln x$ and $K$ approximately constant.

\begin{remark}
  \label{rem:onDoeblin}
  Assumption~\ref{assumptionDoebSec} holds in the following situations.
  \begin{enumerate}[label={(\roman*)},ref={\ref{rem:onDoeblin}(\roman*)}]
    \item  
      \label{assumptionDoeb}
      Assume that there exists a positive constant $a>0$, a non-empty open
      interval~$I\subset (0,+\infty)$
      and a probability measure $\kappa$
      on
      $(0,+\infty)$,
      such that
      \[
        k(x,\mathrm dy) \geq a\kappa(\mathrm dy) \qquad x \in I.
      \]
      Then Assumption~\ref{assumptionDoebBoth} and~\ref{assumptionDoebF} both hold true. The fact that~\ref{assumptionDoebF}  holds true is immediate. For~\ref{assumptionDoebBoth}, let $\mu$ be
      the Lebesgue measure on $[0,1]$ and $T(\theta,x)=F_\kappa^{-1}(\theta)$,
      where $F_\kappa^{-1}$ is the generalized inverse of the cumulative
      distribution function of $\kappa$. Then $\kappa(\mathrm dy)=\int_{[0,1]}
      \delta_{T(\theta,x)}(\mathrm dy)\,\mu(\mathrm d\theta)$, so
      that~\eqref{eq:conditionderivativethetaTer} holds true. Since
      $T(\theta,x)$ does not depend on $x$, the rest of the assumption holds
      true.
    \item
      \label{rem:equamitosi}
      \label{assumptionDoebBis}
      Assume that $k$ is locally lower bounded by the equal mitosis kernel, that is,
      there exists a non-empty open interval $I\subset (0,+\infty)$ and $a>0$
      such that
      \[
        k(x,\dd y) \geq  a\delta_{x/2}(\dd y), \qquad x\in I,
      \]
      and moreover assume that $s(x)=\int_1^x
      1/c(y)\,\mathrm dy$ for some  positive function $c\colon (0,+\infty)\to(0+\infty)$,
      continuous on $I$, such that $c(x) \ne 2c(x/2)$ for all $x\in I$.
      Then,
      Assumption~\ref{assumptionDoebBoth} holds by taking
      $T(x)=x/2$. Clearly, Assumption~\ref{assumptionDoebF} does not hold in this situation.
    \item \label{exa:generalauto}
      More generally,
      consider the situation where 
      $k(x,\cdot)=K(x) p\circ m_x^{-1}$ with $m_x(u)=xu$ and
      $p$ is a finite
      measure on $(0,1)$,  with $s(x)=\int_1^x
      1/c(y)\,\mathrm dy$ for some  positive function $c\colon (0,+\infty)\to(0+\infty)$
      continuous on some sub-interval~$I$.
      Assume in addition $K$ is lower bounded away from $0$ on $I$ by a constant $a>0$
      and that there exists measurable $A\subset (0,1)$ such that $p(A)>0$ and
      for all $\theta\in A$,
      \begin{align}
          \label{eq:ineq}
        \theta c(x) \neq c(\theta x), \qquad x\in I.
      \end{align}
      Then, Assumption~\ref{assumptionDoebSec} holds by setting $T(\theta,x) = \theta x$
      and $\mu(\dd\theta)=p(\dd\theta)/p((0,1))$. 
      
     \item \label{exa:absocontinuous} Consider the situation where 
     $k(x,\cdot)=K(x) p\circ m_x^{-1}$ with $m_x(u)=xu$ and
     $p$ is a finite non-zero
     measure on $(0,1)$,  with $s(x)=\int_1^x
     1/c(y)\,\mathrm dy$ for some  positive function $c\colon (0,+\infty)\to(0+\infty)$
     continuous on $(0,+\infty)$. Assume in addition that $K$ is locally lower bounded away from $0$.     
     If $s(0+)<+\infty$, then Assumption~\ref{assumptionDoebBoth} holds.
     Indeed, let $\theta_0\in[0,1]$ such that $p(U)>0$ for all neighborhood $U$ of $\theta_0$. Then, since $1/c$ is summable in a neighborhood of $0$, there exists $x_0\in (0,+\infty)$ such that $\theta_0c(x_0)\neq c(\theta_0x_0)$, which then holds true for all $(\theta,x)\in A\times I$, where $A$ and $I$ are neighborhoods $\theta_0$ and $x$ respectively. This shows that~\eqref{eq:ineq} holds true for all $\theta\in A$ and $x\in I$, with $p(A)>0$.
     
     \item Finally, assume that $k(x,\cdot)=K(x) p\circ m_x^{-1}$ with $m_x(u)=xu$ and 
     $p$ is a finite
     measure on $(0,1)$ which is not singular with respect to the Lebesgue measure. Assume in addition that $K$ is  lower bounded away from $0$ on an open subset of $(0,+\infty)$.  Then Assumption~\ref{assumptionDoebF} holds true. Indeed, let $\underline K$ be a non-zero continuous function such that $0\leq \underline K\leq K$. Then $k(x,\cdot)\geq \underline K(x) p_\lambda\circ m_x^{-1}$, where $p_\lambda$ is the absolutely continuous part of $p$, and it only remains to prove that $ p_\lambda\circ m_x^{-1}$ defines a lower semi-continuous kernel. Let $g$ be the density of $ p_\lambda$ with respect to the Lebesgue measure and let $g_n$ be a sequence of continuous functions which converge to $g$ in $L^1(\mathrm dx)$. Then, for all measurable $A\subset (0,+\infty)$ and all $x\in (0,+\infty)$,
     \begin{align*}
          p_\lambda\circ m_x^{-1}(A)& = \int_{(0,1)} \1_{ux\in A}\, g(u)\,\mathrm du\\
          & \geq  \int_{(0,1)} \1_{ux\in A}\, g_n(u)\,\mathrm du -\|g-g_n\|_{L_1}\\
          & = \int_{(0,x)} \1_{y\in A}\, g_n(y/x)\,\frac{\mathrm dy}{x} -\|g-g_n\|_{L_1}
     \end{align*}
    The first term in the last line is continuous in $x$, while the second term goes to $0$ when $n\to+\infty$, so that, for all $x\in(0,+\infty)$,
    \begin{align*}
        \liminf_{y\to x} p_\lambda\circ m_y^{-1}(A)\geq p_\lambda\circ m_{x}^{-1}(A).
    \end{align*}
    This shows that Assumption~\ref{assumptionDoebF} holds true. Note that in this particular case, $p_\lambda\circ m_y^{-1}(A)$ is actually continuous with respect to $x$.
  \end{enumerate}
\end{remark}

\subsubsection{Entrance boundary}
\label{sec:entrance}

In this section, we provide a simple criterion for processes with an entrance boundary at $0$ (i.e. $s(0+)>-\infty$) and with a locally bounded fragmentation rate, inspired by the main result of~\cite{Cavalli2019}. As in this reference, and contrarily to the following sections, the result depends on $\lambda_0$.

\begin{proposition}
  \label{prop:appli0}
  Assume that $s(0+)>-\infty$, that $\sup_{x\in (0,M)} k(x,(0,x))<+\infty$ for all $M>0$, that $K$ is non-negative and that
  \begin{align}
    \label{as:1appli1}
    \limsup_{x\to+\infty} k(x,(0,x))-K(x)<+\infty.
  \end{align}
  Then Assumption~1 holds true. 
  If 
in addition Assumptions~\ref{assumptionIrr} and~\ref{assumptionDoebSec}
  hold true, and if
  \begin{align}
    \label{as:2appli1}
    \limsup_{x\to+\infty} k(x,(0,x))-K(x)<-\lambda_0,
  \end{align}
  then Assumption~\ref{assumptionLyapGeneral} holds true.
\end{proposition}

Before proceeding with the proof of Proposition~\ref{prop:appli0}, we remark on
the strong similarities with Theorem~1.1 of~\cite{Cavalli2019}.
There, the author reaches the conclusion of Theorem~\ref{thm:spectgap},
making some additional regularity and further assumptions on $s$, $k$ and $K$;
these ensure in particular that
$T$ is a strongly continuous Feller semigroup on the space of bounded
functions vanishing at infinity, which is not in general true for us.

\begin{proof}[Proof of Proposition~\ref{prop:appli0}]
  Let $a<-\limsup_{x \to 0} k(x,(0,x))-\lambda_0$ such that $a\leq 0$. Let $x_0\geq 1$ be such that $\exp(-as(0+))+s(x_0)=1$  and set, for all $x\in (0,+\infty)$,
  \begin{align*}
    h(x)=
    \exp\left(a\,(s(x)-s(0+))\right)\1_{x<1}+1\wedge\left(\exp(-as(0+))+s(x)\right)\1_{x\geq 1}.
  \end{align*}
  Then, for all $x\in(0,1)$,
  \begin{align*}
    \frac{\cA h(x)}{h(x)}&=a+\int_{(0,x)} \exp\left(a(s(y)-s(x))\right) k(x,\mathrm dy)-K(x)\\
    &\leq a+ \exp(a (s(0+)-s(x))) k(x,(0,x)),
  \end{align*}
  which is uniformly bounded from above on $x\in (0,1)$ by assumption. For $x\in[1,x_0)$, we have
  \begin{align*}
  \frac{\cA h(x)}{h(x)}
  &=\frac{1}{h(x)}+\int_{(0,1)}\frac{\exp\left(a\,(s(y)-s(0+))\right)}{\exp(-as(0+))+s(x)}\,k(x,\mathrm dy)
  \\
  & \quad {}
  + \int_{[1,x)}\frac{\exp(-as(0+))+s(y)}{\exp(-as(0+))+s(x)}k(x,\mathrm dy)-K(x)\\
  &\leq \exp(a s(0+))+\exp(as(0+))k(x,(0,1))+k(x,[1,x)),
  \end{align*}
  which is uniformly bounded from above on $x\in [1,x_0)$ by assumption.
   For $x\geq x_0$, we have
  \begin{align*}
    \frac{\cA h(x)}{h(x)}&=\int_{(0,1)} \exp\left(a\,(s(y)-s(0+))\right)k(x,\mathrm dy)\\
                         &\qquad   +\int_{[1,x_0)} \left(\exp(-as(0+))+s(y)\right) k(x,\mathrm dy)
                         +k(x,[x_0,x))-K(x)\\
                         &\leq  k(x,(0,x))-K(x),
  \end{align*}
  which is locally bounded from above and is bounded when $x\to+\infty$ 
  by~\eqref{as:1appli1}. 
  This entails that $\frac{\cA h(x)}{h(x)}$ is bounded from above on $(0,+\infty)$. 
  It is clearly locally bounded, and, in addition, for all $M>0$,
  \begin{align*}
    \sup_{x\in (0,M)} k_h(x,(0,x))
    \leq \sup_{x\in(0,M)}\int_{(0,x)} \exp\left(a s(0+)\right) k(x,\mathrm dy)
  \end{align*}
  which is finite by assumption. We conclude that Assumption~\ref{assumption1} holds true.

  We now work under the additional assumptions and set $\psi=h$. We have $s(x)\to s(0+)$ when $x\to0$, and hence
  \begin{align*}
    \limsup_{x\to0} \frac{\cA \psi(x)}{\psi(x)}
    &\leq  \limsup_{x\to 0} \bigl[ a+ \exp\left(a (s(0+)-s(x))\right) k(x,(0,x)) \bigr]
    \\
    & = a+ \limsup_{x\to 0} k(x,(0,x)) 
    < -\lambda_0.
  \end{align*}
  Using~\eqref{as:2appli1}, we also obtain
  \begin{align*}
    \limsup_{x\to+\infty} \frac{\cA \psi(x)}{\psi(x)}
    &\leq \limsup_{x\to+\infty} \bigl[k(x,(0,x))-K(x) \bigr]
    < -\lambda_0.
  \end{align*}
  This concludes the proof of Proposition~\ref{prop:appli0}.
\end{proof}

\subsubsection{Pseudo-entrance boundary and mass conservation}

\label{sec:pseudo-entrance}

In this section, we consider the situation where $\int_{(0,1)} K(x)\,s(\mathrm dx)<+\infty$. Informally, this means that a PDMP with drift determined by $s$ and jump rate $K$ has a positive, lower bounded probability to reach $1$ before its first jump when starting from any $x\in (0,1)$.

For simplicity, we consider the situation 
$k(x,\cdot)=K(x) p\circ m_x^{-1}$ where $m_x(u)=xu$, 
$p$ is a measure on $(0,1)$ such that $\int_{(0,1)}up(\mathrm du)=1$.
We also assume that $K$ is right-continuous.

\begin{proposition}
  \label{prop:appli1}
  Assume  that $\sup_{x\in (0,M)} K(x)<+\infty$ for each $M>0$. 
  Assume in addition that Assumptions~\ref{assumptionIrr} 
  and~\ref{assumptionDoebSec}
  hold true, that
  $p$ is a finite measure, that
  \begin{align*}
    \int_{(0,1)} K(x)\,s(\mathrm dx)<+\infty
  \end{align*}
  and that there exists $\alpha>1$  such that, for all $u\in(0,1)$,
  \begin{align}
    \label{as:2appli1bis}
    \liminf_{x\to+\infty}\  \int_{ux}^x K(y)\,s(\mathrm dy)> \frac{-\alpha\ln u}{1-\int_{(0,1)} v^\alpha p(\mathrm dv)}.
  \end{align}
  Then,
  Assumptions~\ref{assumption1} and~\ref{assumptionLyapGeneral} hold,
  $\lambda_0 < 0$ and the conclusions of Theorem~\ref{thm:spectgap} hold true.
\end{proposition}

Suppose moreover that, for all $x\in(0,+\infty)$, 
$s(x)=\int_1^x \frac1{c(y)}\mathrm dy$ where  
$c:(0,+\infty)\to (0,+\infty)$ is a right-continuous and locally bounded function.
In the case of uniform
mass repartition, where $p(\dd u) = 2\dd u$, 
the right hand term in~\eqref{as:2appli1bis} reaches its minimal value 
$(-\ln u)(3+2\sqrt{2})$ at $\alpha= 1+\sqrt 2$. 
In particular,~\eqref{as:2appli1bis} holds true if
\begin{align*}
  \liminf_{y\to+\infty} \frac{yK(y)}{c(y)}>3+2\sqrt{2}.
\end{align*}

Before turning to the proof of this proposition, it is interesting to compare it
with the findings of~\cite{CGY-spectral}. 
In this recent paper, the authors use advanced methods from functional analysis 
to derive the existence of an eigenfunction $h$.
This gives them access to a (conservative) Markov process using an $h$-transform
(see also~\cite{Miura2014,Ocafrain2020}, where similar approaches were used 
to study non-conservative semigroups). 
This allows them to study 
the growth fragmentation equation under mild conditions.
The main drawback of this approach is that it requires the preliminary 
proof of the existence  and fine properties of a positive right eigenfunction $h$, 
which typically requires additional assumptions on regularity
and asymptotic behaviour of the coefficients.
On the contrary, our approach, 
based on the study of non-conservative Markov processes, 
only requires the existence of a Lyapunov function $h$,
and the existence of an eigenfunction is then a consequence of our theorem, 
instead of a preliminary step in the proof.
This lets us consider more general situations.

More precisely, in the case  where $p$ is the uniform measure over 
$(0,1)$ (where Assumption~\ref{assumptionDoebSec} is clearly satisfied
by Remark~\ref{assumptionDoeb}), 
Theorem~1.3 in~\cite{CGY-spectral} states that the conclusions of 
our Theorem~\ref{thm:spectgap} hold true, 
assuming in addition (compared to Proposition~\ref{prop:appli1}) that  
$c$ is locally Lipschitz, 
that $\limsup_{x\to+\infty} \frac{c(x)}{x}<+\infty$, 
that $c(x)=o(x^{-p})$ when $x\to0$ for some $p\geq 0$, 
that $K$ is continuous on $[0,+\infty)$, 
that $xK(x)/c(x)\to 0$ when $x\to 0$ and 
that $xK(x)/c(x)\to +\infty$ when $x\to +\infty$. 
Similarly, the mitosis kernel case considered 
in~\cite{CGY-spectral} is a special case of Proposition~\ref{prop:appli1} 
(using this time Remark~\ref{assumptionDoebBis}
instead of Remark~\ref{assumptionDoeb}). 

Our result also covers and extends the setting considered
in~\cite{BernardGabriel2020}, where the authors assume either (1) that $c(x)=x$
and $p$ is absolutely continuous with respect to the Lebesgue measure (this is
a particular case of Remark~\ref{exa:absocontinuous}, our result shows in
particular that it is sufficient for $p$ to have a non-zero absolute
continuous part with respect to the Lebesgue measure) or (2) that $c\equiv 1$
and $p$ has compact support in $(0,1)$ (this is a particular case of
Remark~\ref{exa:generalauto}, and we show that the condition on $p$ can be
dispensed of entirely).

We can also compare Proposition~\ref{prop:appli1} with Theorem~4.3 in the
recent paper~\cite{BCGM-Harris}, where the authors consider the special case
where $c\equiv 1$ (which means that $s(x)=x-1$) and $K$ is a continuously
differentiable increasing function, and under the additional assumption that
$p$ is lower bounded by a uniform measure over a subinterval of $[0,1]$ or by a
Dirac measure (these situations clearly satisfy
Assumption~\ref{assumptionDoebBoth} via Remarks~\ref{exa:absocontinuous}
and~\ref{assumptionDoebBis} respectively). In this situation, both assumptions
of Proposition~\ref{prop:appli1} are clearly satisfied, with
$\liminf_{x\to+\infty}\  \int_{ux}^x K(y)\,s(\mathrm dy)=+\infty$ for all
$u\in(0,1)$ and Theorem~4.3 in~\cite{BCGM-Harris} is thus a special case of
Proposition~\ref{prop:appli1}.

\begin{remark}
  \label{rem:pseudo}
  In the proof, we make use of the functions $\psi$ and $h$ defined by 
  \[
    \psi(x)
    =h(x)
    = \exp\left(\int_{(x,1)} a_0 K(y)\,s(\mathrm dy)\right)\1_{x<1}
    +\exp\left(\int_{(1,x)} a_\infty K(y)\,s(\mathrm dy)\right)\1_{x\geq 1},
  \]
  where $a_0,a_\infty\in\mathbb R$, so that, for all $x<1$,
  \[
    \frac{\cA \psi(x)}{\psi(x)}
    = K(x)
    \left( \int_{(0,1)} \exp\left(a_0\int_{(ux,x)}\,K(y)\,s(\mathrm dy)\right) \,p(\dd u)
    - 1 - a_0 \right)
  \]
  and similarly for $x\geq 1$ (see \eqref{e:xge1} for the exact expression). 
  Our assumptions are then used to derive
  asymptotics on  $\frac{\cA \psi(x)}{\psi(x)}$ when $x\to 0$ and
  $x\to+\infty$. In this situation, the main point of the mass conservation
  assumption is to ensure that $x\in(0,+\infty)\mapsto x$ is a natural
  candidate for $\xi$ in Proposition~\ref{prop:lambda0upperbound}, and is
  thus used to derive a lower bound for $\lambda_0$. The strategy developed in
  the proof, and in particular the use of $\psi$ with this form, is relevant
  even outside of the structure for $k$ assumed in Proposition~\ref{prop:appli1}.

 For example, let us assume that Assumptions~\ref{assumptionIrr} and~\ref{assumptionDoebSec} hold, and let us make the following assumptions:
$K$ is locally bounded, $\sup_{x>0}\frac{k(x,(0,x)}{K(x)}<+\infty$,
$\int \frac{y}{x} k(x,\mathrm{d}y) = K(x)$ (conservation of mass), 
there exists $\alpha>1$ such that 
$m_\alpha:=\limsup_{x\to\infty} \int (y/x)^\alpha k(x,\mathrm{d}y)/K(x)<1$, $c$ is right continuous, $\int_{(0,1)} \frac{K(x)}{c(x)}\,\mathrm dx<+\infty$ (pseudo-entrance boundary),
and there exists $a_\infty\in(0,m_\alpha)$ such that, for all $x\geq y\geq 1$ with $y$ large enough,
\begin{align}
    \label{eq:new-condition}
\exp\left(-a_\infty \int_{y}^x K(z)\,s(\mathrm dz)\right)\leq \left(\frac{y}{x}\right)^{\alpha}.
\end{align}
We emphasize that this implies that $\frac{1}{\ln x}\int_{1}^x K(z)\,s(\mathrm dz)$ goes to $+\infty$ when $x\to+\infty$.
Then, using the same $\psi$ as in the proof, we obtain for all $x\geq 1$, 
\begin{align*}
    \frac{\mathcal A \psi(x)}{\psi(x)}&= K(x)\Bigg(\int_{(0,1)} \exp\left(a_0 \int_{y}^1 K(z)\,s(\mathrm dz)-a_\infty \int_1^x  K(z)\,s(\mathrm dz)\right)\frac{1}{K(x)}k(x,\mathrm dy) \nonumber \\
    &\qquad\qquad+\int_{[1,x)} \exp\left(-a_\infty \int_{y}^x K(z)\,s(\mathrm dz)\right)\frac{1}{K(x)}k(x,\mathrm dy)-1+a_\infty\Bigg).
\end{align*}
The first term in the parenthesis goes to $0$ when $x\to+\infty$ since $\int_0^1 K(z)\,s(\mathrm dz)<+\infty$, $\int_1^x  K(z)\,s(\mathrm dz)\to+\infty$ when $x\to+\infty$ and the total mass of $\frac{1}{K(x)}k(x,\cdot)$ is uniformly bounded. For the other terms, we obtain using~\eqref{eq:new-condition} that, for all $x\geq y\geq 1$ with $y$ large enough,
\[
\limsup_{x\to+\infty} \int_{[1,x)} \exp\left(-a_\infty \int_{y}^x K(z)\,s(\mathrm dz)\right)\frac{1}{K(x)}k(x,\mathrm dy)-1+a_\infty\leq m_\alpha-1+a_\infty<0.
\]
This entails that
\[
\lim_{x\to+\infty} \frac{\mathcal A \psi(x)}{\psi(x)}<0.
\]
Choosing $a_0>\sup_{x>0}\frac{k(x,(0,x)}{K(x)}$ and proceeding as in the proof, we deduce that
$\lambda_0 < 0$ and the conclusions of Theorem~\ref{thm:spectgap} hold true.

In particular, this recovers and improve the convergence of the model
in~\cite{DebiecDoumicEtAl2018}, where the authors prove the (a priori
non-geometric) convergence for a more specific model under stronger assumptions
inherited from~\cite{DoumGab} (in order to obtain the leading eigenelements)
and additional regularity assumptions (see
(2.5)--(2.7) in~\cite{DebiecDoumicEtAl2018} and (2.1)--(2.9)
in~\cite{DoumGab}). We emphasize that, although we proved that there is
actually a spectral gap under the assumptions of~\cite{DebiecDoumicEtAl2018}
and~\cite{DoumGab},  the authors of the former use the latter to obtain the
existence, uniqueness and additional properties on the eigenelements of the
adjoint operator, so their methods may apply to situations where there is no
spectral gap and thus where our result does not hold true.
Similarly, we also cover the setting of~\cite{CCM11}, where the authors prove a convergence with geometric rate for two specific models under additional regularity assumptions (see Hypotheses 1.1 to 1.7 therein), some of which are also derived from their use of~\cite{DoumGab}.

To conclude this remark, we consider the sharp and quite explicit result
of~\cite{MaillardPaquette2020},
where the authors consider the special case where $s(x)=\ln x$, $K(x)=x R(x)$
and $k(x,\mathrm du)= x\,R(x) \frac{2u}{x^2}\,\mathrm du$. There is no mass
conservation in this case, but instead conservation of the number of fragments
($\mathcal A 1\equiv 0$). The authors show that the condition 
\[
\int_0^\infty y\,\exp\left(-\int_1^x R(y)\,\mathrm dy\right)\mathrm dx <+\infty
\]
holds true if and only if $\left\|\delta_x T_t -m\right\|_{TV}$ goes to $0$
when $t\to+\infty$ (without a geometric rate). In order to ensure that
Theorem~\ref{thm:spectgap} applies to this case (and thus to ensure geometric
convergence), and using the same Lyapunov function as above (observing that
$\lambda_0=0$), we require that there exists $a_\infty\in(0,1)$ such that
\begin{align*}
    \int_{(0,1)} R(y)\,\mathrm dy<+\infty\text{ and }\limsup_{x\to+\infty} \int_{(1,x)} \exp\left(-a_\infty\int_y^x R(z)\,\mathrm dz\right)\frac{2y\mathrm dy}{x^2}-1+a_\infty <0.
\end{align*}
The last property holds true for instance if 
$\int_y^x R(z)\,\mathrm dz\geq (2+\varepsilon)\ln\frac{y}{x}$ 
for some $\varepsilon>0$ and all $x>y$ with $y$ large enough. In this case, it
is clear that the condition of~\cite{MaillardPaquette2020} holds true. It is
also the case that their condition (and of course ours) does not hold true if
$\int_y^x R(z)\,\mathrm dz\leq  2\ln\frac{y}{x}$ for all $x>y$ with $y$ large
enough.

\end{remark}

\begin{proof}[Proof of Proposition~\ref{prop:appli1}]
  For all $a\in \mathbb R$, we set $p_a:=\int_{(0,1)} u^a \,p(\mathrm du)$.

  \medskip \textbf{(1) Identification of $h=\psi$.} For all $u\in(0,1)$, we define
  \[
  \varepsilon_u:=\liminf_{x\to+\infty}\  \frac{\int_{ux}^x K(y)\,s(\mathrm dy)}{-\ln u}-\frac{\alpha}{1-p_\alpha}
  \]
  and set
  \[
  \ell:=\frac{\alpha}{1-p_\alpha}+\frac{\varepsilon_{1/2}}{2}.
  \]
  Note that $\varepsilon_u>0$ by assumption and hence  $\alpha/\ell<1-p_\alpha$ and
  \[
  \lim_{a\to 1-p_\alpha} \int_{(0,1)} u^{a(\varepsilon_u+\alpha/(1-p_\alpha))}\,p(\mathrm du)=\int_{(0,1)} u^{(1-p_\alpha)\varepsilon_u+\alpha}\,p(\mathrm du)<p_\alpha.
  \]
  In particular, there exists $a_\infty\in \left(\frac{\alpha}{\ell},1-p_\alpha\right)$ such that
  \begin{equation}
  \label{eq:defainfty}
  \int_{(0,1)}  u^{a_\infty(\varepsilon_u+\alpha/(1-p_\alpha))}\,p(\mathrm du)<p_\alpha.
  \end{equation}
  We also fix $a_0>p_0-1$ and define the function
  \begin{align*}
    \psi(x)=\begin{cases}
      \exp\left(-a_0\int_1^x K(y)\,\mathrm s(\mathrm dy)\right)&\text{ if }x\leq 1\\
      \exp\left(a_\infty\int_1^x K(y)\,s(\mathrm dy)\right)&\text{ if }x\geq 1.
    \end{cases}
  \end{align*}
  We have, for all $x < 1$,
  \begin{align*}
    \frac{\mathcal A \psi(x)}{\psi(x)}= K(x)\left(\int_{(0,1)} \exp\left(a_0 \int_{ux}^x K(y)\,s(\mathrm dy)\right)p(\mathrm du)-1-a_0\right).
  \end{align*}
  Since $\exp\left(a_0 \int_{ux}^x K(y)\,s(\mathrm dy)\right)\leq  \exp\left(a_0 \int_{0}^1 K(y)\,s(\mathrm dy)\right)$, with 
  \[
    \int_{(0,1)} \exp\left(a_0 \int_{0}^1 K(y)\,s(\mathrm dy)\right)p(\mathrm du)
    = \exp\left(a_0 \int_{0}^1 K(y)\,s(\mathrm dy)\right) p_0<+\infty
  \]
  and since $\int_{ux}^x K(y)\,s(\mathrm dy)\to 0$
  as $x\to 0$,
  we deduce from the dominated convergence theorem that
  \[
    \lim_{x\to 0} \int_{(0,1)} \exp\left(a_0 \int_{ux}^x K(y)\,s(\mathrm dy)\right)p(\mathrm du)-1-a_0=p_0-1-a_0<0.
  \]
  Hence there exists $x_0>0$ such that
  \begin{align}
    \label{eq:Apsipsi1}
    \frac{\mathcal A \psi(x)}{\psi(x)}\leq 0\text{,  for all $x\in(0,x_0)$.}
  \end{align}
  For all $x\geq 1$, we have
  \begin{align}
    \frac{\mathcal A \psi(x)}{\psi(x)}&= K(x)\Bigg(\int_{(0,1/x)} \exp\left(a_0 \int_{ux}^1 K(y)\,s(\mathrm dy)-a_\infty \int_1^x  K(y)\,s(\mathrm dy)\right)p(\mathrm du) \nonumber \\
    &\qquad\qquad+\int_{[1/x,1)} \exp\left(-a_\infty \int_{ux}^x K(y)\,s(\mathrm dy)\right)p(\mathrm du)-1+a_\infty\Bigg). \label{e:xge1}
  \end{align}
  On the one hand, we have (noting that $a_\infty>0$)
  \begin{multline*}
    \int_{(0,1/x)} \exp\left(a_0 \int_{ux}^1 K(y)\,s(\mathrm dy)-a_\infty \int_1^x  K(y)\,s(\mathrm dy)\right)p(\mathrm du) \\
    \leq \exp\left(a_0 \int_{0}^1 K(y)\,s(\mathrm dy)\right)\,p((0,1/x))
    \xrightarrow[x\to+\infty]{}0.
  \end{multline*}
  On the other hand, for all $u\in(0,1)$,
  \begin{align*}
    \limsup_{x\to+\infty} \,\mathbf 1_{u\in[1/x,1)}\,\exp\left(-a_\infty \int_{ux}^x K(y)\,s(\mathrm dy)\right)=u^{a_\infty(\varepsilon_u+\alpha/(1-p_\alpha))}
  \end{align*}
  and hence, by Fatou's Lemma and using \eqref{eq:defainfty},
  \begin{align*}
    \limsup_{x\to+\infty} \int_{[1/x,1)} \exp\left(-a_\infty \int_{ux}^x K(y)\,s(\mathrm dy)\right)p(\mathrm du)-1+a_\infty
    \le p_\alpha-1+a_\infty
    < 0.
  \end{align*}
  We deduce that there exists $x_\infty\geq 1$ such that, for all $x\geq x_\infty$,
  \begin{align}
    \label{eq:Apsipsi2}
    \frac{\mathcal A \psi(x)}{\psi(x)}\leq 0.
  \end{align}
  Taking $h=\psi$, we observe that $x\in(0,+\infty)\mapsto \frac{\mathcal A h(x)}{h(x)}$ is locally bounded, and we deduce from~\eqref{eq:Apsipsi1} and~\eqref{eq:Apsipsi2} that it is bounded from above. Moreover, the above calculations show that, for all $M>0$,
  \[
    \sup_{x\in (0,M)} \int_{(0,x)} \frac{h(y)}{h(x)} k(x,\mathrm dy)<+\infty.
  \]
  We conclude that Assumption~1 holds true.

  \medskip \textbf{(2) Identification of $\xi$ and conclusion.} We choose $\xi(x):=x$ for all $x\in(0,+\infty)$.  
  We first prove that $\xi(x)=x=o(\psi(x))$ close to $0$ and $+\infty$. Since $\psi$ is bounded away from $0$ in a vicinity of $0$, this is immediate for $x$ close to $0$.  Now, according to our assumptions and the definition of $\ell$, there exists $x_1\geq 1$ (which is fixed from now on) such that, for all $x\geq x_1$,
  \[
    \int_{x/2}^x K(y)\,s(\mathrm dy)\geq \ell \ln 2.
  \]
  For any  $x>x_1$, let $n\geq 0$ such that $2^{-n}x\geq x_1\geq 2^{-(n+1)}x$ (in particular $n\ln 2\geq \ln x-\ln x_1 -\ln 2$). Then
  \begin{align*}
    \int_1^x K(y)\,s(\mathrm dy)
    & \geq \int_1^{2^{-n}x} K(y)\,s(\mathrm dy)
    + \int_{2^{-n}x}^{2^{-(n-1)}x} K(y)\,s(\mathrm dy)
    + \cdots 
    + \int_{2^{-1}x}^x K(y)\,s(\mathrm dy) \\
    & \geq n \ell \ln 2
    \geq \ell\ln x - \ell \ln (2x_1).
  \end{align*}
  Since $a_\infty> \alpha/\ell$, we deduce that, for all $x>x_1$,
  \begin{align*}
    a_\infty\int_1^x K(y) \, s(\mathrm dy)
    & \geq \alpha\ln x - a_\infty\ell\ln (2x_1).
  \end{align*}
  This shows that 
  $\liminf_{x\to +\infty} \psi(x)/x^\alpha > 0$ and hence, 
  since $\alpha >1$ by assumption, that $x=o(\psi(x))$ when $x\to+\infty$.

  We also observe that, for all $M>0$, 
  $\sup_{x\in (0,M)} \int_{(0,x)}\frac{\xi(y)}{\xi(x)}k(x,\mathrm dy)
  =\sup_{x\in (0,M)} K(x)<+\infty$,
  by assumption.
  Finally, for all $x\in (0,+\infty)$,
  \begin{align*}
    \frac{\mathcal A \xi(x)}{\xi(x)}=\frac{1}{x}\frac{\partial \xi}{\partial s}(x)=\frac{c(x)}{x}.
  \end{align*}
  Since $c(x)/x$ is not zero, we deduce that it is either lower bounded 
  by a positive constant or that it is not constant. Using 
  Proposition~\ref{prop:lambda0upperbound} together with~\eqref{eq:Apsipsi1} 
  and~\eqref{eq:Apsipsi2}, we deduce that $\lambda_0 < 0$.
  This also entails that Assumption~\ref{assumptionLyapGeneral} holds true, 
  which concludes the proof.
\end{proof}

\subsubsection{Critical case, \texorpdfstring{$s$}{s} comparable to \texorpdfstring{$\ln x$}{ln x}}
\label{sec:criticallnx}

It is well known that, when $K$ is constant and $s(x)=\ln x$, 
the results of Theorem~\ref{thm:spectgap} do not hold true in general 
(see, for instance,~\cite[end of \S 2]{DoumGab}).
In this section, we consider first the situation where 
$s(x)=\ln x$ and $K$ is not constant, 
and then the situation where $s(x)/\ln x$ has positive limit inferior
when $x\to 0$ and $x\to+\infty$
and finite limit superior when  $x\to+\infty$.

As in the previous section, we consider for simplicity the situation 
where $k(x,\cdot)=K(x)p\circ m_x^{-1}$, with $p$ a positive measure on $(0,1)$ 
such that $\int_{(0,1)} u\,p(\mathrm du)=1$;
we do not assume that $p$ is a finite measure. 
We assume that $K$ is right-continuous and non-negative.

\begin{proposition}
  \label{prop:lnx}
  Assume that Assumptions~\ref{assumptionIrr}
 and~\ref{assumptionDoebSec}
  hold true.
  Assume in addition that $s(x)=\ln x$ for all $x\in(0,+\infty)$ 
  and that there exist $\alpha<1<\beta$ such that 
  $\int_{(0,1)} u^\alpha\,p(\mathrm du)<\infty$ and
  \begin{equation}
    \label{as:1appli2}
    \limsup_{x\to 0} K(x) < \frac{1-\alpha}{\int_{(0,1)} u^\alpha\,p(\mathrm du)-1}
    \text{ and }
    \liminf_{x\to\infty} K(x) > \frac{\beta-1}{1-\int_{(0,1)} u^\beta\,p(\mathrm du)}.
  \end{equation}
  Then Assumptions~\ref{assumption1} and~\ref{assumptionLyapGeneral} hold true.
\end{proposition}

We note that
in the case of uniform
mass repartition, i.e.\ $p(\dd u) = 2\dd u$, condition \eqref{as:1appli2}
reduces to
\[
  \limsup_{x\to 0} K(x) < 2 < \liminf_{x\to\infty} K(x).
\]
This may be compared with the conditions in section 6 of \cite{BW-lln},
whose effectiveness in this setting relies on \cite[Theorem~1.2]{Ber-fk}.
Similar conditions can be found in \cite[section~3.6]{Ber-fk}.
We leave as an open problem to check whether this condition is sharp;
one natural approach to this question would be to follow the strategy 
developed in~\cite{Cav-refr}.

Proposition~\ref{prop:lnx} is actually a particular case of the following result, 
which applies when the drift $c(x)$ is only approximately linear in $x$. 
We assume here that, for all $x\in(0,+\infty)$, 
$s(x)=\int_1^x \frac1{c(y)}\mathrm dy$, where  
$c:(0,+\infty)\to (0,+\infty)$ is a right-continuous and locally bounded function.

\begin{proposition}
  \label{prop:lnxmodified}
 Assume that Assumptiosn~\ref{assumptionIrr}
  and~\ref{assumptionDoebSec}
  hold true. 
  Assume in addition that there exist $\alpha,\beta\geq 0$ such that
  \begin{align}
    \label{as:appli3as1}
    \alpha<\inf_{x> 0}\frac{c(x)}{x}
    \quad\text{ and }\quad 
    \int_{(0,1)} u^{\alpha\,\inf_{x<1}\frac{x}{c(x)}}\,p(\mathrm du)<+\infty
  \end{align}
  and
  \begin{align}
    \label{as:appli3as2}
    \beta>\limsup_{x\to+\infty}\frac{c(x)}{x}
    \quad\text{ and }\quad 
    \int_{(0,1)} u^{\beta\,\inf_{x\geq 1}\frac{x}{c(x)}}\,p(\mathrm du)<+\infty.
  \end{align}
  If 
  \begin{align}
    \label{as:appli3as3}
    \limsup_{x\to 0} K(x)<\frac{\inf_{x}\frac{c(x)}{x}-\alpha}{\int_{(0,1)} u^{\alpha\,\liminf_{x\to 0}\frac{x}{c(x)}}\,p(\mathrm du)-1}
  \end{align}
  and
  \begin{align}
    \label{as:appli3as4}
    \liminf_{x\to +\infty} K(x)>\frac{\beta-\inf_{x}\frac{c(x)}{x}}{1-\int_{(0,1)} u^{\beta\,\liminf_{x\to +\infty}\frac{x}{c(x)}}\,p(\mathrm du)},
  \end{align}
  then Assumptions~\ref{assumption1} and~\ref{assumptionLyapGeneral} hold true.
\end{proposition}

\begin{proof} 
  For all $a\in \mathbb R$, we set $p_a:=\int_{(0,1)} u^a \,p(\mathrm du)$.

  Note that $\limsup_{x\to 0}K(x)<+\infty$ and hence, since $K$ is locally bounded, $K$ is bounded on $(0,M)$, for all $M>0$.
  We define, for all $x\in (0,+\infty)$,
  \[
    \psi(x)=h(x)=\exp\left(\alpha s(x)\right)\1_{x<1}+\exp\left(\beta s(x)\right)\1_{x\geq 1}\text{ and }\xi(x)=x.
  \]
  In particular, for all $x\in (0,+\infty)$, 
  \begin{align*}
    \frac{\mathcal A \xi(x)}{\xi(x)}=\frac{c(x)}{x}.
  \end{align*}

  We first prove that $\psi/\xi\to+\infty$ when $x\to 0$ and $+\infty$. 
  According to~\eqref{as:appli3as1}, there exists 
  $x_0\in(0,1)$ and $\varepsilon>0$ such that, 
  for all $y\in (0,x_0)$, $\alpha/c(y)\leq (1-\varepsilon)/y$, 
  so that, for all $x\in (0,x_0)$,
  \begin{equation*}
    \alpha s(x)-\ln x
    =\int_{(x,1)} \left(\frac{-\alpha}{c(y)}+\frac{1}{y}\right) \,\mathrm dy
    \geq \varepsilon \int_{(x,x_0)} \frac{1}{y} \,\mathrm dy
    +\int_{(x_0,1)} \left(\frac{-\alpha}{c(y)} +\frac{1}{y}\right) \,\mathrm dy
    \xrightarrow[x\to 0]{}+\infty.
  \end{equation*}
  This shows that $\psi/\xi\to+\infty$ when $x\to 0$. 
  Similarly,~\eqref{as:appli3as2} implies that there exists $x_\infty\geq 1$ 
  and $\varepsilon>0$ such that, for all $y> x_\infty$, 
  $\beta/c(y)\geq (1+\varepsilon)/y$, so that, for all $x>x_\infty$,
  \begin{align}
    \label{eq:betasx}
    \beta s(x)-\ln x
    =\int_{(1,x)}  \left(\frac{\beta }{c(y)}-\frac{1}{y}\right) \,\mathrm dy
    \geq  \int_{(1,x_\infty)}  \left(\frac{\beta }{c(y)}-\frac{1}{y}\right) 
    +\varepsilon \int_{(x_\infty,x)} \frac{1}{y}\,\mathrm dy
    \xrightarrow[x\to +\infty]{}+\infty.
  \end{align}
  This shows that $\psi/\xi\to+\infty$ when $x\to +\infty$.

  We observe that, for all $x\in (0,1)$,
  \begin{align*}
    \frac{\cA h(x)}{h(x)}
    &=\alpha+ \int_{(0,x)} \frac{h(y)}{h(x)}\,k(x,\mathrm dy)-K(x) \\
    &= \alpha+K(x)\left(\int_{(0,1)} \exp(\alpha(s(ux)-s(x)))\,p(\mathrm du)-1\right).
  \end{align*}
  We have, for all $u\in(0,1)$ and $x\in(0,1)$,
  \begin{align*}
    s(ux)-s(x)\leq -\left(\inf_{y\in(0,1)}\frac{y}{c(y)} \right)\int_{ux}^x \frac{1}{y} \,\mathrm dy = \left(\inf_{y\in(0,1)}\frac{y}{c(y)} \right)\ln u,
  \end{align*}
  so that $ \exp(\alpha(s(ux)-s(x)))\leq u^{\alpha\inf_{y\in(0,1)}\frac{y}{c(y)}}$, which does not depend on $x$ and is integrable with respect to $p(\mathrm du)$ by Assumption~\eqref{as:appli3as1}. We conclude that 
  \begin{align}
    \label{eq:appli3assumption1leq1}
    \sup_{x\in(0,1)} \int_{(0,x)} \frac{h(y)}{h(x)}\,k(x,\mathrm dy)<+\infty.
  \end{align}
  In addition, for all $u\in(0,1)$,
  \begin{align*}
    \limsup_{x\to 0} \bigl(s(ux)-s(x)  \bigr)
    & \leq \limsup_{x\to 0} \left(\inf_{y\in(0,x)}\frac{y}{c(y)}\right) \ln u 
    = \liminf_{x\to0} \frac{x}{c(x)} \ln u.
  \end{align*}
  Using Fatou's Lemma, we deduce that
  \begin{align*}
    \limsup_{x\to 0}  \int_{(0,x)} \frac{h(y)}{h(x)}\,k(x,\mathrm dy)-K(x)
    \leq \limsup_{x\to 0} K(x)\,\left(
      \int_{(0,1)} u^{\alpha\,\liminf_{y\to 0}\frac{y}{c(y)}}\,p(\mathrm du)-1
    \right).
  \end{align*}
  We conclude, using in addition~\eqref{as:appli3as3} and the fact that $\alpha\liminf_{x\to 0}\frac{x}{c(x)}<1$, that
  \begin{align}
    \label{eq:ahhxto0}
    \limsup_{x\to0} \frac{\cA h(x)}{h(x)}
    \leq \alpha+\limsup_{x\to 0} K(x)\left(
      \int_{(0,1)} u^{\alpha\,\liminf_{y\to 0}\frac{y}{c(y)}}\,p(\mathrm du)-1
    \right)
    <\inf_x\frac{\mathcal A \xi(x)}{\xi(x)}.
  \end{align}

  For all $x\geq 1$, we have
  \begin{align}
    \frac{\cA h(x)}{h(x)}&=\beta+ \int_{(0,x)} \frac{h(y)}{h(x)}\,k(x,\mathrm dy)-K(x)\notag\\
    &=\beta+K(x)\left(\int_{(0,1/x)} \exp(\alpha s(ux)-\beta s(x))\,p(\mathrm du) \right.
    \notag
    \\
    &\hphantom{{} =\beta+K(x)\left(\int_{(0,1/x)} \right.}
    \left.
    {} +\int_{(1/x,1)} \exp(\beta(s(ux)-s(x)))\,p(\mathrm du)-1\right).\label{eq:ahhinfty}
  \end{align}
  According to~\eqref{eq:betasx}, there exists $x'_\infty\geq 1$ such that, for all $x\in (x'_\infty,+\infty)$, $\beta s(x)\geq \ln x$, so that, for all  $x\in(x'_\infty,+\infty)$ and $u\in(0,1/x)$,
  \begin{align*}
    \alpha s(ux)-\beta s(x)\leq \alpha \left(\inf_{y\in(0,1)}\frac{y}{c(y)} \right) \left(\ln u+\ln x\right)-\ln x\leq \alpha\left(\inf_{y\in(0,1)}\frac{y}{c(y)} \right) \ln u,
  \end{align*}
  since $\alpha \inf_{y\in(0,1)}\frac{y}{c(y)}< 1$ by~\eqref{as:appli3as1}. 
  Since, by~\eqref{as:appli3as1}, 
  $u^{\alpha \inf_{y\in (0,1)}\frac{y}{c(y)}}$
  is integrable with respect to $p(\mathrm du)$, we deduce by dominated convergence that
  \begin{align*}
    \int_{(0,1/x)} \exp(\alpha s(ux)-\beta s(x))\,p(\mathrm du)\xrightarrow[x\to +\infty]{} 0.
  \end{align*}
  For all $x> 1$ and $u\in(1/x,1)$, we have
  \begin{align*}
    s(ux)-s(x)\leq \left(\inf_{y\geq 1} \frac{y}{c(y)}\right)\ln u,
  \end{align*}
  so that $ \exp(\beta(s(ux)-s(x)))\leq u^{\beta \inf_{y\geq 1} \frac{y}{c(y)}}$, which does not depend on $x$ and is integrable with respect to $p(\mathrm du)$ by~\eqref{as:appli3as2}. We conclude that
  \begin{align}
    \label{eq:appli3assumption1geq1}
    \sup_{x\in[1,M)} \int_{(0,x)} \frac{h(y)}{h(x)}\,k(x,\mathrm dy)<+\infty,\ \forall M>1.
  \end{align}
  Similarly as above, we have in addition, for all $u\in(0,1)$,
  \begin{align*}
    \limsup_{x\to +\infty} \bigl(s(ux)-s(x)\bigr)
    = \liminf_{x\to+\infty} \frac{x}{c(x)} \ln u.
  \end{align*}
  Using again Fatou's Lemma, we obtain
  \begin{align*}
    \limsup_{x\to +\infty} \int_{(1/x,1)} \exp(\beta(s(ux)-s(x)))\,p(\mathrm du)
    \leq \int_{(0,1)} u^{\beta \liminf_{x\to+\infty} \frac{x}{c(x)}}\,p(\mathrm du).
  \end{align*}
  Using~\eqref{eq:ahhinfty}, we deduce that
  \begin{align}
    \label{eq:ahhxtoinfty}
    \limsup_{x\to+\infty} \frac{\cA h(x)}{h(x)}
    \leq \beta
    + \limsup_{x\to+\infty} K(x)
    \left(
      \int_{(0,1)} u^{\beta \liminf_{x\to+\infty} \frac{x}{c(x)}} \,p(\mathrm du) 
      - 1
    \right)
    < \inf_x\frac{\mathcal A \xi(x)}{\xi(x)},
  \end{align}
  where we used~\eqref{as:appli3as4} and the fact that 
  $\beta \liminf_{x\to+\infty} \frac{x}{c(x)}>1$ for the last inequality.

  By~\eqref{eq:appli3assumption1leq1} and~\eqref{eq:appli3assumption1geq1}, we deduce that the first part of Assumption~\ref{assumption1} holds true. Since $\cA h/h$ is locally bounded, we deduce from~\eqref{eq:ahhxto0} and~\eqref{eq:ahhxtoinfty} that it is bounded from above. We conclude that Assumption~\ref{assumption1} holds true.

  Finally, using Proposition~\ref{prop:lambda0upperbound}, 
  we deduce from~\eqref{eq:ahhxto0} and~\eqref{eq:ahhxtoinfty} that 
  Assumption~\ref{assumptionLyapGeneral} holds true.
\end{proof}

To once again give an explicit example, we offer:

\begin{corollary}
Assume that Assumptions~\ref{assumptionIrr} 
    and~\ref{assumptionDoebSec}
  hold true.
  Let  $p(\dd u) = 2\dd u$ and
  \[
    c(x) =
    \begin{cases}
      c_0x, & x \le x_c, \\
      c_\infty x, & x > x_c,
    \end{cases}
  \]
  for some $x_c>0$ and $0<c_\infty < c_0 < \infty$.
  Assume that 
  \begin{equation}\label{e:reggen-ex-0}
    \limsup_{x \to 0} K(x)
    <
    3c_0 - c_\infty-2\sqrt{2c_0(c_0-c_\infty)}
  \end{equation}
  and $\liminf_{x\to \infty} K(x) > 2c_\infty$.
  Then, the conditions of Proposition~\ref{prop:lnxmodified} are satisfied.
\end{corollary}
\begin{proof}
  With this particular choice of $c$, the conditions of Proposition~\ref{prop:lnxmodified} are that there exist $\alpha\in[0,c_\infty)$ and $\beta>c_\infty$, such that
  \begin{align}
    \label{as:appli3as3incor}
    \limsup_{x\to 0} K(x)<\frac{c_\infty-\alpha}{\int_{(0,1)} u^{\alpha/c_0} 2\mathrm du-1}=\frac{(\alpha+c_0)(c_\infty-\alpha)}{c_0-\alpha}
  \end{align}
  and
  \begin{align}
    \label{as:appli3as4incor}
    \liminf_{x\to +\infty} K(x)>\frac{\beta-c_\infty}{1-\int_{(0,1)} u^{\beta/c_\infty}\,2\,\mathrm du}=\beta+c_\infty.
  \end{align}

  The maximum of $\alpha\mapsto \frac{(\alpha+c_0)(c_\infty-\alpha)}{c_0-\alpha}$
  on the domain $\alpha \in[0, c_\infty)$
  is given by the right-hand side of \eqref{e:reggen-ex-0}, which shows
  that the condition given suffices for~\eqref{as:appli3as3incor} to hold.

  Since
  $\liminf_{x\to\infty}K(x) > 2c_\infty$, one can find $\beta>c_\infty$ 
  such that~\eqref{as:appli3as4incor} holds, which  concludes the proof.
\end{proof}

\subsubsection{Critical case, \texorpdfstring{$K$}{K} comparable to a constant}

\label{sec:criticalKcst}

We consider now the situation where $K$ is the constant function $1$,
and then the situation when $K$ is bounded away from $0$ 
and bounded from above by $1$.

As in the previous section, we consider for simplicity the situation where $k(x,\cdot)=K(x)p\circ m_x^{-1}$, with $p$ a positive measure on $(0,1)$ such that $\int_{(0,1)} u\,p(\mathrm du)=1$. We also assume  that, for all $x\in(0,+\infty)$, $s(x)=\int_1^x \frac1{c(y)}\mathrm dy$ where  $c:(0,+\infty)\to (0,+\infty)$ is a right-continuous and locally bounded function.

\begin{proposition}
 Assume that  Assumptions~\ref{assumptionIrr} 
  and~\ref{assumptionDoebSec}
  hold true, and that $K \equiv 1$.
  Assume in addition that there exists $\delta>0$ such that $\int_{(0,1)} u^{-\delta}\,p(\mathrm du)<+\infty$. If $s(0+)=-\infty$ and  
  \begin{align}
    \label{as:appli4as1}
    \limsup_{x\to+\infty} \frac{c(x)}{x} < -\int_{(0,1)}\ln u\,p(\mathrm du) < \liminf_{x\to0} \frac{c(x)}{x},
  \end{align}
  then Assumptions~\ref{assumption1} and~\ref{assumptionLyapGeneral} hold true.
\end{proposition}

In~\cite{Cav-refr}, the author considers the case where $p(\mathrm du)$ 
is absolutely continuous with respect to the Lebesgue measure and where 
there exist positive constants $a_-$ and $a_+$ such that
\begin{align*}
  c(x)=\begin{cases}
    a_- x&\text{ if }x<1,\\
    a_+ x&\text{ if }x\geq 1.
  \end{cases}
\end{align*}
In this case, our assumption reads
\begin{align*}
  a_+ < -\int_{(0,1)}\ln u\,p(\mathrm du) < a_-,
\end{align*}
which is sharp, according to~\cite{Cav-refr}, 
in the sense that, if one of the inequalities fails, 
then $e^{\lambda_0 t} T_t f$ does not converge 
(for some bounded, compactly supported function $f$). 
Additional properties, and in particular fine estimates 
on the limiting profile of $e^{\lambda_0 t} T_t$, can be found in the above reference.

The previous result is a particular case of the following proposition, 
where we do not assume any more that $K$ is constant. 
Here $K$ is a locally bounded right-continuous function.

\begin{proposition}
  \label{prop:appli4}
Assume that Assumptions~\ref{assumptionIrr} 
    and~\ref{assumptionDoebSec}
  holds true.
  Assume in addition that there exists $\delta>0$ such that $\int_{(0,1)} u^{-\delta}\,p(\mathrm du)<+\infty$, and $0<\inf K\leq 1$. If $s(0+)=-\infty$,
  \begin{align}
    \inf K=\limsup_{x\to 0} K(x)=\limsup_{x\to+\infty} K(x)
  \end{align}
  and
  \begin{align}
    \label{as:appli4as1bis}
    \limsup_{x\to+\infty} \frac{c(x)}{x}
    < -\int_{(0,1)}\ln u\,p(\mathrm du) 
    < \liminf_{x\to 0} \frac{c(x)}{x},
  \end{align}
  then Assumptions~\ref{assumption1} and~\ref{assumptionLyapGeneral} hold true.
\end{proposition}

We start with a simple technical lemma, whose proof is standard and thus omitted.

\begin{lemma}
  \label{lem:appli4}
  If there exists $\delta>0$ such that $\int_{(0,1)} u^{-\delta}\,p(\mathrm du)<+\infty$ and constants $a_0,a_1$ such that
  \begin{align*}
    a_0 < -\int_{(0,1)}\ln u\,p(\mathrm du)  < a_1,
  \end{align*}
  then there exists $\varepsilon_0>0$ such that, for all $\varepsilon\in(0,\varepsilon_0)$,
  \begin{align*}
    \varepsilon+\int_{(0,1)}(u^{\varepsilon/a_0}-1)\,p(\mathrm du)<0 \text{ and } \int_{(0,1)}u^{-\varepsilon/a_1}\,p(\mathrm du)<\varepsilon+p((0,1)).
  \end{align*}
\end{lemma}

\begin{proof}[Proof of Proposition~\ref{prop:appli4}]
  Let $\xi(x)=1$ for all $x\in (0,+\infty)$ and set
  \begin{align*}
    \psi(x)=h(x)=\exp(-\alpha s(x))\1_{x<1}+\exp(\beta s(x))\1_{x\geq 1},
  \end{align*}
  where $\alpha>0$ and $\beta>0$ are (small enough) constant which will be chosen later. We already observe that, by assumption, $\psi(x)/\xi(x)\to+\infty$ when $x\to 0$ and when $x\to +\infty$. In addition,
  \begin{align}
    \label{eq:appli4Apsiprime}
    \frac{\cA \xi(x)}{\xi(x)}=K(x)\left(\int_{(0,1)} p(\mathrm du)-1\right)\geq \inf K\,\int_{(0,1)} (1-u)\, p(\mathrm du).
  \end{align}

  For all $x<1$, we have
  \begin{align*}
    \int_{(0,x)} \frac{h(y)}{h(x)}\,k(x,\mathrm dy)=K(x)\int_{(0,1)} \exp\left(-\alpha(s(ux)-s(x))\right)\,p(\mathrm du),
  \end{align*}
  where
  \begin{align*}
    \exp\left(-\alpha(s(ux)-s(x))\right)\leq u^{-\alpha \sup_{y\in (0,x)} \frac{y}{c(y)}}.
  \end{align*}
  On the one hand, choosing $\alpha<\delta/\sup_{y\in (0,1)} \frac{y}{c(y)}$, we deduce that
  \begin{align}
    \label{eq:appli4assumption1a}
    \int_{(0,x)} \frac{h(y)}{h(x)}\,k(x,\mathrm dy)\leq \sup_{y\in(0,1)} K(y) \int_{(0,1)} u^{-\delta} \,p(\mathrm du),
  \end{align}
  and, on the other hand, letting $x\to 0$ and using Fatou's lemma, we deduce that
  \begin{align*}
    \limsup_{x\to 0} \frac{\cA h(x)}{h(x)}&=-\alpha +\limsup_{x \to 0} \int_{(0,x)} \frac{h(y)}{h(x)}\,k(x,\mathrm dy)-K(x)\notag\\
    &\leq -\alpha + \limsup_{x\to 0} K(x)\left(\int_{(0,1)} u^{-\alpha \limsup_{x\to 0} \frac{y}{c(y)} }\,p(\mathrm du)-1\right)\\
    &=-\alpha + \inf K\,\left(\int_{(0,1)} u^{-\alpha \limsup_{x\to 0} \frac{y}{c(y)} }\,p(\mathrm du)-1\right).
  \end{align*}
  According to Lemma~\ref{lem:appli4} and the second inequality in~\eqref{as:appli4as1bis}, there exists $\alpha_0>0$ such that, for all $\alpha<\alpha_0$, 
  \begin{align*}
    \int_{(0,1)} u^{-\alpha \limsup_{x\to 0} \frac{y}{c(y)} }\,p(\mathrm du)<\alpha+p((0,1)).
  \end{align*}
  This implies, choosing $\alpha <\alpha_0\wedge (\delta/\sup_{y\in (0,1)} \frac{y}{c(y)})$ (which we will assume from now on) and using in addition~\eqref{eq:appli4Apsiprime}, that
  \begin{align}
    \label{eq:appli4lyap0}
    \limsup_{x\to 0} \frac{\cA h(x)}{h(x)} < -\alpha (1-\inf K)+\inf K(p(0,1)-1)<\inf_x\frac{\cA \xi(x)}{\xi(x)}.
  \end{align}

  For all $x\geq 1$, we have
  \begin{multline*}
    \int_{(0,x)} \frac{h(y)}{h(x)}\,k(x,\mathrm dy)=K(x)\int_{(0,1/x)} \exp\left(-\alpha s(ux)-\beta s(x)\right)\,p(\mathrm du)\\
    +K(x) \int_{(1/x,1)} \exp\left(\beta(s(ux)-s(x))\right)\,p(\mathrm du)-K(x),
  \end{multline*}
  where 
  \begin{align*}
    \exp(-\alpha s(ux)-\beta s(x))\leq (ux)^{-\alpha \sup_{y\in (0,ux)} \frac{y}{c(y)}} \leq u^{-\delta}
  \end{align*}
  and
  \begin{align*}
    \exp\left(\beta(s(ux)-s(x))\right)\leq u^{\beta \inf_{y\in (ux,x)}  \frac{y}{c(y)}}.
  \end{align*}
  On the one hand, we deduce that
  \begin{align}
    \label{eq:appli4assumption1b}
    \int_{(0,x)} \frac{h(y)}{h(x)}\,k(x,\mathrm dy)\leq \sup_{y\in(0,M)} K(y) \int_{(0,1)} u^{-\delta} \,p(\mathrm du)
  \end{align}
  and, on the other hand, choosing $\beta<1/\inf_{y\geq 1} \frac{y}{c(y)}$, letting $x\to +\infty$ and using Fatou's Lemma, we deduce that
  \begin{align*}
    \limsup_{x\to +\infty} \frac{\cA h(x)}{h(x)}&=\beta+\limsup_{x\to +\infty} \int_{(0,x)} \frac{h(y)}{h(x)}\,k(x,\mathrm dy)-K(x)\\
    &\leq \beta +\limsup_{x\to+\infty} K(x)\left(\int_{(0,1)} u^{\beta \liminf_{y\to+\infty} \frac{y}{c(y)}}\,p(\mathrm du)-1\right)\\
    &=\beta +\inf K\left(\int_{(0,1)} u^{\beta \liminf_{y\to+\infty} \frac{y}{c(y)}}\,p(\mathrm du)-1\right).
  \end{align*}
  According to Lemma~\ref{lem:appli4} and the first inequality in~\eqref{as:appli4as1bis}, there exists $\beta_0>0$ such that, for all $\beta<\beta_0$, 
  \begin{align*}
    \int_{(0,1)} u^{\beta \liminf_{y\to +\infty} \frac{y}{c(y)} }\,p(\mathrm du)<-\beta+p((0,1)).
  \end{align*}
  Choosing $\beta < \beta_0\wedge(1/\inf_{y\geq 1} \frac{y}{c(y)})$, we deduce that
  \begin{align*}
    \limsup_{x\to+\infty}\frac{\cA h(x)}{h(x)}\leq \beta(1-\inf K) + \inf K (p(0,1)-1)
  \end{align*}
  and hence, choosing $\beta$ small enough and using~\eqref{eq:appli4Apsiprime},
  \begin{align}
    \label{eq:appli4lyapinfty}
    \limsup_{x\to+\infty}\frac{\cA h(x)}{h(x)}<\inf_x\frac{\cA \xi(x)}{\xi(x)}.
  \end{align}
  By~\eqref{eq:appli4assumption1a} and~\eqref{eq:appli4assumption1b}, and observing that our assumptions imply that $K$ is uniformly bounded, we deduce that the first part of Assumption~\ref{assumption1} holds true. In addition, $\cA h/h$ is locally bounded and, by~\eqref{eq:appli4lyap0} and~\eqref{eq:appli4lyapinfty}, it is thus bounded from above. We conclude that Assumption~\ref{assumption1} is verified.

  Finally,~\eqref{eq:appli4lyap0} and~\eqref{eq:appli4lyapinfty} in combination with Proposition~\ref{prop:lambda0upperbound} entail that Assumption~4 holds true. This concludes the proof of Proposition~\ref{prop:appli4}.
\end{proof}

\subsection{Proof of Proposition~\ref{prop:Markov}}
\label{sec:proofMarkov}

Since the process $X$ is a PDMP, it is a strong Markov process with respect to its completed natural filtration according to Theorem~25.5 in~\cite{Davis} (its proof remains correct under our assumptions).

Let us now prove the irreduciblity of $X$.  Fix $x_0\in(0,+\infty)$ and set
\[
  A:=\{x\in(0,+\infty),\ \P_x(H_{x_0}<+\infty)>0\}.
\]
We first note that $A$ is non-empty since $x_0\in A$.  Our strategy is to prove that $A$ is open and closed in $(0,+\infty)$, so that $A=(0,+\infty)$ since $(0,+\infty)$ is connected.

\textbf{(1) $A\cap(0,x_0)$ is open.} For all $x<x_0\in(0,+\infty)$,  $m_{x}:=\sup_{z\in [x,x_0]} k_h(z,(0,z))$ is finite according to Assumption~\ref{assumption1}. Setting $t_{x}=s(x_0)-s(x)$, we deduce from the construction of the process  (see Step~1 in the proof of Proposition~\ref{tNirr}) that
\begin{align*}
  \P_x(H_{x_0}\leq t_{x})
  &\geq \P_x(\text{the process $X$ does not jump during the time interval $[0,t_{x}]$}) \\
  & \geq e^{-m_{x} t_{x}}>0.
\end{align*}
In particular, $(0,x_0)\subset A$ so that $A\cap(0,x_0)$ is open.

\textbf{(2) $A$ contains a neighbourhood of $x_0$.} According to the previous step, for all $\varepsilon\in(0,x_0)$, $(x_0-\varepsilon,x_0]\subset A$. It remains to prove that there exists $\varepsilon>0$ such that $(x_0,x_0+\varepsilon)\subset A$. According to Assumption~\ref{assumptionIrr}, the Lebesgue measure of $s(\{y\in(x_0,+\infty),\ k(y,(0,x_0))>0\})$ is positive. Since
\[
  \{y\in(x_0,+\infty),\ k(y,(0,x_0))>0\}=\bigcup_{n\geq 1,\ m\geq 1} \{y\in(x_0,n),\ k(y,(0,x_0))>1/m\},
\]
we deduce that there exists a bounded $I_0\subset (x_0,+\infty)$ such that
\[
  \lambda_1(s(I_0))>0\text{ and }\inf_{y\in I_0} k(y,(0,x_0))>0.
\]
Choosing $\varepsilon>0$ small enough, we deduce that, for all $x\in(x_0,x_0+\varepsilon)$, $\lambda_1(s(I_0\cap (x,+\infty)))>0$.

We also have, denoting by $\sigma$ the first jump time of $X$ and using the strong Markov property at time $\sigma$,
\begin{align}
  \label{eq:useful1}
  \P_x(H_{x_0}<+\infty)\geq \E_x\left(\1_{\sigma<+\infty}\P_{X_\sigma}(H_{x_0}<+\infty)\right).
\end{align}
Since $\P_y(H_{x_0}<+\infty)>0$ for all $y\in(0,x_0)$, it is sufficient to prove that $\P(\sigma<+\infty\text{ and }X_\sigma\in (0,x_0))>0$ to conclude that $\P_x(H_{x_0}<+\infty)>0$. By construction of the process $X$, we have
\begin{align}
  \P_x(\sigma<+\infty\text{ and }X_\sigma\in (0,x_0))&\geq \P_x(\sigma<+\infty\text{ and }X_{\sigma-}\in I_0\text{ and }X_\sigma\in (0,x_0))\nonumber\\
  &\geq  \P_x(s^{-1}(s(x)+\sigma)\in I_0)\,\frac{\inf_{y\in I_0} k_h(y,(0,x_0))}{\sup_{y\in I_0} k_h(y,(0,y))+q(y)}  \label{eq:useful2}
\end{align}
since $s^{-1}(s(x)+t)$ is the position of the process $X_{t-}$ under $\P_x$, conditionally to $t\leq\sigma$. We also have
\begin{multline}
  \label{eq:useful3}
  \P_x(s^{-1}(s(x)+\sigma)\in I_0)=\P_x(\sigma\in s(I_0)-s(x))\\
  = \int_{s(I_0)-s(x)}\frac{1}{[k_h+q](s^{-1}(s(x)+t))} \exp\left(-\int_0^t [k_h+q](s^{-1}(s(x)+u))\,\mathrm du\right) \mathrm dt>0.
\end{multline}
Using~\eqref{eq:useful1}, \eqref{eq:useful2} and~\eqref{eq:useful3}, we deduce that, for all $x\in (x_0,x_0+\varepsilon)$,
\[
  \P_x(H_{x_0}<+\infty)>0.
\]
This concludes the second step of the proof.

\textbf{(3) $A\cap(x_0,+\infty)$ is open.} Fix $x\in A\cap(x_0,+\infty)$. Then, for all $\varepsilon\in(0,x)$, for all $y\in (x-\varepsilon,x)$, we have using the strong Markov property at time $H_x$, 
\[
  \P_y(H_{x_0}<+\infty)\geq \P_y(H_x<+\infty)\P_x(H_{x_0}<+\infty)>0,
\]
since $y<x$ and $x\in A$. In particular, $(x-\varepsilon,x)\subset A$. Moreover, since $X$ is right-continuous,
\[
  \lim_{y\to x,y>x}\P_x(H_y<H_{x_0})=1.
\]
Hence there exists $\varepsilon>0$ such that, for all $y\in (x,x+\varepsilon)$, 
\[
  \P_x(H_y<H_{x_0})\geq 1-\P_x(H_{x_0}<+\infty)/2
\]
This implies that
\[
  \P_x(H_y<H_{x_0}\text{ and }H_{x_0}<+\infty)>0.
\]
Since, by the strong Markov property applied at time $H_y$, we have $\P_x(H_y<H_{x_0}\text{ and }H_{x_0}<+\infty)=\P_x(H_y<H_{x_0})\P_y(H_{x_0}<+\infty)$, we deduce that, for all $y\in (x,x+\varepsilon)$, 
\[
  \P_y(H_{x_0}<+\infty)>0.
\]
This concludes the third step of the proof.

\textbf{(4) $A$ is closed in $(0,+\infty)$.} We prove that $A$ is sequentially closed in $(0,+\infty)$. Let $(x_n)_{n\in\N}\in A^{\Z_+}$ be a sequence converging to a point $x\in (0,+\infty)$. 

If there exists $n\in\Z_+$ such that $x\leq x_n$, then $\P_x(H_{x_n}<+\infty)>0$ and hence, using the Markov property at time $H_{x_n}$, we deduce that $\P_x(H_{x_0}<+\infty)>0$ and hence that $x\in A$.  

Assume now that $x_n<x$ for all $n\in\Z_+$. Without loss of generality, we assume that $(x_n)_{n\in\Z_+}$ is non-decreasing. According to Assumption~\ref{assumptionIrr}, the Lebesgue measure of $s(\{y\in (x,+\infty),\ k_h(y,(0,x))>0\})$ is positive. Since
\[
  \{y\in (x,+\infty),\ k_h(y,(0,x))>0\}=\bigcup_{n\geq 1,\ m\geq 1,p\geq 1} \{y\in(x,p),\ k_h(y,(0,x_n))>1/m\},
\]
we deduce that there exists a bounded $I_1\subset(x,+\infty)$ and $n\in\Z_+$ such that
\[
  \lambda_1(s(I_1))\text{ and }\inf_{y\in I_1} k_h(y,(0,x_n))>0.
\]
Using the same procedure as in Step~2 above, we deduce that $\P_x(H_{x_n}<+\infty)>0$. Using the strong Markov property at time $H_{x_n}$ and the fact that $x_n\in A$, we deduce that $\P_x(H_{x_0}<+\infty)>0$, so that $x\in A$.

\textbf{(5) Conclusion.} Steps~1, 2, 3 and 4 above imply that $A$ is both open and closed in the connected set $(0,+\infty)$, so that $A=(0,+\infty)$ and, for all $x,y\in(0,+\infty)$
\[
  \P_x(H_y<+\infty)>0.
\]
Now let $l<r\in(0,+\infty)$ and set $t_{l,r}=s(r)-s(l)$ . Then, for all $x\leq y\in [l,r]$,
\[
  \P_x(H_y< t_{l,r})\geq \P_l(\sigma\geq t_{l,r})>0.
\]
Moreover, since $\P_r(H_l<+\infty)>0$, we deduce that there exists $t'_{l,r}>0$ such that $\P_r(H_l<t'_{l,r})>0$. Using the strong Markov property, we deduce that, for all $x>y\in[l,r]$,
\begin{align*}
  \P_x(H_y<t_{l,r}+t'_{l,r}+t_{l,r})&\geq \P_x(H_r<t_{l,r})\P_r(H_l<t'_{l,r})\P_l(H_y<t_{l,r})\\
  &\geq \P_l(\sigma\geq t_{l,r})\,\P_r(H_l<t'_{l,r})\,\P_l(\sigma\geq t_{l,r})>0.
\end{align*}
Setting $t_0=t_{l,r}+t'_{l,r}+t_{l,r}$, this concludes the proof of Proposition~\ref{prop:Markov}.

\subsection{Proof of Proposition~\ref{prop:doeblin} under Assumption~\ref{assumptionDoebBoth}}
\label{sec:proofprop:doeblineTer}

This proof is a direct adaptation of the proof of Proposition~1 in~\cite{PichorRudnicki2000}, where the problem is already solved when $s$ is of the form $\int_1^x1/c(y)\mathrm dy$, with $c$ continuous and positive, and the measure $\mu$ is a Dirac measure.

Let $I=(\fa,\fb)$ and fix $t_1>0$ small enough so that $\phi(\fa,t)\in I$ for
all $t\in (0,t_1)$. Let us denote by 
$r(x)=k_h(x,(0,x))+ q(x) = b+K(x)- \frac{1}{h(x)}\frac{\d h}{\d s}(x)$
the jump rate of $X$ at position $x\in(0,+\infty)$; recall the definition
of $b$ in \eqref{e:b}.
Restricting to the event where the process jumps only one
time in the time interval $(0,t_1)$, we deduce that, for all $t\in(0,t_1)$ and
all positive measurable function $f$ vanishing on the cemetery point
$\partial$,
\begin{align*}
    Q_t f(\fa)
    &\geq \int_0^t \int_{[0,1]} f(\phi(T(\theta,\phi(\fa,u)),t-u))\,a\,e^{-\int_0^u r(\phi(\fa,u))\,\mathrm dv}\,e^{-\int_0^u r(\phi(T(\phi(\fa,u)),t-u))\,\mathrm dv}
    \\
    & \hphantom{{} \geq \int_0^t \int_{[0,1]}}
    \ \frac{h(T(\theta,\phi(\fa,u)))}{h(\phi(\fa,u))}\mu(\mathrm d\theta)\, \mathrm du\\
    & =\int_{[0,1]}\int_0^t  f(\phi(T(\theta,\phi(\fa,u)),t-u))\,a\,e^{-\int_0^u r(\phi(\fa,u))\,\mathrm dv}\,e^{-\int_0^u r(\phi(T(\phi(\fa,u)),t-u))\,\mathrm dv}
    \\
    & \hphantom{{} =\int_{[0,1]}\int_0^t}
    \ \frac{h(T(\theta,\phi(\fa,u)))}{h(\phi(\fa,u))} \mathrm du \,\mu(\mathrm d\theta).
\end{align*}
By assumption, $h$ is upper bounded on $I$ and $r$ is uniformly bounded away from $\infty$ on compact subsets of $(0,+\infty)$, so that there exists a constant $a_1>0$ such that
\begin{align*}
    Q_t f(\fa)&\geq a_1\,\int_{[0,1]}\int_0^t  f(\phi(T(\theta,\phi(\fa,u)),t-u))\,h(T(\theta,\phi(\fa,u)))\, \mathrm du \,\mu(\mathrm d\theta)\\
    &\geq a_1\,\int_{[0,1]}\int_0^t f\circ s^{-1}(s\circ\phi(T(\theta,\phi(\fa,u)),t-u))\,\mathrm du\,m_\theta\mu(\mathrm d\theta),
\end{align*}
where $m_\theta=\inf_{u\in[0,t_1]} h(T(\theta,\phi(\fa,u)))>0$, where we used the fact that, for any fixed $\theta$, $u\mapsto h(T(\theta,\phi(\fa,u)))$ is positive and continuous.
We observe that, for all $\theta\in[0,1]$,  $s\circ\phi(T(\theta,\phi(\fa,u)),t-u)=s(T(\theta,\phi(\fa,u)))+t-u$, so that (recall~\eqref{eq:starrevis}) for all $\theta\in[0,1]$, 
\begin{align*}
    \frac{\mathrm d s\circ\phi(T(\theta,\phi(\fa,u)),t-u)}{\mathrm du}=\frac{\mathrm d s\circ T(\theta,\phi(\fa,u))}{\mathrm d u}-1=\frac{\partial s\circ T(\theta,\cdot)}{\partial s}(\phi(\fa,u))-1\neq 0,
\end{align*}
with the left hand side continuous in $u$ and in particular  bounded away from
$0$ and $\infty$ for $u\in(0,t_1)$. We are now in position to use the change of
variable $y= s(\phi(T(\theta,\phi(\fa,u)),t-u))$,
and deduce that, for all
$\theta\in[0,1]$, there
exists a positive constant $a_2(\theta)>0$ such that, for all $t\in(0,t_1)$,
\begin{align*}
    \int_0^t f\circ s^{-1}(s\circ\phi(T(\theta,\phi(\fa,u)),t-u))\,\mathrm du
    \geq
    a_2(\theta)\,\left|\int_{s(\phi(T(\theta,\fa),t))}^{s(T(\theta,\phi(\fa,t)))}
    f\circ s^{-1}(y)\,\mathrm dy\right|.
\end{align*}
We have $\phi(T(\theta,\fa),t_1)\neq T(\theta,\phi(\fa,t_1))$ and hence, by
continuity of both term at $t_1$, there exist
two values $s_1(\theta)<s_2(\theta)$
and a fixed time
$t'_1(\theta)\in(0,t_1)$ such that, for all $t\in(t'_1(\theta),t_1)$,
\begin{align*}
    \int_0^t f\circ s^{-1}(s\circ\phi(T(\theta,\phi(\fa,u)),t-u))\,\mathrm du
    \geq a_2(\theta)\,m_\theta\,\int_{s_1(\theta)}^{s_2(\theta)} f\circ
    s^{-1}(y)\,\mathrm dy
\end{align*}
and hence
\begin{align*}
    Q_t f(\fa)\geq
    \int_{[0,1]}\,a_2(\theta)\,m_\theta 
    \,\1_{t\in(t'_1(\theta),t_1)}\int_{s_1(\theta)}^{s_2(\theta)}
    f\circ s^{-1}(y)\,\mathrm dy\,\mu(\mathrm d\theta).
\end{align*}
Let $t'_1\in(0,t_1)$ such that
\[
  \int_{[0,1]}\,a_2(\theta)\,m_\theta
  \,\1_{t'_1\in(t'_1(\theta),t_1)} \int_{s_1(\theta)}^{s_2(\theta)}
  f\circ s^{-1}(y)\,\mathrm dy\,\mu(\mathrm d\theta)>0
\]
for all positive continuous $f$,
and define the positive measure  $\upsilon'$ on $(0,+\infty)$ by
\[
  \upsilon'(f)=\int_{[0,1]}\,a_2(\theta)\,m_\theta
  \,\1_{t'_1\in(t'_1(\theta),t_1)}\int_{s_1(\theta)}^{s_2(\theta)} 
  f\circ s^{-1}(y)\,\mathrm dy\,\mu(\mathrm d\theta)
\]
so that, for all $t\in(t'_1,t_1)$, setting $a_2=\upsilon'((0,+\infty))$ and
$\upsilon=\nicefrac{1}{a_2}\upsilon'$,
\begin{align}
    \label{eq:Step1inF0Ter}
    \E_{\fa}(f(X_t))=Q_tf(\fa)\geq a_2\upsilon(f).
\end{align}

Now, since $\upsilon$
is a non-zero measure on $(0,\infty)$, we have, by the irreducibility property proved in Proposition~\ref{prop:Markov},
\[
\P_\upsilon(H_\fa<+\infty)= \int_{(0,+\infty)} \upsilon(\mathrm dy)\P_y(H_\fa<\infty)>0,
\]
where $H_\fa = \inf\{t \geq  0: X_t = \fa\}$. In particular, there exists $t_2>0$ such that
\[
a_3:=\P_\upsilon(H_\fa\in [t_2,t_2+(t_1-t'_1)/2])>0.
\]

Hence, using the strong Markov property at time $H_\fa$, we deduce that
for all $t\in [t_2+t'_1,t_2+(t'_1+t_1)/2]$,
\begin{align}
    \E_\upsilon(f(X_t))&\geq
    \E_\upsilon\left[\1_{H_\fa\in[t_2,t_2+(t_1-t'_1)/2]}
    \bigl.\E_\fa[f(X_{t-u})]\bigr\rvert_{H_\fa}\right]\geq a_3a_2\,\upsilon(f).
    \label{e:upsilon-iter}
\end{align}
Iterating the above inequality (i.e., applying the Markov property successively
at times $tk/n$, $k=1,\dotsc,n-1$) we deduce that
\begin{align}
    \label{eqNineq}
    \E_\upsilon(f(X_t)) \ge (a_3a_2)^n \upsilon(f), \qquad
    t \in [n(t_2+t'_1),n(t_2+(t'_1+t_1)/2)].
\end{align}
We set $n_1=\left\lfloor\frac{2t_2+2t'_1}{t_1-t'_1}\right\rfloor+1$ (so that $(n+1)(t_2+t'_1)\leq n(t_2+(t'_1+t_1)/2)$ for all $n\geq n_1$), and define $t_3=n_1(t_2+t'_1)$. For any $t\geq t_3$, the integer $n=\left\lfloor\frac{t}{t_2+t'_1}\right\rfloor$ satisfies $t\in [n(t_2+t'_1),(n+1)(t_2+t'_1)]$ and $n\geq n_1$, so that $t\in [n(t_2+3t_1/2),n(t_2+(t'_1+t_1)/2)]$.
Hence, setting
\[
  \beta_t=(a_3a_2)^{\left\lfloor\frac{t}{t_2+t'_1}\right\rfloor}>0, \qquad t\geq t_3,
\]
we deduce from~\eqref{eqNineq} that 
\begin{align*}
    \E_\upsilon(f(X_t))&\geq  \beta_t \upsilon(f), \qquad t\geq t_3.
\end{align*}

Using again the irreducibility property stated in Proposition~\ref{prop:Markov}, we know that, 
for any
compactly contained interval $L\subset (0,+\infty)$
containing $\fa$,
there exists a
constant $t_4(L)>0$ such that 
\[
  a_4(L):=\inf_{x\in L} \P_x(H_\fa\leq
  t_4(L))>0.
\]
Hence Markov's property applied at time $H_\fa$ and the above inequalities gives, for $t\geq  t_1+t_3+t_4(L)$ and $x\in L$,
\begin{align*}
  \E_x\left(f(X_t)\right)
  &\geq \E_x\left[\1_{H_\fa\leq t_4(L)}\bigl.\E_\fa[f(X_{t-u})]\bigr\rvert_{u=H_\fa}\right]\\
  &\geq \E_x\left[\1_{H_\fa\leq t_4(L)}a_2\bigl.\E_\upsilon[f(X_{t-t_1-u})]\bigr\rvert_{u=H_\fa}\right]\\
  &\geq \E_x\left[\1_{H_\fa\leq t_4(L)}a_2\beta_{t-t_1-H_\fa}\right] \upsilon(f)\\
  &\geq c_{L,t}\, \upsilon(f),
\end{align*}
where $c_{L,t}:=a_4(L)\,a_2\beta_{t-t_1-t_4(L)}$. This concludes the proof of Proposition~\ref{prop:doeblin} under Assumption~\ref{assumptionDoebBoth}.

\subsection{Proof of Proposition~\ref{prop:doeblin} under Assumption~\ref{assumptionDoebF}}
\label{sec:proofprop:doeblinFeller}

\textbf{(1) We find a lower bound for $Q$ by a simpler semigroup $S$.}

Let $I=(\fa,\fb)$ such that $\beta(x,(0,x))$ is positive for all  $x\in I$ and fix $t_1>0$ small enough so that $\phi(\fa,t)\in I$ for
all $t\in (0,t_1)$. Without loss of generality, we assume that there exists a compact set $A\subset (0,\fa)$ such that $\beta(x,A^c)=0$ for all $x\in(0,+\infty)$, and that $\beta(x,\cdot)=0$ for all $x\geq 2\fb$.
Let $(S_t)_{t\geq 0}$ be the semigroup of a PDMP on $[\min A,+\infty)$, with jump kernel $\beta(x,\mathrm dy)$ and flow directed by $\phi$. Then, using the fact that $h$, $q$ and $k_h(x,(0,x))$ are locally bounded, we deduce that, for all $t\geq 0$, there exists a constant $a_1(t)>0$ such that, for all non-negative function $f$ with support in $(\min A,2\fb)$,
\begin{align*}
    Q_tf(x)\geq a_1(t) S_tf(x)\,\forall x\in (\fa,\fb).
\end{align*}

\textbf{(2) We prove a Feller-type property for $S$.}

Let $t_0>0$ be such that $\phi(\min A,t_0)>2\fb$, so that the process with semi-group $S$ has jumped at least once before time $t_0$ or remains outside $[\min A,2\fb]$. Then for all $t\geq t_0$, all measurable $B\subset [\min A,2\fb]$ and all $x\in [\min A,+\infty)$,
\begin{align*}
   S_t \1_B(x)=\int_0^t \int_{(0,\phi(x,u))} S_{t-u}\1_B(\phi(y,t-u))\,e^{-\int_0^u r_\beta(\phi(x,u))\,\mathrm dv}\beta(\phi(x,u),\mathrm dy)\, \mathrm du
\end{align*}
with $r_\beta(z):=\beta(z,(0,+\infty))$.
Using the property~\eqref{eq:semi-cont-pos-assumption} for $\beta$, one deduces that it also holds for $S_t$: for all $t\geq t_1$, for all $x_0\in [\min A,2\fb]$
\begin{align}
    \label{eq:semi-cont-pos}
    S_t \1_B(x_0) > 0
    \text{ implies that }
    \liminf_{x\to x_0, \, x<x_0} S_1\1_B(x) > 0
\end{align}

\textbf{(3) We find a recurrent set.}
To be more specific, in this part, we obtain a set $C$ of arbitrarily small radius
with the property that $S_t\1_C(x) > 0$ for all $t$ sufficiently large and $x\in C$.
We begin with a useful inequality, before making a case distinction.

Restricting to the event where the process jumps exactly once, and
does so in the time interval $[0,t_1]$, we deduce that, for 
all positive measurable function $f$ and all $x\in I$, for all $t\geq t_1$, there exists a constant $a_2(t)>0$ such that
\begin{align*}
    S_{t} f(x)&\geq  a_2(t)\,\delta_x R^{t} f,\text{ where }\delta_x R^t f:=\int_0^{t_1} \int_{(0,+\infty)} f(\phi(y,t-u))\,\beta(\phi(x,u),\mathrm dy)\, \mathrm du.
\end{align*}
We set $\fc=\frac{\fa+\fb}{2}$. The term $\delta_\fc R^{t_1}$ defines a  measure such that $\delta_\fc R^{t_1}((0,\fa))>0$, and hence there exists $t_2> t_1$ such that $\delta_\fc R^{t_2}((\fc-\varepsilon,\fc])>0$ for all $\varepsilon>0$.  Let $\fc'\in(\fa,\fc)$ be such that $\phi(\fc',t_2-t_1)>\fc$. Then there exists a constant $a_3>0$ such that, for all $x\in (\fc',\fc)$ and all positive function $f:(0,+\infty)\to[0,+\infty)$,
 \begin{align}
     \label{eq:minoration}
     \delta_x R^{t_2}f\geq a_3 \delta_\fc R^{t_2-t(x)}f,
 \end{align}
where $t(x)\in(0,t_2-t_1)$ is such that $\phi(x,t(x))=\fc$.
In particular, for all $x\in (\fc',\fc)$, noting that $\phi(\fc',t(x))<\phi(x,t(x))=\fc$, we have
\begin{align}
    \delta_x S_{t_2}((\fc',\fc])&\geq a_2(t_2)  \delta_x R^{t_2}((\fc',\fc])\nonumber\\
                                & \geq a_2(t_2) a_3 \delta_\fc R^{t_2-t(x)}((\fc',\fc])\nonumber\\
                                &= a_2(t_2) a_3 \delta_\fc R^{t_2}((\phi(\fc',t(x)),\phi(\fc,t(x))])\nonumber\\
                                &\geq a_2(t_2) a_3 \delta_\fc R^{t_2}((\phi(\fc',t(x)),\fc])>0,\label{eq:gooduse}
\end{align}
where (here and later) we define the measure $\delta_x S_t(A) = S_t\1_A(x)$.

\textbf{Case (a): $\delta_\fc S_{t_2}(\{\fc\})>0$.} When this is true, we can prove
Proposition~\ref{prop:doeblin} in a straightforward way.
If this holds then, using~\eqref{eq:semi-cont-pos}, we deduce that there exists $\fc''\in(\fc',\fc)$ such that $\inf_{x\in (\fc'',\fc]} \delta_x S_{t_2}(\{\fc\})>0$. Fixing $t_3\in(t_1,t_2)$ such that
\[
\inf_{t\in[t_3,t_2]} \delta_\fc S_{t}((\fc'',\fc])>0,
\]
we deduce that there exists a constant $a_4>0$ such that, for all $t\in [t_3+t_2,2t_2]$,
\[
\delta_\fc S_{t}\geq a_4 \delta_{\fc}.
\]
We can now follow the same strategy as in the previous section, since the equation
above is essentially \eqref{e:upsilon-iter} with $\upsilon = \delta_\fc$.
This allows us deduce that $Q$ satisfies~\eqref{eq:doeblininprop}, which completes the proof of Proposition~\ref{prop:doeblin}
in this case.

\textbf{Case (b): $\delta_\fc S_{t_2}(\{c\})=0$.} 
Take  $\fc''\in(\fc',\fc)$ such that $\phi(\fc',t(\fc''))= \fc''$. Then there exists $\varepsilon>0$ such that  $\delta_\fc R^{t_2}((\fc'',\fc-\varepsilon))>0$.
According to~\eqref{eq:gooduse}, we have
\begin{align*}
    \delta_x R^{t_2}((\fc',\fc-\varepsilon))&\geq a_3 \delta_\fc R^{t_2-t(x)}((\fc',\fc-\varepsilon))\\
    &= a_3 \delta_\fc R^{t_2}((\phi(\fc',t(x)),\phi(\fc-\varepsilon,t(x))]).
\end{align*}
If $x\in(\fc',\fc'']$, then $\phi(\fc',t(x))< \fc$ and $\phi(\fc-\varepsilon,t(x))> \phi(\fc'',t(x)) \geq \fc$ so that
\begin{align*}
\delta_x R^{t_2}((\fc',\fc-\varepsilon))&\geq  a_3 \delta_\fc R^{t_2}((\phi(\fc',t(x)),\fc))>0.
\end{align*}
If $x\in( \fc'',\fc)$, then $\phi(\fc',t(x))\leq \phi(\fc',t(\fc''))= \fc''$ and $\phi(\fc-\varepsilon,t(x))>\fc-\varepsilon$, hence  
\begin{equation*}
    \delta_x R^{t_2}((\fc',\fc-\varepsilon))
    \geq  a_3 \delta_\fc R^{t_2}((\fc'',\fc-\varepsilon))>0.
\end{equation*}
 Choosing $t_4>t_2$ such that $\phi(\fc-\varepsilon,t_4-t_2)=\fc$, we deduce that 
\begin{align*}
    \delta_x S_{t}((\fc',\fc])\geq a_2(t) \delta_x R^{t}((\fc',\fc])\geq a_2(t) \delta_x R^{t_2}((\fc',\fc-\varepsilon])>0,\ \forall x\in (\fc',\fc],\ \forall t\in[t_2,t_4].
\end{align*}
Hence there exists $t_5>0$ such that, for all $t\geq t_5$, 
\begin{align*}
   S_t ((\fc',\fc])(x)>0,\ \forall x\in (\fc',\fc].
\end{align*}
Since $\fc'$ can be replaced in this argument by any point arbitrarily close to $\fc$ and since $S$ is irreducible on $(\fc',\fc)$, we deduce that, for all $\varepsilon>0$, there exists $t(\varepsilon)$ such that, for all $t\geq t(\varepsilon)$, 
\begin{align}
    \label{eq:open-irr}
    S_t ((\fc-\varepsilon,\fc])(x)>0,\ \forall x\in (\fc',\fc].
\end{align}

\textbf{(4) Conclusion.} In this part, we give an adaptation of
\cite[Proposition~6.2.1]{MT-book} in order to conclude.

Define the measures 
$\delta_x\bar S_n \coloneqq \1_{(\fc',\fc]}(x)\delta_x S_{nt_1}(\cdot \cap (\fc',\fc])$ for $x\in(0,+\infty)$ and $n\in\mathbb{N}$.
For all measurable $B\subset (\fc',\fc]$ such that $\delta_\fc\bar S_1(B)>0$, we deduce from~\eqref{eq:semi-cont-pos} that there exists $\varepsilon>0$  such that 
$\inf_{x\in (\fc-\varepsilon,\fc]} \delta_x\bar S_1(B)>0$ and hence from~\eqref{eq:open-irr} that, for any $n_0\geq 1$ such that $n_0 t_1\geq t(\varepsilon)$, 
\[
\delta_x \bar S_{n_0+1} (B)>0.
\]
This shows that
$(\bar S_n)_{n\in \mathbb N}$ is $\psi$-irreducible with $\psi=\delta_{\fc}\bar S_1$.
In particular, by~\cite[Theorem~5.2.2]{MT-book} and its proof, $(\bar S_n)_{n\in\mathbb N}$ admits a `small set' $C\subset (\fc',\fc]$ such that $\delta_{\fc}\bar S_1(C)>0$; this means that there exists $m\geq 1$, a constant $\eta>0$ and a probability measure $\upsilon$ on $(\fc',\fc]$ such that
$\delta_x \bar S_m\geq \eta \upsilon$ for all $x\in C$. Since $\delta_{\fc}\bar S_1(C)>0$, we deduce from~\eqref{eq:semi-cont-pos} that there exists a neighborhood $U$ of $\fc$ in $(\fc,',\fc]$ such that $\inf_{x\in U}\delta_x \bar S_1(C)>0$. In particular,
\[
\delta_x Q_{(m+1)t_1}\geq a_1((m+1)t)\delta_x S_{(m+1)t_1}\geq a_1((m+1)t)\delta_x \bar S_{m+1} \geq \eta\,\inf_{x\in U}\delta_x \bar S_1(C)\, \upsilon,\ \forall x\in U.
\]
Since there exists $t_6>t_7$ such that $\inf_{t\in[t_6,t_7]} \delta_\fc Q_t U>0$, we deduce as above that~\eqref{eq:doeblininprop} holds true. This concludes the proof of Proposition~\ref{prop:doeblin}.

\subsection{Proof of Theorem~\ref{thm:spectgap}}
\label{sec:proofspectgap}

Our aim is to prove that Assumptions~\ref{assumption1},~\ref{assumptionIrr},~\ref{assumptionDoebSec} and~\ref{assumptionLyapGeneral} 
together imply that
Assumption~F of~\cite{CV-qsd} is satisfied for the Markov semigroup $(Q_t)_{t\in[0,+\infty)}$. Let us recall this assumption.

\medskip\noindent\textbf{Assumption (F).} There exist positive real constants $\gamma_1,\gamma_2,c_1,c_2$ and $c_3$, $t_1,t_2\in [0,+\infty)$,
a measurable function $\psi_1:(0,+\infty)\rightarrow [1,+\infty)$, and a probability measure $\nu$ on a measurable subset $L\subset (0,+\infty)$ such
that
\begin{itemize}
  \item[(F0)] \textit{(A strong Markov property).} 
    Defining
    \begin{align}
      \label{eq:def-T_L}
      H_L:=\inf\{t\geq 0,\ X_t\in L\},
    \end{align}
    assume that for all $x\in (0,+\infty)$, $X_{H_L}\in L$, $\PP_x$-almost surely on the event $\{H_L<\infty\}$ and for all $t>0$ and
    all measurable $f:(0,+\infty)\cup\{\d\}\rightarrow\RR_+$,
    \begin{align*}
      \EE_x\left[f(X_t)\1_{H_L\leq t<\zeta}\right]=\EE_x\left[\1_{H_L\leq
      t\wedge\zeta}\EE_{X_{H_L}}\left[f(X_{t-u})\1_{t-u<\zeta}\right]\rvert_{ u=H_L}\right].
    \end{align*}
  \item[(F1)] \textit{(Local Dobrushin coefficient).} $\forall x\in L$, 
    \label{i:F1}
    \begin{align*}
      \P_x(X_{t_1}\in\cdot)\geq c_1 \nu(\cdot\cap L).
    \end{align*}
  \item[(F2)] \textit{(Global Lyapunov criterion).} We have $\gamma_1<\gamma_2$ and
    \begin{align*}
      &\E_x(\psi_1(X_{t_2})\1_{t_2<H_L\wedge\zeta})\leq \gamma_1^{t_2}\psi_1(x),\ \forall x\in (0,+\infty)\\
      &\E_x(\psi_1(X_t)\1_{t<\zeta})\leq c_2,\ \forall x\in L,\ \forall t\in[0,t_2],\\
      &\gamma_2^{-t}\P_x(X_t\in L)\xrightarrow[t\rightarrow+\infty]{} +\infty,\ \forall x\in L.
    \end{align*}
  \item[(F3)] \textit{(Local Harnack inequality).} We have
    \begin{align*}
      \sup_{t\geq 0}\frac{\sup_{y\in L} \P_y(t<\zeta)}{\inf_{y\in L} \P_y(t<\zeta)}\leq c_3
    \end{align*}
\end{itemize}

We prove in the following subsections that F0, F1, F2 and~F3
are satisfied, in this order, with the aim to apply the following result, which is Theorem~3.5 in~\cite{CV-qsd} combined with the continuous time adaptation of Theorem~1.7 in~\cite{CV-qsd}.

\begin{theorem}[\cite{CV-qsd}]
  \label{thm:E-F}
  Under Assumption (F), $(X_t)_{t\in[0,+\infty)}$ admits a quasi-stationary distribution $\nu_{\text{QS}}$ on $(0,+\infty)$, which is the unique one satisfying
  $\nu_{\text{QS}}(\psi_1)<\infty$ and $\PP_{\nu_{\text{QS}}}(X_{t}\in L)>0$ for some $t\in [0,+\infty)$. 
In addition, there exists a
  constant $\lambda_0^X\geq 0$ such that $\lambda_0^X\leq\log(1/\gamma_2)<\log(1/\gamma_1)$ and $\PP_{\nu_{\text{QS}}}(t<\zeta)=e^{-\lambda^X_0 t}$
  for all $t\geq 0$, and there exists a function $\eta:(0,+\infty)\to[0,+\infty)$ lower bounded away from $0$ on $L$ and  such that
  \begin{equation}
  \label{eq:eta-prop-E-F}
  \left|\eta(x)-e^{\lambda^X_0 t}\P_x(t<\zeta)\right|\leq C e^{-\gamma t}\,\psi_1(x),\quad\forall x\in (0,+\infty)
  \end{equation}
  and such that
  $\E_x(\eta(X_t)\1_{t<\zeta})=e^{-\lambda_0^X t}\eta(x)$ for all $x\in (0,+\infty)$ and $t\geq 0$. 
  Finally, setting $E'=\{x\in(0,+\infty),\ \eta(x)>0\}$,  we have, for all $f:E'\to\R$ such that $\|f\eta/\psi_1\|_\infty<+\infty$,
\begin{equation}
\label{eq:Q-proc}
\left|\frac{e^{\lambda^X_0 t}}{\eta(x)}\E_x(\eta(X_t)f(X_t)\1_{t<\zeta})-\nu_{QS}(\eta f)\right|\leq C e^{-\gamma t}\,\frac{\psi_1(x)}{\eta(x)}\,\|f\eta/\psi_1\|_\infty,\quad\forall x\in E',
\end{equation}
  for some constants $\gamma>0$
and $C>0$.
\end{theorem}

Note that, in the above result, it is clear that $\lambda_0^X$ is the same as the one defined in~\eqref{eq:lambda0X}. We conclude by proving that the property obtained from this result entails Theorem~\ref{thm:spectgap}.

In
what follows, we only consider functions
$f$ vanishing on the cemetery point, so that
$Q_t f(x)=\E_x(f(X_t)\1_{t<\zeta})=\E_x(f(X_t))$ for all $x\in E$, where $\zeta$ is the first hitting time of $\d$. Moreover, $b$ and $h$ are the objects defined in section~\ref{sec:markov}.

\subsubsection{Proof of F0 and F1}

The completed natural filtration of $X$ is right continuous (see Theorem~25.3 in~\cite{Davis}). Hence the Début Theorem (see for instance Lemma~75.1 in~\cite{RW1}) implies that $H_L$ is a stopping time with respect to this filtration. By Proposition~\ref{prop:Markov}, we deduce that F0 holds true for any compact interval $L\subset (0,+\infty)$ (this set shall be chosen in subsection~\ref{sec:F2}).

According to Proposition~\ref{prop:doeblin}, the condition F1 holds true for $L$, assuming in addition (and without loss of generality) that $L$ large enough so that $\upsilon(L)\geq 1/2$.

\subsubsection{Proof of F2}
\label{sec:F2}

Take $\psi_1=\psi/h$ (we assume without loss of generality that $\psi\geq h$), extended to $\d$ by the value $0$. We deduce from Assumption~\ref{assumptionLyapGeneral} that there exists $\lambda^X_1>\lambda_0^X$ and a compact interval $L\subset (0,+\infty)$ such that 
\[
  \cL \psi_1(x) \leq  -\lambda^X_1\psi_1(x)+C\1_L(x),\quad\forall x\in(0,+\infty).
\]
Let $(f_k)_{k\geq 2}$ be a non-decreasing sequence of non-negative functions in $C_c^{(s)}$ such that, for all $k \geq 2$, 
$f_k(x)=\psi_1(x)$ for all $x\in(1/k,k)$. We deduce that, for all $x\in (1/k,k)$,
\begin{align*}
  \cL f_k(x)&=\frac{\d f}{\d s}(x)+k_h(x,f_k)-f_k(x)k_h(x,(0,x))-q(x)f_k(x)\\
  &= \frac{\d \psi_1}{\d s}(x)+k_h(x,f_k)-\psi_1(x)k_h(x,(0,x))-q(x)\psi_1(x)\\
  &\leq \cL\psi_1(x)\leq -\lambda^X_1\psi_1(x)+C\1_L(x)= -\lambda^X_1 f_k(x)+C\1_L(x).
\end{align*}
Since $f_k$, extended by the value $0$ on $\d$, belongs to the domain of the extended infinitesimal generator of $X$, we deduce that
\[
  M_t^k:= e^{\lambda_1^X t}f_k(X_t)-f_k(x)-\int_0^t e^{\lambda_1^X u}(\lambda_1^X+\cL f_k(X_u))\,\mathrm du
\]
is a local martingale. Since $ e^{\lambda_1^X u}(\lambda_1^X+\cL f_k(X_u))$ is uniformly bounded on $[0,t]$ (where we used Lemma~\ref{lem:2bis} (iii)), we deduce that it is a martingale. 
In particular, for any $2\leq k'\leq k$, denoting by $\tau_{k'}=\inf\{t\geq 0, X_t\text{ or }X_{t-}\notin (1/k',k')\}$, we have, using the optional stopping theorem,
\[
  \E\left(e^{\lambda_1^X t\wedge \zeta\wedge\tau_{k'}\wedge H_L}f_k(X_{t\wedge \zeta\wedge \tau_{k'}\wedge H_L})\right)\leq f_k(x),\quad\forall x\in(1/{k'},k').
\]
Letting $k\to+\infty$, we deduce that
\[
  \E\left(e^{\lambda_1^X t\wedge \zeta\wedge \tau_{k'}\wedge H_L}\psi_1(X_{t\wedge \zeta\wedge \tau_{k'}\wedge H_L})\right)\leq \psi_1(x),\quad\forall x\in(1/{k'},k').
\]
Using Fatou's Lemma and the non-explosion of the process $X$, we conclude by letting $k'\to+\infty$ that
\[
  \E\left(e^{\lambda_1^X t\wedge \zeta\wedge H_L}\psi_1(X_{t\wedge \zeta\wedge H_L})\right)\leq \psi_1(x),\quad\forall x\in(0,+\infty).
\]
This entails that
\[
  \E\left(e^{\lambda_1^X t}\psi_1(X_{t})\1_{t<\zeta\wedge H_L}\right)\leq \psi_1(x),\quad\forall x\in(0,+\infty),
\]
which implies the first line of F2 for any $t_2>0$ and $\gamma_1=e^{-\lambda_1^X}$.

The same procedure, but replacing $\lambda_1^X$ by $-C$, stopping the process at time $t\wedge \zeta\wedge \tau_{k'}$ instead of $t\wedge \zeta\wedge \tau_{k'}\wedge H_L$ and using the fact that $\cL f_k\leq C$ for all $x\in(1/k,k)$, one deduces that, for all $t\geq 0$.
\[
  \E\left(\psi_1(X_t)\1_{t< \zeta}\right)\leq e^{Ct}\psi_1(x),\quad\forall x\in(0,+\infty).
\]
This implies the second line of F2.

Finally, choosing any $\gamma_2\in(e^{-\lambda_1^X},e^{-\lambda_0^X})$, the last line of F2 is a direct consequence of the definition of $\lambda_0^X$.

\subsubsection{Proof of F3}

The irreducibility property of Proposition~\ref{prop:Markov} implies that there
exists $t_L>0$ such that $\inf_{x,y\in L}\P_x(H_y<t_L)>0$. Moreover, for any fixed $x_0\in L$, $\P_{x_0}(t_L<\zeta)>0$, hence
\[
  c_3:= \inf_{x,y\in L}\P_x(H_y<t_L)\P_{x_0}(t_L<\zeta)>0.
\]
For
all all $t\geq t_L$ and all $x,y\in L$, we obtain, using the fact that
$\P_x(t<\zeta)$ is decreasing with respect to $t$ and the strong
Markov property at time $H_y$,
\begin{align*}
  \P_x(t<\zeta)
  &\geq \E_x\left(\1_{H_y\leq t}\P_y(t-u<\zeta)\rvert_{u=H_y}\right)\\
  &\geq \P_x(H_y\leq t_L)\P_y(t <\zeta)\geq c_3\,\P_y(t<\zeta).
\end{align*}
For $t<t_L$,  we observe that, for all $x,y\in L$, using the strong Markov property at time $H_{x_0}$,
\[
  \P_x(t<\zeta)\geq \P_x(H_{x_0}<+\infty)\P_{x_0}(t_L<H_{x_0})\geq c_3\geq c_3\P_y(t<\zeta).
\]		
Hence,
\begin{equation}\label{e:localH-1}
  \sup_{t\geq 0} \frac{\sup_{y\in L}\PP_y(t<\zeta)}{\inf_{x\in L}\PP_x(t<\zeta)} \le \frac{1}{c_3}< \infty.
\end{equation}
This concludes the proof of F3.

\subsubsection{Conclusion of the proof of Theorem~\ref{thm:spectgap}}

We proved in the above subsections that the semigroup $Q$ satisfies the
conditions of Theorem~\ref{thm:E-F}. The Doeblin property obtained in
Proposition~\ref{prop:doeblin} entails that $\eta$ is positive on $(0,+\infty)$
(and in particular $E'=(0,+\infty))$. Hence, for all 
$f\in L^\infty(\psi_1)$ and all $t\geq 0$, applying~\eqref{eq:Q-proc} to $f/\eta$, we
deduce that
\[
  \left| \frac{e^{\lambda_0^X t}}{\eta(x)}\mu Q_t f - \nu_{\text{QS}}(f)\right|\leq C\,e^{-\gamma t}\,\frac{\psi_1(x)}{\eta(x)}\,\|f/\psi_1\|_\infty,\ \forall x\in E'.
\]
Since $\delta_x Q_t f=e^{-bt}\frac{1}{h(x)} \delta_x T_t(fh)$,  we obtain, 
taking $f=g/h$ with $g\in L^\infty(\psi)=L^\infty(\psi_1 h)$,
\begin{align*}
  \left| e^{(\lambda_0^X-b) t}\delta_x T_t g - \nu_{\text{QS}}(g/h)\,\eta(x)h(x)\right|
  &\leq C\,e^{-\gamma t}\,\psi_1(x)h(x)\|g/(h\psi_1)\|_\infty \\
  & =C\,e^{-\gamma t}\,\psi(x)\,\|g/\psi\|_\infty.
\end{align*}
Finally, using that $\lambda_0=\lambda_0^X-b$ and setting $m(g):=\nu_{\text{QS}}(g/h)$ and $\varphi(x)=\eta(x)h(x)$ we deduce that, for all $g\in L^\infty(\psi)$,
\begin{align*}
  \left| e^{\lambda_0 t}\delta_x T_t g - m(g)\,\varphi(x)\right|\leq C\,e^{-\gamma t}\,\psi(x).
\end{align*}
Integrating with respect to $\mu$ such that $\mu(\psi)<+\infty$ concludes the proof.

\subsection{Proof of Proposition~\ref{prop:alterlambda}}
\label{sec:proofprop:alterlambda}

Let $\lambda_0'=\inf\{\lambda\in\mathbb R,\ \int_0^\infty e^{\lambda t} T_t \1_L(x)\,\mathrm dt=+\infty\}$, where $x\in(0,+\infty)$ is fixed and $L\subset (0,+\infty)$ is a non-empty, compactly embedded open interval.
We clearly have  $\lambda_0\geq \lambda'_0$. Let us prove the converse inequality.

Fix $\lambda>\lambda'_0$, so that $ \int_0^\infty e^{\lambda t} T_t \1_L(x)\,\mathrm dt=+\infty$ for some $x\in (0,+\infty)$ and some compactly embedded non-empty interval $L\subset (0,+\infty)$. In particular, setting $\lambda^X=\lambda+b$, we have $ \int_0^\infty e^{\lambda^X t} \P_x(X_t\in L)\,\mathrm dt=+\infty$. For any $y\in (0,+\infty)$, there exists, according to Proposition~\ref{prop:Markov}, $u_0>0$ such that $\P_y(H_x\leq u_0)>0$, and hence, using the strong Markov property at time $H_x$,
\begin{align*}
  \int_{u_0}^\infty e^{\lambda^X t} \mathbb P_y(X_t\in L)\,\mathrm dt
  &\geq  \int_{u_0}^\infty e^{\lambda^X t} \mathbb E_y\left(\1_{H_x\leq u_0}\P_x(X_{t-u}\in L)\rvert_{{u=H_x}}\right)\,\mathrm dt\\
  &= \mathbb E_y\left(\1_{H_x\leq u_0} \int_{u_0}^\infty e^{\lambda^X t} \P_x(X_{t-u}\in L)\rvert_{{u=H_x}} \,\mathrm dt\right)\\
  &\geq \mathbb E_y\left(\1_{H_x\leq u_0} \int_{u_0}^\infty e^{\lambda^X v} \P_x(X_{v}\in L) \,\mathrm dt\right)=+\infty.
\end{align*}
In particular, $ \int_0^\infty e^{\lambda^X t}\P_y(X_t\in L)\,\mathrm dt=+\infty$ for all $y\in(0,+\infty)$. This implies that the probability measure $\upsilon$ from Proposition~\ref{prop:doeblin} satisfies
\begin{align}
  \label{eq:prooprostep1}
  \int_0^\infty e^{\lambda^X t} \P_\upsilon(X_t\in L)\,\mathrm dt=+\infty.
\end{align}

Consider $t_L$, $\upsilon$ and $c_{L,t}$ from Proposition~\ref{prop:doeblin}. Then, for all $T\geq t_L+1$ and all $x\in L$, we have, applying the Markov property at time $t_L+u$ for all $u\in[0,1]$,
\begin{align*}
  \P_x(X_T\in L)\geq c_{L,t_L+u} \P_\upsilon(X_{T-t_L-u}\in L)\geq c_{L,t_L+1} \P_\upsilon(X_{T-t_L-u}\in L)
\end{align*}
and hence
\begin{align*}
  e^{\lambda^X T}\P_x(X_T\in L)&\geq c_{L,t_L+1} e^{\lambda^X T} \int_0^1 \P_\upsilon(X_{T-t_L-u}\in L)\,\mathrm du\\
  &\geq  c_{L,t_L+1}  \int_0^1 e^{\lambda^X (T-t_L-u)}\,\P_\upsilon(X_{T-t_L-u}\in L)\,\mathrm du\\
  & = c_{L,t_L+1} \int_{T-t_L-1}^{T-t_L} e^{\lambda^X t}\,\P_\upsilon(X_{t}\in L)\,\mathrm dt.
\end{align*}
Now, according to~\eqref{eq:prooprostep1}, for any fixed $\varepsilon>0$, there exists $T_\varepsilon\in\{0,1,\ldots\}$ such that $T_\varepsilon\geq t_L+1$ and
\[
  \int_{T_\varepsilon-t_L-1}^{T_\varepsilon-t_L} e^{(\lambda^X+\varepsilon) t}\,\P_\upsilon(X_{t}\in L)\,\mathrm dt\geq \frac{1}{c_{L,t_L+1}}
\]
and hence such that
\begin{align}
  \label{eq:prooprostep2}
  e^{(\lambda^X +\varepsilon)T_\varepsilon}\P_x(X_{T_\varepsilon}\in L)&\geq 1.
\end{align}
We define the function $w_\varepsilon:(0,+\infty)\to [0,+\infty)$ by
\begin{align*}
  w_{\varepsilon}(x)=\sum_{i=0}^{T_\varepsilon-1} e^{(\lambda^X +\varepsilon)i}\P_x(X_i\in L)=\sum_{i=0}^{T_\varepsilon-1} e^{(\lambda^X +\varepsilon)i}\,Q_i\1_L(x),
\end{align*}
where we recall that $Q$ is the semigroup associated to the Markov process $X$. We thus have
\begin{align*}
  Q_1w_\varepsilon(x)&=\sum_{i=0}^{T_\varepsilon-1} e^{(\lambda^X +\varepsilon)i}\,Q_{i+1}\1_L(x)\\
  &=e^{-(\lambda^X +\varepsilon)} \sum_{i=1}^{T_\varepsilon} e^{(\lambda^X +\varepsilon)i}\,Q_{i}\1_L(x)\\
  &=e^{-(\lambda^X +\varepsilon)} \left(w_\varepsilon(x)+e^{(\lambda^X +\varepsilon)T_\varepsilon}Q_{T_\varepsilon}\1_L(x)-\1_L(x)\right).
\end{align*}
But, by~\eqref{eq:prooprostep2}, $e^{(\lambda^X +\varepsilon)T_\varepsilon}Q_{T_\varepsilon}\1_L(x)\geq 1$ for all $x\in L$, and hence we obtain, for all $x\in (0,+\infty)$,
\[
  Q_1w_\varepsilon(x)\geq  e^{-(\lambda^X +\varepsilon)} w_\varepsilon(x)
\]
and hence, by iteration,
\[
  Q_n w_\varepsilon(x)\geq e^{-(\lambda^X +\varepsilon)n} w_\varepsilon(x), \ \forall n\in\{0,1,\ldots\}.
\]
Since $w_\varepsilon(x)>0$ for all $x\in L$, we deduce that
\begin{align}
  \label{eq:prooprostep3}
  e^{-(\lambda^X +\varepsilon)n}  Q_n w_\varepsilon(x)&\xrightarrow[n\to+\infty]{} +\infty,\quad\forall x\in L.
\end{align}

Proposition~\ref{prop:Markov} and~\ref{prop:doeblin} entail that there exists $t_0>0$ such that
\[
  c_0:=\P_\upsilon(X_{t_0}\in L)>0.
\]
We can assume without loss of generality that $T_\varepsilon>t_L+t_0$.
Hence, for all $y\in L$, we have according to the Markov property and by Proposition~\ref{prop:doeblin}, for all $u\geq t$ such that $u-t\geq t_0+t_L$,
\begin{align*}
  \P_{y}(X_{u-t}\in L)\geq \E_y\left(\P_{X_{u-t-t_0}}(X_{t_0}\in L)\right)\geq c_{L,u-t-t_0}  \P_{\upsilon}(X_{t_0}\in L)\geq c_{L,u-t-t_0}  c_0.
\end{align*}
Using again the Markov property, we thus observe that, for all $u>t_0+t_L$, all $x\in(0,+\infty)$ and all $t\in[0,u-t_0-t_L]$,
\begin{align*}
  \P_x(X_u\in L)&\geq \E_x\left(\1_{X_t\in L}\P_{X_t}(X_{u-t}\in L)\right)\\
  &\geq \P_x(X_t\in L)\,c_{L,u-t-t_0} c_0.
\end{align*}
In particular, for all $u>t_0+t_L+T_\varepsilon$ and all $k\in\{0,1,\ldots,T_\varepsilon-1\}$,
\begin{align*}
  \P_x(X_u\in L)&\geq  \P_x(X_{\lfloor u\rfloor -T_\varepsilon+k}\in L)\,c_{L,u-\lfloor u\rfloor +T_\varepsilon-k-t_0} c_0\geq \P_x(X_{\lfloor u\rfloor-T_\varepsilon+k}\in L)\,c_{L,1 +T_\varepsilon-t_0} c_0.
\end{align*}
Hence, setting $\delta_\varepsilon= \sum_{k=0}^{T_\varepsilon-1} e^{(\lambda^X +\varepsilon)k}$, we have
\begin{align*}
  e^{(\lambda^X+2\varepsilon) u}\P_x(X_u\in L)& = \frac{e^{(\lambda^X+2\varepsilon) u}}{\delta_\varepsilon}  \sum_{k=0}^{T_\varepsilon-1} e^{(\lambda^X +\varepsilon)k} \P_x(X_u\in L)\\
  & \geq \frac{e^{(\lambda^X+2\varepsilon) u}}{\delta_\varepsilon} c_{L,1 +T_\varepsilon-t_0} c_0\sum_{k=0}^{T_\varepsilon-1} e^{(\lambda^X +\varepsilon)k}  \P_x(X_{\lfloor u\rfloor-T_\varepsilon+k}\in L)\\
  & = \frac{e^{(\lambda^X+2\varepsilon) u}}{\delta_\varepsilon} c_{L,1 +T_\varepsilon-t_0} c_0\sum_{k=0}^{T_\varepsilon-1} e^{(\lambda^X +\varepsilon)k}  Q_{\lfloor u\rfloor-T_\varepsilon+k}\1_L(x)\\
  & \geq  \frac{e^{(\lambda^X+2\varepsilon) (\lfloor u\rfloor-T_\varepsilon)}}{\delta_\varepsilon} c_{L,1 +T_\varepsilon-t_0} c_0 Q_{\lfloor u\rfloor-T_\varepsilon}w_\varepsilon(x).
\end{align*}
By~\eqref{eq:prooprostep3}, this shows that $e^{(\lambda^X+2\varepsilon) u}\P_x(X_u\in L)$ goes to infinity when $u\to+\infty$. In particular, $\lambda^X+2\varepsilon\geq \lambda_0^X$. Since this is true for all $\varepsilon>0$, we deduce that $\lambda^X\geq \lambda_0^X$ and hence $\lambda \geq \lambda_0$. Since this is true for all $\lambda>\lambda'_0$, we deduce that $\lambda'_0\geq \lambda_0$, which concludes the proof of the proposition.

\subsection{Proof of Proposition~\ref{prop:lambda0upperbound}}
\label{sec:lambda0upperbound}

\textbf{(1) Proof of \ref{i:lambda0upperbound:a}}
We set $\psi_1=\psi/h$ and $\psi_2=\xi(x)/h$, both extended by the value $0$ at point $\d$. We observe that (up to a change in the constant $C>0$)
\[
  \cL \psi_1\leq -(\lambda_1+b)\psi_1+C\1_L\text{ and }\cL \psi_2\geq -(\lambda_2+b)\psi_2.
\]
Since $\psi_2$ is continuous and positive, it is lower bounded on the compact interval $L$, and hence we have $\cL\psi_1\leq -(\lambda_1+b)\psi_1+C'\psi_2$, for some constant $C'>0$. Hence setting $F=\psi_1-\frac{C'}{\lambda_1-\lambda_2}\psi_2$, we obtain
\[
  \cL F\leq -(\lambda_1+b) \psi_1+C'\psi_2+\frac{(\lambda_2+b) C'}{\lambda_1-\lambda_2} \psi_2=-(\lambda_1+b) F.
\]
Fix $x\in (0,+\infty)$. Using the same approach as in section~\ref{sec:F2}
(note that $F$ is lower bounded on $(0,+\infty)$ and positive in a
neighbourhood of $\{0,+\infty\}$), we deduce that, for all $t\geq 0$,
\[
  \E_x[e^{(\lambda_1+b) t}F(X_t)] \leq F(x).
\]
In particular, for all $t\geq 0$,
\[
  \E_x(\psi_1(X_t))\leq \frac{C'}{\lambda_1-\lambda_2}\E_x(\psi_2(X_t))+e^{-(\lambda_1+b) t}F(x).
\]
For all $M>0$, there exists a compact interval $L_M\subset (0,+\infty)$ such that $\psi_1\geq M\psi_2$ on the set $E\setminus L_M$. Hence, for all $t\geq 0$,
\[
  M\E_x(\psi_2(X_t)\1_{X_t\notin L_M})\leq \frac{C'}{\lambda_1-\lambda_2}\E_x(\psi_2(X_t))+e^{-(\lambda_1+b) t}F(x),
\]
so that, choosing $M=\frac{C'}{\lambda_1-\lambda_2}+1$,
\[
  \E_x(\psi_2(X_t)\1_{X_t\notin L_M})\leq \frac{C'}{\lambda_1-\lambda_2}\E_x(\psi_2(X_t)\1_{X_t\in L_M})+e^{-(\lambda_1+b) t}F(x),
\]
which entails
\begin{align}
  \E_x(\psi_2(X_t))&\leq \left(1+\frac{C'}{\lambda_1-b}\right)\E_x(\psi_2(X_t)\1_{X_t\in L_M})+e^{-(\lambda_1+b) t} F(x)\notag\\
  &\leq \left(1+\frac{C'}{\lambda_1-b}\right)\P_x(X_t\in L_M)+e^{-(\lambda_1+b) t} F(x)\label{eq:psi2bound}.
\end{align}
In addition, by Corollary~\ref{cor:1} (and more precisely its proof), we have, for all $t\geq 0$,
\[
  \mathbb E_x(\psi_2(X_{t}))= \psi_2(x)+\int_0^t \mathbb E_x(\mathcal L\psi_2(X_u))\,\mathrm du\geq \psi_2(x)-(\lambda_2+b)\int_0^t \E_x(\psi_2(X_u))\,\mathrm du
\]
and hence, by Grownwall's Lemma,
\[
  \psi_2(x)\leq e^{(\lambda_2+b) t}\E_x(\psi_2(X_t)).
\]
The last inequality and~\eqref{eq:psi2bound} and the fact that $\lambda_2+b<\lambda_1+b$ imply that, for any fixed $\lambda'\in(\lambda_2,\lambda_1)$,
\[
  e^{(\lambda'+b) t}\P_x(X_t\in L_M)\xrightarrow[t\to+\infty]{} +\infty.
\]
In particular $\lambda_0^X\leq \lambda'+b$, so that $\lambda_0\leq \lambda'$ for any $\lambda'\in(\lambda_2,\lambda_1)$, which concludes the proof of Proposition~\ref{prop:lambda0upperbound}~(a).

\textbf{(2) Proof of \ref{i:lambda0upperbound:b}}
The intuition for this part is that we wish to consider a semigroup generated by
the operator $f \mapsto \frac{\cA(f\xi)}{\xi}$, represent this in terms of a
Markov process  $Y$ together with a potential term $e^{\int_0^t d(Y_s)\, \dd s}$, 
and use \cite{BW-spec} to bound its growth coefficient. However, we must be
cautious: $\xi$ cannot be used in place of $h$ in Assumption~\ref{assumption1},
so we cannot use our established existence and uniqueness results,
and moreover, the potential term we would get from the calculation above is not bounded.
Instead, we first define a process $Y$ with the desired properties, and then
consider introducing a truncated potential ($d^M$ below) to allow us to apply
\cite{BW-spec}. Once this is done, we relate this back to the original semigroup $T$
by applying Theorem~\ref{thm:semigroup} to a truncated version of $\cA$, and this allows
us to bound $\lambda_0$.

We consider the right-continuous PDMP $Y$ with drift $s$ and jump kernel $\bar k(x,\mathrm dy)=\frac{\xi(y)}{\xi(x)}k(x,\mathrm dy)$. The function $V=\psi/\xi$ satisfies
\begin{align*}
  &\!\! \frac{\partial V}{\partial s}(x)+\int_{(0,x)} (V(y)-V(x))\bar k(x,\mathrm dy)
  \\
  &=V(x)\left(\frac{1}{\psi(x)}\frac{\partial \psi}{\partial s}(x)-\frac{1}{\xi(x)}\frac{\partial \xi}{\partial s}(x)
  +\int_{(0,x)} \frac{\psi(y)}{\psi(x)} k(x,\mathrm dy)-\bar k(x,(0,x))\right)\\
  &=V(x)\left(\frac{\cA \psi}{\psi}-\frac{\cA \xi}{\xi}\right)\leq V(x)(-\lambda_1+C\1_L(x)+\lambda_2)\\
  &\leq -\lambda V(x)+C\,\max_L V\,\1_L(x),
\end{align*}
where $\lambda=\lambda_1-\lambda_2\geq 0$. Since in addition $V(x)\to+\infty$ when $x\to0$ or $x\to+\infty$, and since the jump rate $\bar k(x,(0,x))$ is locally bounded, this entails that $Y$ is non explosive and recurrent. Its extended infinitesimal generator, denoted by $\cL^Y$, satisfies, for all $f\in C_c^{(s)}$,
\begin{align*}
  \cL^Y f(x)=\frac{\partial f}{\partial s}(x)+\int_{(0,x)} (f(y)-f(x))\,\bar k(x,\mathrm dy),\ \forall x\in(0,+\infty).
\end{align*}

Let $M > \inf_{x>0} \cA\xi(x)/\xi(x)$ and define
$d^M:x\in(0,+\infty)\mapsto d(x)\wedge M$. Consider the semigroup
\begin{align*}
  S^M_tf(x):=\mathbb E_x\left(\exp\left(\int_0^t d^M(Y_s)\,\mathrm ds\right)f(Y_t)\right),\forall t\geq 0,x\in(0,+\infty).
\end{align*}
According to Proposition~2.1 in~\cite{Cavalli2019} (see also Proposition~3.4
in~\cite{BW-spec}), if $Y$ is recurrent, then $-\lambda_0(S^M)\geq \inf_{x>0}
d(x)$, with strict inequality if $d$ is not constant, where $\lambda_0(S^M)$ is
the growth coefficient of $S^M$ (beware of the difference of sign convention in
the definition of the growth coefficient in the cited works).

We therefore need to prove that $\lambda_0 \le \lambda_0(S^M)$.
For all $f\in C_c^{(s)}$ and for $f\equiv 1$, we have
\begin{align*}
  S^M_tf(x) 
  &= f(x)+\mathbb E_x\left(\int_0^{t} \exp\left(\int_0^u d^M(Y_s)\,\mathrm ds\right)\left(d^M(Y_u)  f(Y_u)+\cL^Y f(Y_u)\right)\,\mathrm du\right)\\
  &= f(x)+\int_0^t S_u^M(d^M\,f+\cL^Y f)(x)\,\mathrm du.
\end{align*}
Let
\[
  K^M(x)=\bar k(x,(0,x)) -d^M(x)
\]
and
\[
  \mathcal B^M f(x)=d^M(x)\,f(x)+\cL^Y f(x)=\frac{\partial f}{\partial s}(x)+\int_{(0,x)} f(y)\,\bar k(x,\mathrm dy)- K^M(x) f(x).
\]
The operator $\mathcal{B}^M$ is a growth-fragmentation operator just like
$\cA$, and indeed, it satisfies Assumption~\ref{assumption1} with $h'\equiv 1$
instead of $h$.  In particular, according to Theorem~\ref{thm:semigroup}, $S^M$
is the unique semigroup such that, for all $f\in C_c^{(s)}$ and for $f\equiv
1$, for all $t\geq 0$ and all $x\in(0,+\infty)$,
\[
  S^M_tf(x)=f(x)+\int_0^t S^M_u(\mathcal B^M f)(x)\,\mathrm du.
\]
We now define, for all $f\in\cD(\cA)$ and all $x\in(0,+\infty)$,
\begin{align*}
  \cA^M f(x)=\cA f(x)-(d(x)-d^M(x))f(x).
\end{align*}
Then
\[
\frac{\cA h(x)}{h(x)}-d(x)\leq \frac{\cA^M h(x)}{h(x)}\leq \frac{\cA h(x)}{h(x)}
\]
 with $d$ locally bounded, since $\xi$ is locally lower bounded away from 0 and since  $\frac{\partial \xi}{\partial s}$ and $\int_{(0,x)} \xi(y)\,k(x,\mathrm dy)$ are locally bounded by assumption.
Hence one easily
checks that  $\cA^M$ satisfies Assumption~\ref{assumption1}. 
Let $T^M$ be the
associated semigroup (whose existence and uniqueness  is ensured by
Theorem~\ref{thm:semigroup}). Since $\xi\in L^\infty(h)$ and 
$\mathcal A \xi/h\geq -\lambda_2\xi/h$ is lower bounded,
we deduce from
Corollary~\ref{cor:1} that
\begin{align*}
  T^M_t \xi(x)=\xi(x)+\int_0^t T^M_u(\cA^M \xi)(x)\,\mathrm du.
\end{align*}
In particular, the semigroup $\widetilde T^M$ defined, for all $f\in C_c^{(s)}$ and for $f\equiv 1$, by 
\[
  \widetilde T^M_t f(x)=\frac{1}{\xi(x)} T^M_t (\xi f)(x), \ \forall t\geq 0,\forall x\in(0,+\infty),
\]
satisfies, for all such $f$, $x$ and $t$,
\begin{align*}
  \widetilde T^M_t f(x)=f(x)+\int_0^t \frac{1}{\xi (x)} T^M_u(\cA^M (\xi f))(x)\,\mathrm du=f(x)+\int_0^t  \widetilde T^M_u(\widetilde \cA^M f)(x)\,\mathrm du
\end{align*}
where
\begin{align*}
  \widetilde \cA^M f(x)&=\frac{\cA^M (f\xi )(x)}{\xi (x)}\\
  &=\frac{\partial f}{\partial s}(x)+\int_{(0,x)} f(y)\,\bar k(x,\mathrm dy)-K(x)f(x)+\frac{1}{\xi (x)}\frac{\partial \xi }{\partial s}(x)f(x)
  \\
  & \quad {} 
  - (d(x)-d^M(x))f(x)\\
  &=\frac{\partial f}{\partial s}(x)+\int_{(0,x)} f(y)\,\bar k(x,\mathrm dy)+\left(\frac{\cA \xi (x)}{\xi (x)}-\bar k(x,(0,x)) -d(x)+d^M(x)\right)f(x)\\
  &=\mathcal B^M f(x).
\end{align*}
This entails that, for all non-negative measurable function $f:(0,+\infty)\to[0,+\infty)$, all $x\in(0,+\infty)$ and all $t\geq 0$,
\begin{align*}
  S_t^M f(x)=\widetilde T^M_t f(x)=\frac{1}{\xi (x)} T^M (\xi  f)(x).
\end{align*}
But, according to the representation of $T^M$ as the $1/h$ transform of a sub-Markov process (see Proposition~\ref{tNirr} and the conclusion of the proof of Theorem~\ref{thm:semigroup} in section~\ref{sec:endProofTh1}), we have
\begin{align*}
  \frac{1}{\xi (x)} T^M (\xi  f)(x)=\frac{h(x)e^{b^M t}}{\xi (x)}\mathbb E_x\left(f(X^M_t)\frac{\xi (X^M_t)}{h(X^M_t)}\1_{X^M_t\neq\partial}\right),
\end{align*}
where $X^M$ is a $(0,+\infty)\cup\{\partial\}$-valued PDMP with drift determined by $s$, jump kernel $\frac{h(y)}{h(x)}k(x,\mathrm dy)$ and killing rate (that is jump rate toward $\partial$) 
\[
  q^M(x)=b^M-\frac{\cA^M h(x)}{h(x)},\quad\text{with }
b^M=\sup_{x\in(0,+\infty)} \frac{\cA^M h(x)}{h(x)}\leq b.
\]
Moreover,
$b^M-q^M(x) = \frac{\cA^M h(x)}{h(x)} \le \frac{\cA h(x)}{h(x)} = b - q(x)$, so
\begin{align*}
  S_t^M f(x)
  = \frac{1}{\xi (x)} T^M_t (\xi  f)(x)
  &= \frac{h(x)e^{b^M t}}{\xi (x)}\mathbb E_x\left(\exp\left(-\int_0^t q^M(Z_u)\,\mathrm du \right)f(Z_t)\frac{\xi (Z_t)}{h(Z_t)}\right) \\
  &\le \frac{h(x)e^{b t}}{\xi (x)}\mathbb E_x\left(\exp\left(-\int_0^t q(Z_u)\,\mathrm du \right)f(Z_t)\frac{\xi (Z_t)}{h(Z_t)}\right)
  \\
  &= \frac{1}{\xi(x)} T_t(\xi f)(x)
\end{align*}
where $Z$ is a (conservative) PDMP with drift determined by $s$ and jump kernel $\frac{h(y)}{h(x)}k(x,\mathrm dy)$. 

We hence obtain, immediately from the definition, that $\lambda_0(S^M) \ge \lambda_0$,
which is what we needed to prove.

\begin{appendix}

\section{Appendix}
\label{sec:appendix}

Let $s$ be continuous (strictly) increasing function from $(0,+\infty)$ to $\R$ such that $s(+\infty)=+\infty$, let $Q$ be a non-negative kernel from $(0,+\infty)\cup\{\d\}$ to $(0,+\infty)\cup\{\d\}$ such that $Q(\d,(0,+\infty)\cup\{\partial\})=0$ and $Q(x,[x,+\infty))=0$ for all $x>0$, where $\d\notin (0,+\infty)$ is an isolated point. From now on, we set $E=(0,+\infty)\cup\{\d\}$. We consider the PDMP $X$ with state space $E$, directed by the flow $\phi$ defined by~\eqref{eq:flow} (with $\phi(\d,t)=\d$ for all $t\geq 0$) between its jumps and with jump kernel $Q$ (note that $\d$ is an absorption point for $X$).

In the following results, $C_b(E)$ denotes the set of bounded real valued continuous functions on~$E$ and $C_0(E)$ the set of bounded continuous function vanishing at infinity. We emphasize that its statement and proof can be easily adapted to the case where $X$ takes its values in $[0,+\infty)$ or $\R$.

The first part of the following proposition is proved by Davis in \cite[Theorem 27.6]{Davis}, when $\phi$ is generated by a Lipschitz vector field and $x\mapsto Q(x,(0,+\infty)\cup\{\d\})$ is continuous and bounded. In our case, we do not assume this regularity, but use instead  the fact that our state space is one dimensional.

\begin{proposition}
  \label{prop:feller}
  Assume that $\sup_{x\in (0,M)} Q(x,E)<+\infty$ for all $M>0$. Then the semigroup $T$ of $X$ maps $C_b(E)$ to itself.

  If in addition $s(0+)=-\infty$, $\sup_{x\in E}Q(x,E)<+\infty$ and, for all $M>0$, 
  we also have $\limsup_{x\to+\infty} Q(x,(0,M)\cup\{\d\})=\limsup_{x\to 0} Q(x,\{\d\})=0$, then the semigroup of $X$ is Feller, meaning that it maps $C_0(E)$ to itself and is strongly continuous on $C_0(E)$.
\end{proposition}

\begin{proof}
  We start by showing the first part, and then the second part of Proposition~\ref{prop:feller}.

  \textbf{(1) $T$ maps $C_b(E)$ to itself.} Our proof is a simple adaptation of the proof of~\cite[Theorem 27.6]{Davis} to our particular one-dimensional setting.  Since $s(+\infty)=+\infty$, the explosion time of $\phi(x,\cdot)$ (denoted by $t_*(x)$ in the cited reference) is equal to infinity for all $x\in E$. Moreover, since $\sup_{x\in (0,M)} Q(x,E)<+\infty$, the process $X$ is non-explosive (as detailed in the first step of the proof of Proposition~\ref{tNirr}) and well defined for all time $t\geq 0$, for any initial distribution. Finally, 
  $Q(x,E)$ is uniformly bounded over $x\in E$.

  The only  difference with the proof of~\cite[Theorem 27.6]{Davis} is that, in our case, it is not immediate that, for any $\psi\in C_b(\R_+\times E)$ and $f\in C_b(E)$, the term
  \begin{align*}
    G\psi(x,t):=f(\phi(x,t)) e^{-\Lambda(t,x)}+\int_0^t \int_E \psi(t-u,y) Q(\phi(x,u),\mathrm dy)e^{-\Lambda(x,u)}\,\mathrm du
  \end{align*}
  where
  \[
    \Lambda(x,t):=\int_0^t Q(\phi(x,u),E)\mathrm du,
  \]
  is continuous in $(t,x)\in[0,+\infty)\times E$ and bounded. The rest of the proof is identical to the one of~\cite[Theorem 27.6]{Davis} and we thus only need to prove that $G\psi\in C_b(\R_+\times E)$ to conclude. 

  First note that $\|G\psi\|_\infty\leq \|f\|_\infty+\|\psi\|_{\infty}$, so that it is bounded. It only remains to prove that $G\psi$ is continuous. Since $Q(\d,\mathrm dy)=0$ and since $\phi(\d,t)=\d$ for all $t\geq 0$, we have $G\psi(\d,t)=f(\d)$ for all $t\geq 0$ and hence $G\psi$ is continuous on $\{\d\}\times [0,+\infty)$. Now let $(x,t)\in(0,+\infty)\times[0,+\infty)$ and $(\varepsilon,h)\in\R\times \R$ such that $(x+\varepsilon,t+h)\in(0,+\infty)\times [0,+\infty)$. We have, for all $u\geq 0$, denoting $\delta_{x,\varepsilon}:=s(x+\varepsilon)-s(x)$
  \begin{align}
    \label{eq:AppUseful1}
    \phi(x+\varepsilon,u)=s^{-1}(s(x+\varepsilon)+u)=s^{-1}(s(x)+(u+s(x+\varepsilon)-s(x)))=\phi(x,u+\delta_{x,\varepsilon}).
  \end{align}
  In particular, 
  \begin{align*}
    \Lambda(x+\varepsilon,t+h)&=\int_0^{t+h}Q(\phi(x+\varepsilon,u),E)\mathrm du\\
    &=\int_0^{t+h}Q(\phi(x,u+\delta_{x,\varepsilon}),E)\mathrm du\\
    &=\int_{\delta_{x,\varepsilon}}^{t+h+\delta_{x,\varepsilon}} Q(\phi(x,u),E)\mathrm du,
  \end{align*}
  so that $\Lambda$ is continuous and more precisely
  \begin{align}
    \label{eq:AppUseful2}
    \left|\Lambda(x+\varepsilon,t+h)-\Lambda(x,t)\right|\leq (2\delta_{x,\varepsilon}+h)\,\sup_{y\in(0,\phi(x,t+h+\delta_{x,\varepsilon}))}Q(y,E) .
  \end{align}
  Using again~\eqref{eq:AppUseful1}, we also obtain
  \begin{multline*}
    \int_0^{t+h} \int_E \psi(t-u,y) Q(\phi(x+\varepsilon,u),\mathrm dy)e^{-\Lambda(x+\varepsilon,u)}\,\mathrm du\\
    \begin{aligned}
      &=\int_0^{t+h} \int_E \psi(t-u,y) Q(\phi(x,u+\delta_{x,\varepsilon}),\mathrm dy)e^{-\Lambda(x+\varepsilon,u)}\,\mathrm du\\
      &=\int_{\delta_{x,\varepsilon}}^{t+h+\delta_{x,\varepsilon}} \int_E \psi(t-u-\delta_{x,\varepsilon},y) Q(\phi(x,u),\mathrm dy)e^{-\Lambda(x+\varepsilon,u-\delta_{x,\varepsilon})}\,\mathrm du.
    \end{aligned}
  \end{multline*}
  By dominated convergence, continuity of $\psi$ and of $\Lambda$, we deduce that the last term converges to $\int_0^{t} \int_E \psi(t-u,y) Q(\phi(x,u),\mathrm dy)e^{-\Lambda(x,u)}\,\mathrm du$ when $(\varepsilon,h)\to 0$. In particular, using this and the continuity of $\phi$, of $f$ and of $\Lambda$, we deduce that $G\psi(x,t)$ is indeed continuous in $(x,t)$, which concludes the proof of the first part of Proposition~\ref{prop:feller}.

  \textbf{(2) $T$ maps $C_0(E)$ to itself.} We assume that $s(0+)=-\infty$, that $\sup_{x\in E}Q(x,E)<+\infty$ and that, for all $M>0$, $\limsup_{x\to+\infty} Q(x,(0,M)\cup\{\d\})=\limsup_{x\to0} Q(x,\{\d\})=0$. 

  Let $f\in C_0(E)$, fix $\varepsilon>0$, and let $n_0$ be large enough that $\sup_{x\in(0,1/n_0)\cup(n_0,+\infty)} f(x)\leq \varepsilon$.

  Denoting by $T_1<T_2<\cdots$ the successive jump times of $X$, we deduce from the boundedness of $Q(\cdot,E)$, that, for all $t\geq 0$,
  \[
    \sup_{x\in E} \P(T_n\leq t\,\mid X_0=x)\xrightarrow[n\to+\infty]{} 0.
  \]
  Fix $n_1$ such that $\sup_{x\in E} \P(T_{n_1}\leq t\mid X_0=x)\leq \varepsilon$. 
  Since the process $X$ is almost-surely non-decreasing between the jumps, its law at time $t$ on the event $T_n \leq  t < T_{n+1}$ stochastically dominates the $n^{th}$ iterate of $Q$, denoted by $Q^n$ (consider that $\d$ is below $0$). By assumption we have $\limsup_{x\to+\infty} Q(x,(0,n_0)\cup\{\d\})=0$, so that, for all $n\geq 0$, $\limsup_{x\to+\infty} Q^n(x,(0,n_0)\cup\{\d\})=0$, and hence there exists $n_2\geq 1$ such that, for all $n\in\{0,\ldots,n_1\}$,
  \[
    \sup_{x\geq n_2} \P(X_t< n_0\,T_n\leq t\leq T_{n+1}\mid\,X_0=x)\leq \varepsilon/(n_1+1).
  \]
  In particular,
  \begin{align*}
    \sup_{x\geq n_2} \P(X_t< n_0\,\mid\,X_0=x)
    & \leq \sup_{x\geq n_2} \sum_{n=0}^{n_1} \P(X_t< n_0\,T_n\leq t<T_{n+1}\mid\,X_0=x)
    \\
    & \quad {} +\sup_{x\geq n_2} \P(T_{n_1}\leq t\mid X_0=x) \leq 2\,\varepsilon.
  \end{align*}
  As a consequence,
  \[
    \sup_{x\geq n_2} \E(f(X_t)\,\mid \, X_0=x)\leq 2\varepsilon\|f\|_\infty+\varepsilon.
  \]
  Since the existence of $n_2$ is true for any fixed $\varepsilon>0$, we deduce that
  \[
    \E(f(X_t)\mid X_0=x)\xrightarrow[x\to+\infty]{} 0.
  \]
  Now, since $s(x)\xrightarrow[x\to 0]{}-\infty$, we deduce that $\phi(x,t)\to 0$ when $x\to 0$. Since $X_t\leq \phi(x,t)$ or $X_t=\d$ almost surely when it starts from $x$ at time $0$, we deduce that, if $f$ vanishes at $0$, then so does $T_tf(x)=\E(f(X_t)\1_{X_t\neq \d}\mid X_0=x)$ when $x\to0$. Moreover, the jumping rate from $y$ to $\d$ goes to $0$ when $y\to0$, so that $\P(X_t=\d\mid X_0=x)\to 0$ when $x\to0$. Finally, we deduce that $T_tf(x)\to 0$ when $x\to 0$ or $x\to+\infty$.

  We conclude that $T_t$ maps the space of continuous functions vanishing at $0$ and infinity to itself.

  \textbf{(3) $T$ is strongly continuous.} We proceed under the same assumptions as in step~(2).  Let $f$ be in the  space of continuous functions vanishing at $0$ and infinity. 
  Fix $\varepsilon>0$. Since $Q(\cdot,E)$ is uniformly bounded, say by a constant $C$, then the probability that the process has no jumps between times $0$ and $t$ is larger than $e^{-tC}$, for any $t\geq 0$. Hence
  \begin{align}
    \label{eq:feller1}
    \left| T_t f(x)-f(\phi(x,t))\right|\leq 1-e^{-tC}\|f\|_\infty.
  \end{align}
  Since $f$ vanishes at infinity, there exists $n_3$ large enough so that $f(x)\leq \varepsilon$ for all $x\geq n_3$ or $x<1/n_3$. Since we have $\phi(x,t)\geq x$ for all starting position $x\geq n_3$ and $t\geq 0$, we deduce that 
  \begin{align}
    \label{eq:feller2}
    \sup_{x\geq n'}\left|f(\phi(x,t))-f(x)\right|\leq 2\,\sup_{x\geq n'}|f(x)|\leq 2\varepsilon.
  \end{align}
  Similarly, $\phi(x,t)\leq \phi(1/(n_3+1),t)$ for all starting position $x\leq 1/(n_3+1)$. Since $\phi(1/(n_3+1),t)\to 1/(n_3+1)$ when $t\to0$, there exists $t_0>0$ such that $\phi(x,t)\leq 1/n_3$ for all $x\leq 1/(n_3+1)$ and $t\in[0,t_0]$. We deduce that
  \begin{align}
    \label{eq:feller3}
    \sup_{x\leq 1/(n_3+1)}\left|f(\phi(x,t))-f(x)\right|\leq 2\varepsilon.
  \end{align}
  Finally, $\phi(x,t)$ converges to $x$ when $t\to 0$, uniformly on compact sets and $f$ is uniformly continuous on $[1/(n_3+1),n_3]$, so that there exists $t_1>0$ such that
  \[
    \sup_{x\in [1/(n_3+1),n_3]} \left| f(\phi(x,t))-f(x)\right|\leq \varepsilon,\ \forall t\leq t_1.
  \]
  Using the last equation and inequalities~\eqref{eq:feller1}, \eqref{eq:feller2} and~\eqref{eq:feller3}, we deduce that there exists $t_2>0$ such that, for all $t\leq t_2$, (note that the case $x=\d$ is trivial)
  \[
    \sup_{x\in E} \left|T_t f(x)-f(x)\right|\leq 3\varepsilon.
  \]
  Since this is true for any $\varepsilon>0$, we deduce that $T_t f$ converges to $f$ in the uniform topology. This means that $T$ is a strongly continuous semigroup on $C_0(E)$ and concludes the proof of Proposition~\ref{prop:feller}.

\end{proof}

In the following result, we characterize the infinitesimal generator of $X$ when its semigroup is Feller. We recall that, if $s^{-1}:(-\infty,+\infty)\to(0,+\infty)$, then, given a function $f:(0,+\infty)\to\R$ such that $f\circ s^{-1}$ is absolutely continuous, the function $f\circ s^{-1}$ is $\lambda_1$-almost everywhere differentiable and that, for any function $g$ equals to this derivative $\lambda_1$-almost everywhere, we have
\[
  f\circ s^{-1}(t)-f\circ s^{-1}(u) = \int_u^t g(v)\,\mathrm dv,\quad\forall u\leq t\in(-\infty,+\infty).
\]
One easily checks that, as a consequence, $f$ is differentiable with respect to $s$, $\lambda_1\circ s$-almost everywhere, with derivative $h=g\circ s$, and that
\[
  f(y)-f(x) = \int_x^y h(z)\,\mathrm ds(z),\quad\forall x\leq y\in(0,+\infty).
\]
In this case, we will say that $f$ is $s$-absolutely continuous and that $h$ is a $s$-derivative of $f$.
We consider the domain $\mathcal D(\mathcal U)$, defined as the set functions $f:E\to\RR$ such that $f\rvert_{(0,+\infty)}$ is an $s$-absolutely continuous function  admitting  a $s$-derivative $h$ such that $x\mapsto h(x)+\int_{(0,x)} (f(y)-f(x)) Q(x,\mathrm dy)$ is an element of $C_0(E)$, where we set $\frac{\partial f}{\partial s}(\partial)=0$. We also define the operator $\mathcal U:\cD(\mathcal U)\to C_0(E)$ by
\[
  \mathcal U f(x)=\frac{\d f}{\d s}(x)+\int_{(0,x)} (f(y)-f(x))\,Q(x,\mathrm dy),\ \forall x\in E,
\]
where $\frac{\d f}{\d s}$ is the $s$-derivative of $f$ extended with $\frac{\partial f}{\partial s}(\partial)=0$ and such that $x\in E\mapsto \frac{\d f}{\d s}(x)+\int_{(0,x)} (f(y)-f(x)) Q(x,\mathrm dy)\in C_0(E)$.

\begin{proposition}
  Assumes that $s(0+)=-\infty$, that $\sup_{x\in E}Q(x,E)<+\infty$ and that, for all $M>0$, we have $\limsup_{x\to+\infty} Q(x,(0,M)\cup\{\d\})=\limsup_{x\to 0} Q(x,\{\d\})=0$. Then the infinitesimal generator of the semigroup $T$ of $X$ acting on $C_0(E)$ is given by  $(\mathcal U,\cD(\mathcal U))$. Moreover, for all bounded $f:E\to\RR$ such that $f\rvert_{ (0,+\infty)}$ is $s$-absolutely continuous, denoting by  $\d f/\d s$ any $s$-derivative of $f\rvert_{ (0,+\infty)}$ extended with $\frac{\partial f}{\partial s}(\d)=0$,
  \[
    M_t^f:=f(X_t)-f(x)-\int_0^t Q_u\widetilde{\mathcal U} f(X_s)\,\mathrm ds,
  \]
  with
  \[
    \widetilde{\mathcal U}f(y):=\frac{\d f}{\d s}(y)+\int_{(0,x)} (f(y)-f(x))\,Q(x,\mathrm dy),
  \]
  defines a local martingale under $\mathbb P_x$, for any $x\in E$.
\end{proposition}

\begin{proof}
  Let us denote by $\mathcal G$ the infinitesimal generator of $T$ and by $\cD(\mathcal G)$ its domain. Our aim is to prove that $(\mathcal G,\cD(\mathcal G))=(\mathcal U,\cD(\mathcal U))$.

  We make use of the fact that the proof of Theorem~26.14 in~\cite{Davis} adapts  directly to our situation where the flow $\phi$ is generated by $s$ on $(0,+\infty)$ and by $0$ on $\d$, instead of a Lipschitz flow $\mathcal X$ on $\R^ d$. The only adaptation lies in the fact that, between two successive jumps, say at times $T_{i-1}$ and $T_i$, and for any function $f$ such that $f\rvert_{(0,+\infty)}$ is $s$-absolutely continuous, we have (for the second equality, recall that $\frac{\mathrm d f}{\mathrm du}(\phi(y,u))=\frac{\d f}{\d s}(\phi(y,u))$ as soon as the derivative is well defined) if $X_{T_{i-1}}\in (0,+\infty)$
  \begin{align*}
    f(X_{T_i-})-f(X_{T_{i-1}})&=\int_{0}^{T_i-T_{i-1}} \frac{\mathrm d f}{\mathrm du}(\phi(X_{T_{i-1}},u))\,\mathrm du\\
    &=\int_{0}^{T_i-T_{i-1}} \frac{\d f}{\d s}(\phi(X_{T_{i-1}},u))\,\mathrm du\\
    &=\int_{T_{i-1}}^{T_i}  \frac{\d f}{\d s}(X_v)\,\mathrm dv,
  \end{align*}
  instead of $f(X_{T_i-})-f(X_{T_{i-1}})=\int_{T_{i-1}}^{T_i}  \mathcal X f(X_v)\,\mathrm dv$ in~\cite{Davis}, while $f(X_t)-f(X_{T_{i-1}})=0$ for all $t\geq T_{i-1}$  if $X_{T_{i-1}}=\partial$.

  In particular, this result implies that any bounded $f:E\to\RR$ such that $f\rvert_{ (0,+\infty)}$ is $s$-absolutely continuous is in the domain of the extended infinitesimal generator of $X$, say   $\mathcal U'$, and that $\mathcal U' f(x)=\frac{\d f}{\d s}(x)+\int_{(0,x)} (f(y)-f(x))\,Q(x,\mathrm dy)$, for any $s$-derivative $\d f/\d s$ of $s$ (note that conditions 2. and 3. of Theorem~26.14 in~\cite{Davis} are trivially satisfied in our case, respectively because the boundary of the domain is not reached and because the number of jumps is finite in any finite time horizon almost surely). This proves that $M^f$ is a local  martingale under $\mathbb P_x$, for all $x\in E$.

  In particular, given $f\in\cD(\mathcal U)$, the stochastic process defined, for all $t\geq 0$, by
  \[
    M^f_t=f(X_t)-f(x)-\int_0^t \mathcal U f(X_u)\,\mathrm du
  \]
  is a local martingale under $\mathbb P_x$, for all $x\in E$. Since $f$ and $\mathcal U f$ are bounded and $M^f$ is c\`adl\`ag, we deduce that it is a martingale and thus, taking the expectation, we obtain
  \[
    \frac{T_t f(x)-f(x)}{t}=\frac{1}{t}\int_0^t T_u \mathcal U f(x)\,\mathrm du,\ \forall x\in E.
  \]
  Moreover, since $T$ is strongly continuous on $C_0(E)$ by Proposition~\ref{prop:feller} and since $\mathcal U f\in C_0(E)$ by assumption, for all $x\in E$,
  \[
    \left| \frac{1}{t}\int_0^t T_u \mathcal U f(x)\,\mathrm du-\mathcal U f(x)\right|\leq  \frac{1}{t}\int_0^t \left\|T_u \mathcal U f-\mathcal Uf\right\|_\infty\,\mathrm du\xrightarrow[t\to 0]{}0.
  \]
  We conclude that, for any $f\in\cD(\mathcal U)$, we have $f\in\cD(\mathcal G)$ and $\mathcal G f=\mathcal U f$.

  Reciprocally, assume that $f\in\cD(\mathcal G)$. Then $f$ is in the domain of the extended infinitesimal generator of $X$, so that, according to Theorem~26.14 in~\cite{Davis}, $f\rvert_{ (0,+\infty)}$ is $s$-absolutely continuous and
  \[
    M_t=f(X_t)-f(x)-\int_0^t \left[\frac{\d f}{\d s}(X_u)+\int_{(0,X_u)} (f(y)-f(X_u)) Q(X_u,\mathrm dy)\right]\,\mathrm du
  \]
  is a local martingale under $\mathbb P_x$ for all $x\in E$, where $\frac{\d f}{\d s}$ is an $s$-derivative of $f\rvert_{ (0,+\infty)}$ extended by $\frac{\d f}{\d s}(\d)=0$. Moreover, denoting by $\mathcal G$ the infinitesimal generator of $X$, we have that
  \[
    M'_t=f(X_t)-f(x)-\int_0^t \mathcal G f(X_u)\,\mathrm du
  \]
  is a martingale under $\mathbb P_x$. In particular, $M_t-M'_t$ is a continuous local martingale with bounded total variation and hence it is constant $\P_x$-almost surely. We deduce that, $\mathbb P_x$-almost surely, 
  \begin{align*}
    \frac{\d f}{\d s}(X_u)+\int_{(0,X_u)} (f(y)-f(X_u)) Q(X_u,\mathrm dy)=\mathcal G f(X_u),\quad\lambda_1(\mathrm du)-\text{almost everywhere}.
  \end{align*}
  Since the jump rate $Q$ is bounded, we know that, for any $t>0$, with positive probability, $X_u=\phi(x,u)$ for all $u\in[0,t]$. This and the previous equality entails that $\frac{\d f}{\d s}(z)+\int_{(0,z)} (f(y)-f(z)) Q(z,\mathrm dy)$ equals $\mathcal G f(z)$ for $\lambda_1\circ s$-almost every $z\geq x$. Since this is true for all $x>0$, we deduce that, up to a modification of $\d f/\d s$ on a $\lambda_1\circ s$-negligible set, $z\mapsto \frac{\d f}{\d s}(z)+\int_{(0,z)} (f(y)-f(z)) Q(z,\mathrm dy)=\mathcal Gf(z)\in C_0(E)$, so that $f\in \cD(\mathcal U)$ and $\mathcal G f=\mathcal U f$.
\end{proof}

We conclude this appendix by two results on the uniqueness of the martingale problem for compactly supported and/or regular functions. Here uniqueness refers to the uniqueness of the finite dimensional distributions. In particular, it entails that two c\`adl\`ag solutions to the martingale problem are indistinguishable.

\begin{proposition}
  \label{prop:martProblem}
  Take the assumptions of the previous proposition.    Let $D$ be the space of  compactly supported functions $f:E\to\R$ such that $f\rvert_{ (0,+\infty)}$ is  $s$-absolutely continuous  and such that $\d f/\d s$ is bounded, with the extension $\frac{\d f}{\d s}(\d)=0$. Then the $(\mathcal U,D)$ martingale problem is well posed, and its unique solution is the Markov process $X$.
\end{proposition}

\begin{proof}
  Note that $X$ is a solution to the $(\mathcal U,D)$ martingale problem, so that the problem admits at least one solution.

  Assume now that $Y$ is a solution to the $(\mathcal U,D)$ martingale problem. Then, for all $h\in D$ and all $x\in E$,
  \[
    h(Y_t)-h(Y_0)-\int_0^t \left[\frac{\d h}{\d s}(Y_u)+\int_{(0,Y_u)}h(y)Q(Y_u,\mathrm dy)-h(Y_u)Q(Y_u,E)\right]\,\mathrm du
  \]
  is a $\PP_x$-martingale.
  Let $f\in \mathcal D(\mathcal U)$ such that $f(1)=0$. Note that $\d f/\d s$ is bounded. For all $n\geq 2$, let $h_n$ be the $s$-absolutely continuous compactly supported function defined by
  \[
    h_n(x)=\begin{cases}
      f(\d)&\text{ if }x=\d,\\
      \int_1^x \frac{\d f}{\d s}(y)\,s(\mathrm dy)&\text{ if }x\in(\nicefrac1n,n),\\
      (f(n)-s(x)+s(n))_+&\text{ if }x\geq n\text{ and }f(n)\geq 0,\\
      -(f(n)+s(x)-s(n))_-&\text{ if }x\geq n\text{ and }f(n)\leq 0,\\
      (f(\nicefrac1n)+s(x)-s(\nicefrac1n))_+&\text{ if }x\leq\nicefrac1n\text{ and }f(\nicefrac1n)\geq 0,\\
      -(f(\nicefrac1n)+s(n)-s(x))_-&\text{ if }x\leq \nicefrac1n\text{ and }f(\nicefrac1n)\leq 0,\\
    \end{cases}
  \]
  Then $h_n$ is bounded by $\|f\|_\infty$ and $h_n(x)$ converges toward $f(x)$ for all $x\in E$. Moreover, $\d h_n/\d s$ is bounded by $\|\d f/\d s\|_\infty\vee 1$ and $\d h_n/\d s(x)$ converges toward $\d f/\d s(x)$ for all $x\in E$, with the extension $\d h_n/\d s(\d)=0$. Finally, since $Q(x,\cdot)$ is a bounded measure and $h_n$ is uniformly bounded in $n$, $Q(\cdot,h_n)-h_n(\cdot)Q(\cdot,E)$ is bounded and, by dominated convergence, $Q(x,h_n)-h_n(x)Q(x,E)$ converges toward $Q(x,f)-f(x)Q(x,E)$ for all $x\in E$. We deduce that $(h_n,\mathcal U h_n)$ converges toward $(f,\mathcal U f)$ in the bounded point-wise sense, and hence that $f(Y_t)-f(x)-\int_0^t \mathcal Uf(Y_u)\,\mathrm du$ is a martingale. If $f(1)\neq 0$, then one derives the same result by considering the function $f-f(1)$.

  Since this is true for all $f\in\cD(\mathcal U)$, we deduce that $Y$ satisfies the $(\mathcal U,\cD(\mathcal U))$ martingale problem (see for instance Proposition~4.3.1 in~\cite{EK-mp}). But $(\mathcal U,\cD(\mathcal U))$ is the infinitesimal generator of the strongly continuous semigroup $T$, and hence its martingale problem is well-posed (this is a consequence of Hille-Yosida Theorem~1.2.6 and Theorem~4.4.1 in~\cite{EK-mp}). As a consequence the finite dimensional laws of $X$ and $Y$ are the same, which concludes the proof of Proposition~\ref{prop:martProblem}.
\end{proof}

\begin{proposition}
  \label{prop:martProblem2}
  Take the assumptions of the previous proposition.    Let $D'$ be the space of  functions $f\in D$, such that $\d f/\d s$ is continuous\footnote{The proof still holds true under stronger regularity conditions on $\d f/\d s$}, with the extension $\d f/\d s(\d)=0$. Then the $(\mathcal U,D')$ martingale problem is well posed, and its unique solution is the Markov process $X$.
\end{proposition}

\begin{proof}
  Similarly to the proof of the previous proposition, we know that $X$ is a solution to the $(\mathcal U,D')$ martingale problem, and our aim is to show the uniqueness of the solution by a density argument. 

  Assume that $Y$ is a solution to the  $(\mathcal U,D')$ martingale problem and let $S$ be its semigroup, so that, for all $h\in D'$,
  \[
    S_t h(x)=h(x)+\int_0^t S_u\frac{\d h}{\d s}(x)\,\mathrm du+ \int_0^t S_u Q(\cdot,h)(x)\,\mathrm du-h(x)\int_0^t S_u Q(\cdot,E)(x)\,\mathrm du,\ \forall x\in E.
  \]

  Let $f\in D$ (where $D$ is defined in Proposition~\ref{prop:martProblem}) and denote by $[a,b]\subset(0,+\infty)$, $a<b$, a compact interval containing the support of $f\rvert_{(0,+\infty)}$. Fix $t>0$ and let $g_n$ be a bounded sequence of continuous functions with support in $[a/2,2b]\cup\{\d\}$ such that $g_n\to \d f/\d s$ in $L^1(\mu_t+\lambda_1)$, where
  \[
    \mu_t(A):= \int_0^t S_u\1_A\,\mathrm du, \text{ for all measurable set }A\subset(0,+\infty)\cup\{\d\}.
  \]
  Defining $h_n(x)=\int_a^x g_n(y)\,\mathrm dy$, we observe that $(h_n)_{n\in\N}$ is a bounded sequence in $D'$ such that $h_n(x)\to f(x)$ when $n\to+\infty$, for all $x>0$. In particular, using the fact that $A\mapsto \int_0^t S_u Q(\cdot,A)(x)\,\mathrm du$ defines a bounded measure and by dominated convergence,
  \[
    \int_0^t S_u Q(\cdot,h_n)(x)\,\mathrm du\xrightarrow[n\to+\infty]{}  \int_0^t S_u Q(\cdot,f)(x)\,\mathrm du.
  \]
  Similarly, $S_t h_n(x)\to S_t f(x)$ when $n\to+\infty$, for all $x>0$.
  Moreover, $\int_0^t S_u\frac{\d h_n}{\d s}(x)\,\mathrm du=\int_0^t S_u g_n(x)\,\mathrm du=\mu_t(g_n)$ converges to $\mu_t(\d f/\d s)=\int_0^t S_u\frac{\d h}{\d s}(x)\,\mathrm du$ when $n\to+\infty$. Finally, we proved that $S$ satisfies
  \[
    S_t f(x)=f(x)+\int_0^t S_u\frac{\d f}{\d s}(x)\,\mathrm du+ \int_0^t S_u Q(\cdot,f)(x)\,\mathrm du-f(x)\int_0^t S_u Q(\cdot,E)(x)\,\mathrm du,\ \forall f\in D.
  \]
  This implies that $Y$ satisfies the $(\mathcal U,D)$ martingale problem, and hence, according to Proposition~\ref{prop:martProblem}, that the finite dimensional laws of $Y$ and $X$ are the same. This concludes the proof of Proposition~\ref{prop:martProblem2}.
\end{proof}
\end{appendix}

\begin{acks}
  We thank two anonymous referees for their detailed and insightful comments
  which helped improve the article. Denis Villemonais acknowledges the support
  of the Institut Universitaire de France (IUF) as a Junior Member. This
  research was conducted with the support of the Inria associated team MAGO,
  which provided financial and collaborative assistance.
\end{acks}

\bibliography{GF}
\bibliographystyle{imsart-nameyear}

\end{document}